%% file: main.tex
\documentclass[12pt]{amsart}
\usepackage{preamble}

\begin{document}
	
\title{Resolutions for Locally Analytic Representations}

\author{Shishir Agrawal}
\address{University of California San Diego, Department of Mathematics, 9500 Gilman Drive, La Jolla, CA 92093}
\email{s2agrawal@ucsd.edu}
\author{Matthias Strauch}
\address{Indiana University, Department of Mathematics, Rawles Hall, Bloomington, IN 47405, U.S.A.}
\email{mstrauch@indiana.edu}
	

\begin{abstract}
 The purpose of this paper is to study resolutions of locally analytic representations of a $p$-adic reductive group $G$. Given a locally analytic representation $V$ of $G$, we modify the Schneider-Stuhler complex (originally defined for smooth representations) so as to give an `analytic' variant $\cS^A_\bullet(V)$. The representations in this complex are built out of spaces of analytic vectors $A_\sigma(V)$ for compact open subgroups $\Us$, indexed by facets $\sigma$ of the Bruhat-Tits building of $G$. These analytic representations (of compact open subgroups of $G$) are then resolved using the Chevalley-Eilenberg complex from the theory of Lie algebras. This gives rise to a resolution $\cS^\CE_{q,\bullet}(V) \ra \cS^A_q(V)$ for each representation $\cS^A_q(V)$ in the analytic Schneider-Stuhler complex. In a last step we show that the family of representations $\cS^\CE_{q,j}(V)$ can be given the structure of a Wall complex. The associated total complex $\cS^\CE_\bullet(V)$ has then the same homology as that of $\cS^A_\bullet(V)$. If the latter is a resolution of $V$, then one can use $\cS^\CE_\bullet(V)$ to find a complex which computes the extension group $\uExt^n_G(V,W)$, provided $V$ and $W$ satisfy certain conditions which are satisfied when both are admissible locally analytic representations. 
\end{abstract}

\maketitle
\tableofcontents

\input{0-introduction}
\input{1-analytic-groups}
\input{2-distributions}
\input{3-representations}

\input{4-schneiderstuhler}

\input{5-mixing}

\bibliographystyle{alphaurl}
\bibliography{mybib}

\end{document}

%% file: 0-introduction.tex
\section*{Introduction}
The purpose of this paper is to investigate resolutions of locally analytic representations of a $p$-adic reductive group $G = \bG(\Qp)$\footnote{In order to simplify the discussion in the introduction we consider here only the case of $\Qp$-analytic groups, whereas in the body of the paper we consider locally $F$-analytic groups, where $F$ is any finite extension of $\Qp$.} with a view towards Ext groups of representations. 

\vskip8pt

{\it The theory for smooth representations.} Our starting point is the work of P. Schneider and U. Stuhler \cite{ScSt97} on resolutions of smooth complex representations.\footnote{Their work was in turn inspired by earlier constructions of Moy-Prasad \cite{MoyPrasad} and Prasad-Raghunatan \cite{PrasadRaghunathan}.} For any integer $e \ge 0$, called a {\it level}, they define compact open subgroups $U_\sigma = U_\sigma^{(e)}$ indexed by facets $\sigma$ of the semisimple Bruhat-Tits building $\BT = \BT(\bG/\Qp)$. Each $\Us$ is a normal subgroup of the stabilizer $\Psd$ of the facet $\sigma$. Given a (smooth) representation $V$, the family $(V^\Us)_\sigma$ of invariant subspaces forms a so-called coefficient system on $\BT$. One can then form the homological complex $\cS_\bullet(V)$ of oriented cochains which has in degree $q$ the $G$-representation
\[\cS_q(V) = \bigoplus_{\sigma \in G\bksl \BT_q} \cind^G_\Psd (V^\Us \otimes \chi_\sigma)\]
where $\BT_q$ is the set of $q$-dimensional facets of $\BT$, and $\chi_\sigma: \Psd \ra \{1,-1\}$ is a character which measures whether or not an element of $\Psd$ preserves the orientation of $\sigma$ (chosen once and for all for every facet). This complex has a natural augmentation $\cS_\bullet(V) \ra V$. A key result about this complex is that if $V$ is generated by the $U_{\sigma_0}$-invariants for some vertex $\sigma_0$, then $\cS_\bullet(V) \ra V$ is a resolution of $V$, and if $V$ has a central character $\chi$, then the representations $\cS_q(V)$ are projective objects in the category of smooth representations of $G$ with central character $\chi$ \cite[II.2.2]{ScSt97}. The theory of \cite{ScSt97} has more aspects (e.g., a sheaf theoretic or cohomological variant), the discussion of which we postpone to a forthcoming paper, and it has had many applications, notably to the existence of the Zelevinsky involution (also called Aubert-Schneider-Stuhler involution in \cite{Moeglin_Paquet-d-Arthur}), to the theory of pseudo-coefficients \cite{Waldspurger_GLN-tordu-I}, to Euler-Poincar\'e pairings \cite{Reeder_EP-pairings}, and to the study of smooth representations in characteristic $p$ \cite{Paskunas_Coeff-systems,OllivierSchneiderGorenstein,Kohlhaase_Coeff-systems}, to name only a few references. 

\vskip8pt

{\it The case of locally analytic representations: analytic vectors. } If a locally analytic representation of $G$ is generated by smooth vectors, it is itself a smooth representation. Therefore, the subspace of smooth vectors of a locally analytic representation vanishes in general, and the complex $\cS_\bullet(V)$, as defined above, is not a resolution of $V$, for any level $e$. However, there is an obvious and natural modification at hand, namely to use {\it analytic} vectors for the groups $\Us$ (instead of invariant vectors). To make sens of the concept of taking analytic vectors, one needs to `promote' $\Us$ to an affinoid rigid analytic group $\bbUs$, which is possible if the level $e$ is large enough, as these groups are then uniform pro-$p$ groups.\footnote{A somewhat weaker condition, namely that of being a saturated $p$-valued group would also suffice, but we work here only with uniform pro-$p$ groups.} With $\bbU_\sigma$ at hand one can define the space $V^{\bbUs\anv}$ of analytic vectors for this group. In fact, there are many variants of this construction: instead of working with analytic vectors for affinoid groups $\bbUs$ one can also consider analytic vectors for `wide-open' groups $\bbUso$, or for overconvergent groups $\bbUsd$. Moreover, one can also consider the space $\uHom_\Us(D_r(\Us),V)$ of vectors which are `analytic' for distribution algebras $D_r(\Us)$ as they have been introduced in \cite{ST03}. And there is yet another version of analytic vectors which we denote by $V^D$, where $D$ is a suitable algebra of analytic distributions on $\bbUs$ or a related group. Depending on $V$, some of these spaces of analytic vectors may be the same. The reason to consider these various constructions of analytic vectors is that the corresponding spaces are of different topological type. When $V$ is an admissible locally analytic representation, then $V^{\bbU\anv}$ is a Banach space, $V^{\bbUso\anv}$ is a Fr\'echet space of compact type, $V^{\bbUsd\anv}$ is of compact type (i.e., a countable inductive limit of Banach spaces with compact and injective transition maps), and $\Hom_\Us(D_r(\Us),V)$ is a Smith space (cf. \cite[3.1]{RR} for the concept of Smith space). Depending on the context, some choice of analytic vectors can be more advantageous, and we have therefore developed part of the theory in this generality. Henceforth we will denote by $(A_\sigma(V))_\sigma$ one fixed choice of coefficient system of this type. The space $A_\sigma(V)$ carries a natural action by $\Psd$. Here and throughout the paper we let $\Psd$ act on $A_\sigma(V)$ via its natural action, but twisted by the character $\chi_\sigma$, which we will drop from the notation from now on. Then we define
\[\cS^A_q(V) = \bigoplus_{\sigma \in G\bksl \BT_q} \cind^G_\Psd A_\sigma(V)\]
and obtain a complex of $G$-representations, where the differentials are defined in the same way as in the case of smooth representations.

\vskip8pt

{\it Topological considerations: a categorical framework.} If one then tries to generalize the theory of \cite{ScSt97} to the analytic Schneider-Stuhler complex, one realizes that one needs a suitable categorical framework for doing homological algebra. In his paper \cite{Kohlhaase_Cohomology} Jan Kohlhaase studies the cohomology of locally analytic representations by equipping the category of locally convex $E$-vector spaces\footnote{Here and in the following $E$ denotes a finite extension of $\Qp$ which serves as our coefficient field throughout this paper. In particular, all representations are tacitly assumed to be on $E$-vector spaces.} with the structure of an exact category. To this end he introduces the concept of a strongly exact sequence: a sequence $V_0 \xrightarrow{\delta} V_1 \xrightarrow{\delta_1} V_1$ of locally convex vector spaces is called strongly exact, if it is exact, the maps are strict with closed image, and kernel and image of these maps admit complements by closed subspaces. This leads to `relative' extension groups ${\mathscr Ext}^n_G(V,W)$ and there are comparison maps ${\mathscr Ext}^n_G(V,W) \ra \Ext_{D(G)}(V,W)$, where $D(G)$ is the locally analytic distribution algebra of $G$. However, these comparison maps are not always isomorphisms. For example, the sequence of $\Zp$-representations
\begin{numequation}\tag{$\ast$}\label{non-split-example}
\begin{tikzcd} 0 \ar[r] & C^\infty(\Zp,E) \ar[r,hook] & C^\la(\Zp,E) \ar[r,"\frac{d}{dx}"] & C^\la(\Zp,E) \ar[r] & 0 
\end{tikzcd}
\end{numequation}
does not split on the level of topological vector spaces \cite[4.3]{Kohlhaase_Cohomology}, and it gives therefore a class in $\Ext^1_{D(\Zp)}(C^\la(\Zp,E), C^\infty(\Zp))$ which is not in the image of the comparison map. However, it is shown in \cite[4.4, 6.5-6.7]{Kohlhaase_Cohomology} that the comparison maps are isomorphisms in many cases of interest when $V = E_\triv$ is the trivial $G$-representation. Yet, many interesting representations in the context of the $p$-adic Langlands correspondence for $\GL_2(\Qp)$ are extensions of locally analytic representations by smooth representations, similar to (\ref{non-split-example}). For this reason and because we continued to encounter technical problems with the homological algebra in the context of topological vector spaces, we decided to work in a different setting, and it was just around the time when we considered using ideas from Condensed Mathematics that the paper \cite{RR} by Rodrigues Jacinto and Rodr\'iguez Camargo appeared. Therefore, in this paper we will be working in the framework of solid (locally analytic) representations. While we do not discuss the details here in the introduction, our paper uses throughout very many ideas and results from \cite{RR}.

\vskip8pt

{\it The analytic Schneider-Stuhler complex: questions regarding its exactness.} Recall that we fixed a coefficient system of analytic vectors $(A_\sigma(-))_\sigma$. We denote by $V$ a solid locally analytic representation. Assume that $V$ is (algebraically) generated by $A_{\sigma_0}(V)$ for some vertex $\sigma_0$, by which we mean that the augmentation map $\cS_0^A(V) \ra V$ is an epimorphism. Then the following questions are obviously of crucial importance:
\vskip8pt
\begin{enumerate}
\item Is $\cS^A_\bullet(V) \ra V$ a resolution of $V$? Equivalently, is $\cS^A_\bullet(V)$ exact in positive degrees and $h_0(\cS^A_\bullet(V)) = V$?
\item Suppose $V$ has a central character. Are the representations $\cS^A_n(V)$ projective objects in the category of solid locally analytic representations with that same central character?
\end{enumerate}
\vskip8pt
With regard to the first question we have some positive results for locally $F$-analytic principal series representations for $\GL_2(F)$, $F$ a finite extension of $\Qp$, when $A_\sigma(V) = V^{\bbUs\anv}$, by \cite{LahiriResolutions}.\footnote{The proofs in \cite{LahiriResolutions} should also be valid when one works with analytic vectors for wide-open or overconvergent groups, or with $D_r(\Us)$-analytic vectors.} Moreover, the work of Dospinescu-Le Bras \cite{DospLeBras} gives some indication that the answer to question (1) might be `yes' for certain spaces of global sections of equivariant vector bundles on the Drinfeld upper half plane, cf. \ref{remarks-exactness-of-analytic-ScSt}. Furthermore, we show that the answer to (1) is positive for a `universal' representation $V^\univ = \cind^G_U D_r(U)$, where $U$ is a fixed compact open subgroup of $G$ and the coefficient system is given by $A_\sigma(V) = \uHom_\Us(D_r(\Us),V)$,\footnote{$\uHom$ and $\uExt$ denote internal Hom and Ext groups in the category of solid $E$-vector spaces.} cf. \cref{universal-rep}. The result that $\cS^A(V^\univ) \ra V^\univ$ is a resolution is analogous to the case of smooth representations (yet the proof is somewhat different). Moreover, in the setting of smooth representations (over a field of characteristic zero), question (1) for an arbitrary smooth representation is reduced to the case of the universal representation \cite[proof of II.3.1]{ScSt97}. In \cite{ScSt97}, the reduction to the case of a universal representation uses two key facts: 
\vskip8pt
\begin{enumerate}[(i)]
\item The exactness of the functor $V \rightsquigarrow V^\Us$ on the category of smooth representations.
\item The Bernstein-Borel-Matsumoto theorem which says that, if $U \sub G$ is a compact open subgroup which possesses an Iwahori decomposition, then the category of smooth representations $V$ which are generated by $V^U$  is stable under passage to subrepresentations.
\end{enumerate}
\vskip8pt
In analogy to (i) one may ask if any of our functors of taking analytic vectors is exact. However, it is not too difficult to see this is not the case, and the paper \cite{RR} also discusses the derived functor of the functor of taking analytic vectors. Therefore, we cannot expect to reduce the question of the exactness of the analytic Schneider-Stuhler complex for a general solid locally analytic representation to the case of a universal representation. With regard to the Bernstein-Borel-Matsumoto theorem (ii), it makes sense to ask if a similar statement holds for (solid) locally analytic representations, namely if the subcategory of those representations $V$ which are (algebraically) generated by $V^{\bbU\anv}$ (or some other space of analytic vectors) is stable under passage to subobjects. We do not know the answer to this question.\footnote{We also do not know the answer to a much more narrow question which is interesting in its own right, namely if any admissible locally analytic representations $V$ which is generated by $V^{\bbU\anv}$ has the property that any closed subrepresentation $W \sub V$ is also generated by $W^{\bbU\anv}$.} Hence, while $\cS^A(V^\univ) \ra V^\univ$ is a resolution of $V^\univ$ (if the level $e$ is large enough with respect to $U$), this result apparently cannot be used in a straightforward manner to treat the case of general locally analytic representations.

\vskip8pt

{\it The representations of the analytic Schneider-Stuhler complex are not projective objects.} While we do not know the answer to the question (1) above, we do know that the answer to question (2) is 'no' (we tacitly assume that $\dim(\bG) \ge 1$). For example, if $V$ is a smooth representation, then $V^{\bbUs\anv} = V^\Us$, and the analytic Schneider-Stuhler complex is equal to the Schneider-Stuhler complex formed with the functor of invariant vectors. In that case, every representation $\cS^A_n(V)$ is an infinite-dimensional smooth representation, but these are not projective in the category of (solid) locally analytic representations (with a fixed central character), as one can see by considering explicit situations, cf. \cref{counterexample-projective}. Another example is given by taking for $V$ the trivial one-dimensional representation $E_\triv$. If $\cS^A(E_\triv)$ would consist of projective objects, then the projective dimension of $E_\triv$ would be at most the semisimple rank $\rk(\bG/\bZ_\bG)$ of $\bG$, and we would find that $\uH^n(G,W) := \uExt^n_G(E_\triv,W)$ would vanish for $n > \rk(\bG/\bZ_\bG)$. This contradicts earlier results of Kohlhaase which show a clear connection of locally analytic cohomology with Lie algebra cohomology, e.g.. \cite[4.10]{Kohlhaase_Cohomology}. Indeed, as is already apparent from \cite[proof of 6.6]{Kohlhaase_Cohomology}, one should also resolve the analytic representations $A_\sigma(V)$ as modules over a suitable distribution algebra of an analytic group. 

\vskip8pt

{\it Resolutions of spaces of analytic vectors.} In the following we will only work with the particular coefficient system given by $A_\sigma(V) = \uHom_\Us(D_r(\Us),V)$. The reason for doing so is the following: when $V$ is an admissible locally analytic representation, the spaces $A_\sigma(V)$ are Smith spaces, as we remarked earlier. This will be used later on in a crucial way. 

At this point we need to introduce yet another family $(\Hs)_\sigma$ of open compact subgroups. Namely, we set $\Hs = \bigcap_{\tau \in \BT_0, \tau \sub \ovsigma} \Us$. This group too is a uniform pro-$p$ compact open normal subgroup of $\Psd$, and it is in particular contained in $\Us$. We have thus available the rigid analytic groups $\bbHso$.
Moreover, for $r \in (\frac{1}{p},p^{-\frac{1}{p-1}})$\footnote{When $p=2$ one needs to replace $p^{-\frac{1}{p-1}}$ by $p^{-\frac{1}{2(p-1)}}$.} there is a canonical morphism of solid $E$-algebras $D(\bbHso) \ra D_r(\Us)$, and $A_\sigma(V)$ becomes a solid module over 
\[D(\bbHso,\Psd) = D(\bbHso) \otimes_{E[\Hso]} E[\Psd] \,,\]  
where $\Hso = \bbHso(\Qp)$. The universal enveloping algebra $U(\frg)$ of $\frg = \Lie(\bG)$ is dense in $D(\bbHso)$, and one has the resolution\footnote{By convention, all algebras are tacitly based changed to $E$, e.g., we write $U(\frg)$ instead of $U(\frg) \ot_\Qp E$.}
\begin{numequation}\tag{$\CE_\bullet(\frg)$}
\begin{tikzcd}
0 \ar[r] & U(\frg) \ot \bigwedge^d \frg \ar[r] & \cdots \ar[r] & U(\frg) \ot \frg \ar[r] & U(\frg) \ar[r] & E_\triv \ar[r] & 0 
\end{tikzcd}
\end{numequation}
of the trivial one-dimensional module $E_\triv$, where $d = \dim(\frg)$. This resolution goes by the name of Chevalley-Eilenberg resolution \cite[7.7]{WeibelH} or Koszul/standard resolution \cite[ch. II, 7.7]{KnappVogan}. In this paper we will adopt the former terminology.\footnote{We note that Koszul in his thesis \cite{Koszul_HomologieLie} writes ``Mon travail reprend certaines parties de leur M\'emoire'', by which he refers to the paper \cite{ChevalleyEilenberg} of Chevalley and Eilenberg.} One then takes the tensor product with $D(\bbHso)$ over $U(\frg)$ of this resolution and obtains a resolution 
\begin{numequation}\tag{$\CE_\bullet(\bbHso)$}
0 \lra D(\bbHso) \ot \bigwedge^d \frg \lra \cdots \lra D(\bbHso) \ot \frg \lra D(\bbHso) \lra  E_\triv \lra 0 
\end{numequation}  
because $D(\bbHso) \ot_{U(\frg)} E_\triv = E_\triv$. Using the theory of the Wall complex \cite[V.3.1]{LazardGroupesAnalytiques} one can furthermore extend the resolution $\CE_\bullet(\bbHso)$ to a resolution of $E_\triv$ as a module over the algebra $\DHPsd$: 
\begin{numequation}\tag{$\CE_\bullet(\bbHso,\Psd)$}
0 \ra \DHPsd \ot \bigwedge^d \frg \ra \cdots \ra \DHPsd \ot \frg \ra \DHPsd \ra E_\triv \ra 0 \,.
\end{numequation}  
It is a resolution by finitely generated free $\DHPsd$-modules. It is no longer of finite length, but one can truncate it and obtain a resolution $\tau_{\le d} \CE_\bullet(\bbHso,\Psd)$ of $E_\triv$ of length $\dim(\bG)$ by finitely generated projective $\DHPsd$-modules. Finally, taking the solid tensor product of $\CE_\bullet(\bbHso,\Psd)$ with $A_\sigma(V)$ gives a resolution
\[\begin{tikzcd}
\cdots \ar[r] & \DHPsd \ot \frg \sotimes A_\sigma(V) \ar[r] & \DHPsd \sotimes A_\sigma(V) \ar[r] & A_\sigma(V) \ar[r] & 0 
\end{tikzcd}\]  
Because $A_\sigma(V)$ is a Smith space, the $\DHPsd$-modules in this latter resolution are projective objects in the category of solid $\DHPsd$-modules. Again, this complex can be truncated to give a resolution $\tau_{\le \bd} \CE_\bullet(\bbHso,\Psd) \sotimes A_\sigma(V)$ of $A_\sigma(V)$ by projective solid $\DHPsd$-modules. Here $\bd = \dim(\bG) + \rk(\bZ_\bG)$ is the sum of $d = \dim(\bG)$ and the $\Qp$-rank of the center of $\bG$. In particular, if $\bG$ is semisimple then $\bd = d$.  

\vskip8pt

{\it Building a Wall complex over $\cS^A_\bullet(V)$.} At this point we set 
\[\cS^\CE_{q,j}(V) = \cS^{A,\CE}_{q,j}(V) = \bigoplus_{\sigma \in G\bksl \BT_q} \cind^G_\Psd\Big(\CE_j(\bbHso,\Psd) \sotimes A_\sigma(V)\Big)\,.\]
For fixed $q$ the complex $\cS^\CE_{q,\bullet}(V)$ is called {\it complexe fibre} in \cite[V.3.1.2, p. 187]{LazardGroupesAnalytiques}, and it is augmented by a natural surjective $G$-homomorphism $\cS^\CE_{q,0}(V) \ra \cS^A_q(V)$. However, one does not obtain a double complex. Instead, one can show that there is a family of morphisms of solid $G$-representations
\[d^{(k)}_{q,j}: \cS^\CE_{q,j}(V) \lra \cS^\CE_{q-k,j+k-1}(V)\,,\; 0 \le k \le q\,,\]
which gives $\cS^\CE_{\bullet,\bullet}(V)$ the structure of a Wall complex, cf. \cref{thm-mixed-res}. To show the existence of these maps one uses Frobenius reciprocity and the previously established fact that each $\CE_j(\bbHso,\Psd) \sotimes A_\sigma(V)$ is a projective solid $\DHPsd$-module. Is it here where the reason for introducing the groups $\Hs$ becomes apparent. Namely, in order to show the existence of the maps $d^{(k)}_{q,j}$ we need to have $H_\sigma \sub H_\tau$ whenever $\tau \sub \ovsigma$. The groups $\Us$ behave not like that: one has $\Us \supset \Ut$ when $\tau \sub \ovsigma$.

\vskip8pt

{\it The total complex $\cS^\CE_\bullet(V)$.} The existence of the Wall complex for $\cS^\CE_{\bullet,\bullet}(V)$ means that there is a total complex $(\cS^\CE_n(V) = \bigoplus_{q+j =n} \cS^\CE_{q,j}(V), \Delta_n)_n$ which comes with an augmentation $\cS^\CE_0(V) \ra V$, and whose homology  is naturally isomorphic to that of $\cS^A_\bullet(V)$. And if $\cS^A_\bullet(V) \ra V$ is a resolution of $V$, then so it $\cS^\CE_\bullet(V) \ra V$. One can also work with the truncated complexes $\tau_{\le \bd} \CE(\bbHso,\Psd) \sotimes A_\sigma(V)$ and obtain $G$-representations $\tau \cS^\CE_{q,j}(V)$, which can be equipped with the structure of a Wall complex, and one has an augmented total complex $\tau \cS^\CE_\bullet(V) \ra V$ which is a resolution of $V$ if $\cS^A_\bullet(V) \ra V$ is. The representations $\tau\cS^\CE_n(V)$ vanish for $n > \dim(\bG) + \rk(\bG)$. 

\vskip8pt

{\it Applications to Ext groups.} Let $V$ be an admissible locally analytic representations of $G$. Assume furthermore that the analytic Schneider-Stuhler complex $\cS^A_\bullet(V) \ra V$ is a resolution of $V$, so that both $\cS^\CE_\bullet(V) \ra V$ and $\tau\cS^\CE_\bullet(V) \ra V$ are resolutions of $V$ too. Let $W$ be a solid representation of $G$. When one then makes $R\uHom_G(\cS^\CE_n(V),W)$ explicit, one encounters a term which involves the space of derived analytic vectors $W^{R\bbHso\anv}$. If we assume that $W$ is admissible, then $W^{R\bbHso\anv} = W^{\bbHso\anv}$, i.e., the higher derived analytic vectors of $W$ vanish, cf. \cref{higher-derived-an-vectors}. This allows us to show that $R\uHom_G(\cS^\CE_n(V),W)$ has cohomology only in degree zero, and we obtain the following result: 
\vskip8pt

{\bf Theorem {\rm (Theorem \ref{Ext-for-admissible})}.} {\it Let $V$ and $W$ be both admissible locally $F$-analytic representations of $G = \bG(F)$. Assume that the augmented analytic Schneider-Stuhler complex $\cS_\bullet(V) = \cS^A_\bullet(V) \ra V$ is a resolution of $V$. Write 
\[\CE_j(\bbHso,\Psd) =  \DHPsd \sotimes_E M_{\sigma,j}\]
with an explicit finite-dimensional $E$-vector space $M_{\sigma,j}$, cf. paragraphs \ref{wall-complex} and \ref{ind-affinoid-mixed-generalized}. 

\begin{enumerate}
\item $\uExt^n_G(V,W)$ can be identified with the $n^{\rm th}$ cohomology group of a complex of solid $E$-vector spaces $(\cE^\bullet(V,W), \partial^\bullet)$, where 
$$\cE^n(V,W) = \bigoplus_{\scalebox{.7}{$\begin{array}{c} 0 \le q \le \ell,\, 0 \le j \le \bd \\
q+j = n \end{array}$}} \bigoplus_{\sigma \in \cR(\BT_q)} \uHom_E\Big(M_{\sigma,j} \sotimes_E V^{D_r(U_\sigma)},W^{\bbHso\anv}\Big) \;.$$
\item The cohomology of this complex vanishes in degrees $n > \dim(\bG) + \rk_F(\bG)$. 
\end{enumerate}}
\vskip8pt

In (1) the set $\cR(\BT_q)$ is a system of representatives for $G\bksl \BT_q$. The proof of (2) makes use of the resolution $\tau \cS^\CE_\bullet(V) \ra V$. From this theorem one obtains easily the 
\vskip8pt

{\bf Corollary {\rm (Corollary \ref{cor-Ext-for-admissible})}.}  {\it Let the assumptions be as in the previous theorem. 
\begin{enumerate}
\item For all $n > \dim(\bG) + \rk_F(\bG)$ one has $\uExt^n_G(V,W) = 0$.
\item If for all facets $\sigma$ of $\BT$\footnote{Equivalently, for every $q = 0, \ldots, \ell$ and every representative for $G\bksl BT_q$.} one has $W^{\bbHso\anv} = 0$, then $\uExt^n_G(V,W) = 0$ for all $n \ge 0$.
\item If for all facets $\sigma$ of $\BT$ one has $W^{\bbHso\anv} = 0$, then for all $n \ge 0$ the cohomology group $\underline{H}^n(G,W) = \uExt^n_G(E_\triv,W)$  vanishes.
\end{enumerate}}
\vskip8pt
As mentioned above, there are many aspects of the theory we have not yet explored. One obvious task which we plan to look into next is to make explicit the connection with Lie algebra cohomology. As the complex $\CE_\bullet(\bbHso)$ is closely related to Lie algebra cohomology, it can be expected that the Ext groups $\uExt_G(V,W)$ have a description which involves extension groups of $\frg$-modules.

\subsection*{Acknowledgments} 
As is abundantly clear from what has been said above, our paper has its origin in the seminal work \cite{ScSt97} by P. Schneider and U Stuhler. The paper \cite{Kohlhaase_Cohomology} by J. Kohlhaase also had a significant impact on our work, especially with regard to the role of the Wall complex. And not only has the work of J. Rodrigues Jacinto and J. E. Rodr\'iguez Camargo \cite{RR} provided us with the right framework, we also use many of their results throughout our paper. 

We are grateful to the following individuals for responding to our questions which helped us to improve the paper in a couple of places: Jessica Fintzen for references regarding Moy-Prasad filtrations; Tobias Schmidt for discussions on $F$-analytic distribution algebras; Juan Esteban Rodr\'iguez Camargo for useful remarks regarding derived analytic vectors. 

The second-named author has given talks about this work at conferences in M\"unster and Carbondale, and he thanks the organizers for giving him the opportunity to present these results. Parts of this paper have been written during a stay at the Hausdorff institute in Bonn during the Trimester on {\it The Arithmetic of the Langlands Program} in 2023, and M.S. would like to acknowledge the perfect working conditions provided by that institution.

\subsection*{Notation}

\begin{para}
All modules are left modules unless otherwise specified. If $R$ is a ring (or a condensed ring, solid ring, etc.) we write $\Mod_R$ for the category of $R$-modules.  
\end{para}

\begin{para}
We denote by $F$ a finite extension of $\Q_p$, by $\cO_F$ its ring of integers, and by $\vpi$ a uniformizer. We use $F$ as our ground field for all ``geometric'' objects (e.g., finite type schemes, analytic spaces, locally analytic manifolds, etc). We use bold letters (e.g., $\bX$) for finite type schemes and blackboard bold letters (e.g., $\bbX$) for analytic spaces, and regular letters (e.g., $X$) for locally analytic manifolds. If $\bbX$ is an analytic space over $F$, we use the corresponding regular letter $X$ to denote the locally analytic manifold $\bbX(F)$.
\end{para}

\begin{para} \label{locally-F-analytic}
By ``locally analytic'' we always mean ``locally $F$-analytic'' unless explicitly stated otherwise. Given a locally $F$-analytic group $G$, we denote by $\Res^F_\Qp G$ the group $G$ considered as a locally $\Qp$-analytic group.      
\end{para}

\begin{para} \label{tacit-base-extension}
Let $E$ be a finite extension of $F$. We use $E$ as our field of coefficients for all ``linear'' objects (vector spaces, associative algebras, Lie algebras, representations, etc). Any such object that is initially defined over $F$ is tacitly base changed up to $E$. For example, if $\bbX$ is an analytic space, we write $\sO(\bbX)$ in place of $\sO(\bbX) \otimes_F E$. Similarly, if $G$ is a locally analytic group over $F$, we use the corresponding lower case fraktur letter $\frg$ to denote $\Lie(G) \otimes_F E$.
\end{para}

\begin{para}
We assume that the $p$-adic norm on $F$ and $E$ has the standard normalization, i.e., the one for which $|p| = p^{-1}$. 
\end{para}

\begin{para}\label{strong-limit-cardinal}
We fix, once and for all, an uncountable strong limit cardinal $\kappa$. Unless explicitly stated otherwise, profinite sets are assumed to be $\kappa$-small, and condensed sets are assumed to be $\kappa$-condensed, so that a condensed sets is in fact a sheaf on the site of profinite sets (with coverings being finite jointly surjective families). We write $\Hom(X, Y)$ and $\uHom(X, Y)$ for the external and internal homs of condensed sets: the former is a set, the latter is a condensed set.
\end{para}

\begin{para} \label{condensed}
Write $\Vec_E^{\rm cond}$ for the category of condensed $E$-vector spaces. It is a bicomplete Grothen- dieck abelian category in which products, coproducts, and filtered colimits are exact. It is generated by compact projectives of the form $E[T]$, where $T$ is a ($\kappa$-small) extremally disconnected set.\footnote{$E[T]$ is the sheafification of of the presheaf that assigns to any extremally disconnected $S$ the free vector space $E[T(S)]$ on the set $T(S)$ of continuous maps $S \to T$. Equivalently, $E[T]$ is the sheaf that represents the ``sections over $T$'' functor $V \rightsquigarrow V(T)$ on $\Vec_E^\cond$.}
It is closed symmetric monoidal category with tensor product $- \otimes_E -$ and internal hom $\uHom_E(-,-)$ over $E$, and we write $(-)^\vee := \uHom_E(-, E)$ for the internal dual. We distinguish the internal hom $\uHom_E(-,-)$ from the external hom $\Hom_E(-,-)$ as the former is a condensed vector space while the latter is an abstract vector space.
\end{para}

\begin{para} \label{solid}
Let $\Vec_E^\solid$ denote the category of solid vector spaces over $E$ (cf. \cite[sec. A.3]{BoscoDrinfeldv2}, \cite[sec. 3]{RR}). It is a reflective bicomplete full abelian subcategory of $\Vec_E^{\rm cond}$ that is stable under extensions. The left adjoint to the inclusion is {\it solidification} $V \mapsto V^\solid$, and the solid vector spaces $E[T]^\solid$, for $T$ a extremally disconnected set, form a family of compact projective generators. It becomes a closed symmetric monoidal category when equipped with the solid tensor product $- \sotimes -$ over $E$ and the internal $\uHom_E(-,-)$ of condensed vector spaces (cf. \cite[A.17]{BoscoDrinfeldv2}).
\end{para}

\begin{para}
There is a natural lax symmetric monoidal functor $V \rightsquigarrow \underline{V}$ into $\Vec_E^\solid$ from the category of complete locally convex vector spaces over $E$ with the completed projective tensor product $- \hat{\otimes}_\pi -$. For any complete locally convex vector space $V$, the corresponding solid vector space $\underline{V}$ is quasi-separated \cite[proposition A.31]{BoscoDrinfeldv2} and hence flat \cite[proposition A.28]{BoscoDrinfeldv2}. It is fully faithful on the full subcategory of $\kappa$-compactly generated $V$ \cite[proposition 1.7]{ScholzeCondensed}. Moreover, when restricted further to the full subcategory of Fr\'echet spaces, the functor $V \rightsquigarrow \underline{V}$ exact \cite[lemma A.33]{BoscoDrinfeldv2} and strong symmetric monoidal \cite[proposition A.68]{BoscoDrinfeldv2}. When $V$ is Fr\'echet and no ambiguity can arise, we will write simply $V$ again in place of $\underline{V}$.
\end{para}

%% file: 1-analytic-groups.tex
\section{Analytic groups}

\begin{para}
Let $\bbG$ be an affinoid group over $F$ and regard $G = \bbG(F)$ with its natural topology as a condensed group (i.e., a group object in the category of condensed sets, or equivalently, a sheaf of groups on the site of profinite sets). Recall that the space $\sO(\bbG)$ of analytic functions on $\bbG$ over $E$ (cf. \cref{tacit-base-extension}) carries left and right regular actions of $G$ which commute with each other, defining an action of $G \times G$ on $\sO(\bbG)$ \cite[sections 3.1, 3.3]{EmertonA}. 
\end{para}

\subsection{Strictly ind-affinoid groups}

\begin{definition} \label{ind-affinoid-space}
A {\it strictly ind-affinoid space} $\bbX^\circ$ is an ind-object of the category of affinoid spaces which has a {\it presentation} by an increasing chain
\[ \bbX^{(0)} \Subset \bbX^{(1)} \Subset \cdots \]
of relatively compact open embeddings of affinoid spaces.\footnote{See \cite[definition 2.1.5]{EmertonA}, \cite[section 9.6.2]{BGR}, or \cite[definition 0.4.2]{Hu96} for a definition of relative compactness. These ind-objects can be identified with the {\it strictly $\sigma$-affinoid spaces} of \cite[definition 2.1.17]{EmertonA}, but we use the ind-object formalism to make the ``duality'' with the strictly pro-affinoid case below more transparent.}
\end{definition}

\begin{para}
The category of strictly ind-affinoid spaces is a full subcategory of the ind-completion of affinoid spaces. Dualize \cite[tag \href{https://stacks.math.columbia.edu/tag/0G2W}{0G2W}]{stacks-project} for a description of morphisms in this category. This category has all finite limits.
\end{para}

\begin{para}
For $\bbX^\circ$ a strictly ind-affinoid analytic space (cf. \cref{ind-affinoid-space}), let 
\[ \sO(\bbX^\circ) = \varprojlim_n \sO(\bbX^{(n)}), \]
where $\bbX^{(\bullet)}$ is a presentation of $\bbX^\circ$. This is independent of choice of presentation. It is a Fr\'echet space of compact type \cite[proposition 2.1.16]{EmertonA}, \cite[definition 3.34]{RR}. 
\end{para}

\begin{definition}\label{dfn-strictly-ind-affinoid-group}
A {\it strictly ind-affinoid group} $\bbG^\circ$ is a group object in the category of strictly ind-affinoid spaces which has a presentation $\bbG^{(\bullet)}$ where each $\bbG^{(n)}$ is an affinoid group, each relatively compact open embedding $\bbG^{(n)} \Subset \bbG^{(n+1)}$ is a homomorphism of affinoid groups, and the group object structure of $\bbG^\circ$ is compatible with the group object structure on each $\bbG^{(n)}$. Hereafter, a {\it presentation} of a strictly ind-affinoid group will always refer to a presentation of this form. 
\end{definition}


\begin{para}
To a strictly ind-affinoid group $\bbG^\circ$, we associate the condensed group $G^\circ = \varinjlim_n \bbG^{(n)}(F)$, where the $\bbG^{(n)}$ form a presentation of $\bbG^\circ$ by affinoid groups and each $\bbG^{(n)}(F)$ is regarded as a condensed\footnotemark{} group using its natural topology. This is independent of choice of presentation.
\footnotetext{The ``underline'' functor of \cite[example 1.5]{ScholzeCondensed} is easily seen to commute with filtered colimits using the fact that profinite sets are compact. Thus this colimit in the category of condensed spaces agrees with the underline of the colimit in the category of topological spaces.}
\end{para}

\subsection{Strictly pro-affinoid groups}

\begin{definition} \label{pro-affinoid-space}
A {\it strictly pro-affinoid space} $\bbX^\dagger$ is a pro-object of the category of affinoid spaces which has a {\it presentation} by a decreasing chain
\[ \bbX_{(0)} \Supset \bbX_{(1)} \Supset \cdots \]
of relatively compact open embeddings of affinoid spaces. 
\end{definition}


\begin{para}
The category of strictly pro-affinoid spaces is a full subcategory of the pro-completion of the category of affinoid spaces. See \cite[tag \href{https://stacks.math.columbia.edu/tag/0G2W}{0G2W}]{stacks-project} for a description of morphisms in this category. This category has all finite limits.
\end{para}

\begin{para}
For $\bbX^\dagger$ a strictly pro-affinoid space, let 
\[ \sO(\bbX^\dagger) = \varinjlim_n \sO(\bbX_{(n)}), \]
where $\bbX_{(\bullet)}$ is a presentation of $\bbX^\circ$. This is independent of choice of presentation. It is an inductive limit of the Banach spaces $\sO(\bbX_{(n)})$ with compact transition maps. By replacing $\sO(\bbX_{(n)})$ by its image in $\sO(\bbX^\dagger)$, we see that $\sO(\bbX^\dagger)$ can also be written as an inductive limit of Banach space with injective and compact transition maps. It is therefore an LS space (limit of Smith spaces) of compact type \cite[corollary 3.38 (2)]{RR}. 
\end{para}

\begin{definition}
A {\it strictly pro-affinoid group} $\bbG^\dagger$ is a group object in the category of strictly pro-affinoid spaces which has a presentation $\bbG_{(\bullet)}$ where each $\bbG_{(n)}$ is an affinoid group, each relatively compact open embedding $\bbG_{(n)} \Supset \bbG_{(n+1)}$ is a homomorphism of affinoid groups, and the group object structure of $\bbG^\dagger$ is compatible with the group object structure on each $\bbG_{(n)}$. Hereafter, a {\it presentation} of a strictly pro-affinoid analytic group will always refer to a presentation of this form. 
\end{definition}

\begin{para}
To a strictly pro-affinoid group $\bbG^\dagger$, we associate the condensed group $G^\dagger = \varprojlim_n \bbG_{(n)}(F)$, where the $\bbG_{(n)}$ form a presentation of $\bbG^\dagger$ by affinoid groups and each $\bbG_{(n)}(F)$ is regarded as a condensed\footnotemark{} group using its natural topology. This is independent of choice of presentation.
\footnotetext{The ``underline'' functor of \cite[example 1.5]{ScholzeCondensed} is a right adjoint \cite[proposition 1.7]{ScholzeCondensed} and therefore preserves all limits. Thus, this limit of condensed sets agrees with the underline of the limit in the category of topological spaces.}
\end{para}

\subsection{\texorpdfstring{$F$}{F}-uniform groups}\label{F-uniform}

\begin{para}
Let $G$ be an $F$-uniform locally analytic\footnote{Recall our convention that ``locally analytic'' means ``locally $F$-analytic'' (cf. \cref{locally-F-analytic}).} group (cf. \cite[lemma 2.2.4 and remark 2.2.5]{OrlikStrauchIRR} or \cite[definition 3.4]{Lahiri}). The aim in this section is to associate affinoid groups to $G$, generalizing the construction of \cite[3.5]{Lahiri}. More precisely, define\footnote{This $\kappa$ is not to be confused with the strong limit cardinal of \ref{strong-limit-cardinal}. The use of $\kappa$ in this context follows other papers which we quote occasionally, which is why we decided to accept this conflict of notation.}
\begin{numequation}\label{kappa}
\kappa =\begin{cases} 1 & \text{if } p > 2 \\ 2 & \text{if } p = 2 \end{cases} \quad \mbox{ and } \quad \rho^* = p^{\kappa - (1/(p-1))} \in \bbR.
\end{numequation}
Note that $\rho^* > 1$ for all primes $p$. In what follows, we will show how to associate an affinoid group $\bbG[\rho]$ to $G$ for every $\rho \in (0, \rho^*) \cap p^\Q$. $\bbG = \bbG[1]$ is the affinoid group defined in \cite[3.5]{Lahiri}. The underlying analytic space of $\bbG[\rho]$ is a polydisk of radius $\rho$, and $\bbG[\rho] \Subset \bbG[\rho']$ whenever $\rho < \rho'$.
\end{para}

\begin{para} \label{polydisk-construction}
For a finite free $\cO_F$-module $M$ and $\rho \in p^\Q$, define the $\rho$-Gauss norm $|-|_\rho$ on $F[M] = \Sym_F(F \otimes_{\cO_F} M)$ to be the unique multiplicative norm with the property that, if $x = (x_1, \dotsc, x_d)$ is any $\cO_F$-basis for $M$, then
\[ \left| \sum_{n \in \bbN^d} a_n x^n \right|_\rho = \max_n |a_n|\rho^{|n|}. \]
Then let $F\langle M/\rho \rangle$ be the completion of $F[M]$ for the $\rho$-Gauss norm. We view 
\[ \Spa F\langle M/\rho \rangle \]
as an affinoid polydisk of radius $\rho$. There are natural relatively compact open embeddings $\Spa F\langle M/\rho \rangle \Subset \Spa F \langle M/\rho' \rangle$ whenever $\rho < \rho'$. 
\end{para}

\begin{para} \label{polydisk-functorial}
Suppose $\alpha : M \to N$ is a homomorphism of finite free $\cO_F$-modules. By putting $\alpha$ into Smith normal form, we see that there exist bases $x = (x_1, \dotsc, x_d)$ of $M$ and $y = (y_1, \dotsc, y_e)$ of $N$, and a sequence of integers $0 \leq a_1 \leq \cdots \leq a_k$ for some $k = 0, 1, \dotsc, d$, such that $\alpha(x_i) = \pi^{a_i}y_i$ for all $i \leq k$ and $\alpha(x_i) = 0$ for all $i > k$. Observe that, for any monomial $x^n \in F[M]$, we have
\[ |\alpha(x^n)|_\rho \leq |x^n|_\rho. \]
This is clear if $n_i \geq 1$ for some $i > k$ since then $\alpha(x^n) = 0$, and otherwise, we have 
\[ |\alpha(x^n)|_\rho = |\pi^{a_1n_1}y_1^{n_1} \cdots \pi^{a_k n_k}y_k^{n_k}|_\rho = |\pi|^{a_1n_1 + \cdots + a_kn_k} \rho^{n_1 + \cdots + n_k} \leq \rho^{|n|} = |x^n|_\rho. \]
Thus the induced map $F[M] \to F[N]$ is continuous with respect to the $\rho$-Gauss norm. It thus induces a continuous map $F\langle M/\rho \rangle \to F\langle N/\rho\rangle$, and a map of adic spaces $\Spa F\langle N/\rho \rangle \to \Spa F\langle M/\rho \rangle$. 
\end{para}

\begin{para} \label{bch-series}
Let $\Phi(X, Y) = \log(\exp(X)\exp(Y))$ be the formal Baker-Campbell-Hausdorff series in the free associative algebra over $\Q$ on indeterminates $X$ and $Y$. Then
\[ \Phi(X, Y) = \underbrace{X + Y}_{u_1(X, Y)} + \underbrace{\frac{1}{2}[X, Y]}_{u_2(X, Y)} + \cdots + u_n(X, Y) + \cdots, \]
where each $u_n(X, Y)$ is a homogeneous Lie polynomial of degree $n$ with rational coefficients, and for
\[ h_n = \left\lfloor \frac{n-1}{p-1} \right\rfloor, \]
the coefficients of $p^{h_n} u_n(X, Y)$ are $p$-adic integers \cite[IV.3.2.1--2]{LazardGroupesAnalytiques}. 
\end{para}

\begin{para} \label{powerful-bch}
It follows that the Lie polynomial $p^{h_n} u_n(X, Y)$ can be evaluated on a pair of elements $\frx, \fry$ in a $\Z_p$-Lie algebra $L$ to obtain another element $p^{h_n} u_n(\frx, \fry)$ of $L$. If $L$ is powerful in the sense that $[L, L] \subseteq p^\kappa L$ \cite[section 9.4]{DDMS}, then
\[ p^{h_n} u_n(\frx, \fry) \in p^{\kappa(n-1)}L \]
using the fact that $u_n(X, Y)$ is homogeneous of degree $n$, which implies that 
\[ u_n(\frx, \fry) \in p^{\kappa(n-1) - h_n}L. \]
Observe that 
\[ \kappa(n-1) - h_n \geq (n-1)\left( \kappa - \frac{1}{p-1} \right), \]
so
\[ |p^{\kappa(n-1) - h_n}| \leq p^{-(n-1)(\kappa - (1/(p-1)) } = (\rho^*)^{-(n-1)}.  \]
\end{para}

\begin{para} \label{comultiplication}
Suppose $L$ is an $\cO_F$-Lie algebra which is powerful as a $\Z_p$-Lie algebra. The calculations of \cref{powerful-bch} show that the Baker-Campell-Hausdorff formula defines a formal group law for the structure of an affinoid group on 
\[ \Spa F\langle L^\vee/\rho \rangle \]
for any $\rho \in (0, \rho^*)$. To be explicit, choose an $\cO_F$-basis $\frx = (\frx_1, \dotsc, \frx_d)$ of $L$. Then
\[ u_n(a_1 \frx_1 + \cdots + a_d \frx_d, b_1 \frx_1 + \cdots + b_d \frx_d) = u_{n,1}(a, b) \frx_1 + \cdots + u_{n,d}(a, b) \frx_d \]
where $u_{n,i}(a, b) = u_{n,i}(a_1, \dotsc, a_d, b_1, \dotsc, b_d)$ is a homogeneous polynomial of degree $n$ in the $2d$ variables $a_1, \dotsc, a_d, b_1, \dotsc, b_d$ with coefficients in $p^{\kappa(n-1) - h_n}\Z_p$. Since $|p^{\kappa(n-1) - h_n}| \leq (\rho^*)^{-(n-1)}$, the radius of convergence of the power series
\[ \Phi_i(a, b) = \sum_{n = 1}^\infty u_{n,i}(a, b) \]
is at least $\rho^*$. Let $x = (x_1, \dotsc, x_d)$ be the basis of $L^\vee$ dual to $\frx$. Let $a_i = (x_i, 0) \in L^\vee \oplus L^\vee$ and $b_i = (0, x_i) \in L^\vee \oplus L^\vee$, so that $(a, b) = (a_1, \dotsc, a_d, b_1, \dotsc, b_d)$ forms a basis for $L^\vee \oplus L^\vee$. With $\Phi_i(a, b)$ as above, we obtain a comultiplication
\[ \Delta : F\langle L^\vee/\rho \rangle \to F \langle L^\vee/\rho \rangle \sotimes_F F \langle L^\vee/\rho \rangle = F \langle (L^\vee \oplus L^\vee)/\rho \rangle \]
by setting $\Delta(x_i) = \Phi_i(a, b)$. The map $\Delta$ is independent of the choice of $\cO_F$-basis of $L$ and it gives $\Spa F\langle L^\vee / \rho \rangle$ the structure of an affinoid group. 
\end{para}

\begin{definition}
Suppose $G$ is an $F$-uniform locally $F$-analytic group. Let $L$ be the powerful $\Z_p$-Lie algebra corresponding to $G$ (cf. \cite[theorem 9.10]{DDMS} or \cite[paragraph 3.2]{Lahiri}), which is in fact an $\cO_F$-Lie algebra by the assumption of $F$-uniformity of $G$. For any $\rho \in (0, \rho^*) \cap p^\Q$, let \[ \bbG[\rho] = \Spa F\langle L^\vee/\rho \rangle, \]
with the affinoid group structure described in \cref{comultiplication}. Since the $\Z_p$-Lie algebra $L$ is functorially associated to $G$, this construction is functorial in $G$ (cf. \cref{polydisk-functorial}). 
\end{definition}

\begin{remark} \label{F-uniform-and-RR}
When $F = \Qp$ and $h \in \Q_{\geq 0}$, the group we denote $\bbG[p^{-h}]$ is denoted $\bbG_h$ in \cite[definition 4.4 and lemma 4.6]{RR}.
\end{remark}

\begin{definition}
For any $\rho \in (0, \rho^*) \cap p^\Q$, we define $\bbG[\rho]^\circ$ to be the strictly ind-affinoid group with a presentation by the affinoid groups $\bbG[\rho']$ for any increasing sequence $\rho' \to \rho^-$. Dually, we define $\bbG[\rho]^\dagger$ to be the strictly pro-affinoid group with a presentation by the affinoid groups $\bbG[\rho']$ for any decreasing sequence $\rho' \to \rho^+$. 
These constructions also define functors on the category of $F$-uniform groups.  
\end{definition}

\begin{remark}
Such constructions are generalized in \cite{LahiriDagger} to the setting of saturated $p$-valued groups \cite[page 187]{SchneiderBookLie}, but we will not make use of this here. 
\end{remark}

%% file: 2-distributions.tex
\section{Distribution algebras}

\subsection{Solid algebras and Hopf algebras}

\begin{para}
A {\it solid algebra} is an (associative, unital) algebra object in $\Vec_E^\solid$. If $R$ is a solid algebra, a {\it solid left (resp. right) $R$-module} is a left (resp. right) $R$-module object in the symmetric monoidal category $\Vec_E^\solid$. We assume modules are left modules unless otherwise specified, and we write $\Mod_R^\solid$ for the category of solid $R$-modules. It is a bicomplete abelian category. For any $\kappa$-small extremally disconnected $T$, define
\[ R[T]^\solid := R \sotimes_E E[T] \in \Mod_R^\solid. \]
\end{para}

\begin{lemma}[{\cite[A.29]{BoscoDrinfeldv2}}] \label{solid-modules-analytic-ring}
Let $R$ be a solid algebra. Then $R$, equipped with the functor of measures $T \rightsquigarrow R[T]^\solid$, is an analytic ring and $\Mod_R^\solid$ is the category of modules over this analytic ring. In particular, $R[T]^\solid$, for varying extremally disconnected sets $T$, forms a family of compact projective generators for $\Mod_R^\solid$. 
\end{lemma}

\begin{proof}
For a complex $C$ of condensed $R$-modules (i.e., $R$-modules in the symmetric monoidal category $\Vec^\cond$), observe that we have a commutative diagram as follows. 
\[ \begin{tikzcd} R\Hom_R(R[T]^\solid, C) \ar{r} \ar{d} & R\Hom(E[T]^\solid, C) \ar{d} \\ R\Hom_R(R[T], C) \ar{r} & R\Hom(E[T], C) \end{tikzcd} \]
The horizontal maps are adjunction isomorphisms. Thus, the vertical map on the left is an isomorphism if and only if the one on the right is, which in turn happens if and only if $C$ is solid. 

When $C$ happens to be a connective complex whose terms are direct sums of terms of the form $R[T]^\solid$ for various $T$, then $C$ is in fact solid. This shows that $R$, equipped with the functor $T \rightsquigarrow R[T]^\solid$, is an analytic ring \cite[definition 7.4]{ScholzeCondensed}. 

Finally, by taking $C = M$ to be a condensed $R$-module (regarded as a complex concentrated in a single degree), we see that $M$ is a module over the analytic ring we have just defined if and only if $M$ is solid \cite[proposition 7.5]{ScholzeCondensed}. The statement thus follows from \cite[proposition 7.5(i)]{ScholzeCondensed}. 
\end{proof}

\begin{para} \label{hopf-tensor}
A {\it solid Hopf algebra} is a Hopf algebra object in $\Vec_E^\solid$. We will always denote by $\Delta: R \ra R \sotimes_E R$ and $S : R \to R$ the comultiplication and antipode, respectively, of $R$. If $R$ is a solid Hopf algebra and $M, N, P$ are solid $R$-modules, then both $M \sotimes_E N$ and $\uHom_E(N, P)$ are again naturally solid $R$-modules, and
\[ \Hom_R(M \sotimes_E N, P) = \Hom_R(M, \uHom_E(N, P)) \]
naturally in $M$, $N$, and $P$. Given an $R$-module $M$, we denote by $M_\triv$ the underlying solid $E$-vector space, and equip it with the $R$-module structure given by the augmentation map $R \ra E$. We denote by $\mu_\triv: R \sotimes_E M_\triv \ra M_\triv$ the structure map of the solid $R$-module $M_\triv$. 
\end{para}

\begin{lemma} \label{hopf-tensor-trivial}
Let $R$ be a solid Hopf algebra, and let $M$ be a solid $R$-module Then the map
\[ \begin{tikzcd}
R \sotimes_E M \ar{r}{\Delta \otimes 1} & R \sotimes_E R \sotimes_E M_\triv \ar{r}{1 \otimes \mu_\triv} & R \sotimes_E M_\triv
\end{tikzcd} \]
is an isomorphism of solid $R$-modules. The domain and the target are given the structure of $R$-modules via $\Delta$.\footnote{In fact, this statement, and that of \cref{hopf-tensor-trivial-free}, are true for any Hopf algebra in any symmetric monoidal category. See \cite[lemma 4.13]{Knapp88} for the statement analogous to \cref{hopf-tensor-trivial-free} for the universal enveloping algebra of a Lie algebra.}
\end{lemma}

\begin{proof}
One first verifies that the map $\phi$ in the statement is in fact $R$-linear. Let $\mu: R \sotimes_E M \ra M$ be the structure map of $M$. Then, one checks that the composite
\[ \begin{tikzcd}
R \sotimes_E M_\triv \ar{r}{\Delta \otimes 1} & R \sotimes_E R \sotimes_E M_\triv \ar{d}{1 \otimes S \otimes 1} \\
& R \sotimes_E R \sotimes_E M_\triv \ar{r}{\id} & R \sotimes_E R \sotimes_E M \ar{r}{1 \otimes \mu} &  R \sotimes_R M
\end{tikzcd} \]
is also $R$-linear and is an inverse for $\phi$. All of this follows from the compatibility axioms for the structure maps of Hopf algebras. 
\end{proof}

\begin{corollary} \label{hopf-tensor-trivial-free}
Let $R$ be a solid Hopf algebra. Suppose $F$ is a free solid $R$-module and $M$ is any solid $R$-module. Then $F \sotimes_E M \cong F \sotimes_E M_\triv$ as solid $R$-modules.
\end{corollary}

\begin{proof}
Since $F$ is a direct sum of copies of $R$, this follows by taking a direct sum of the isomorphism of \cref{hopf-tensor-trivial}.
\end{proof}

\begin{proposition} \label{hopf-projective}
Let $R$ be a solid Hopf algebra. Suppose $P$ is a direct summand of a free solid $R$-module and $M$ is a solid $R$-module such that $M_\triv$ is Smith. Then $P \sotimes_E M$ is a projective solid $R$-module. 
\end{proposition}

\begin{proof}
We will show this first when $P$ is a direct sum of copies of $R$. Since solid tensor products commute with direct sums, and direct sums of projectives are again projective, it is sufficient to consider the case $P = R$. Since $M_\triv$ is Smith, it is of the form 
\[ M_\triv = E \sotimes_\Z \prod_I \Z \]
for some set $I$ (cf. \cite[lemma 3.8(ii)]{RR} and either \cite[equation (A.5)]{BoscoDrinfeldv2} or \cite[remark 2.17]{RR}), which means that  
\[ R \sotimes_E M \cong R \sotimes_E M_\triv = R \sotimes_\Z \prod_I \Z, \]
using \cref{hopf-tensor-trivial} for the first isomorphism. Since $\prod_I \Z$ is projective in the category of solid abelian groups \cite[theorem 5.8]{ScholzeCondensed}, its base change up to $R$ is projective in the category of solid $R$-modules. This shows the case when $P$ is a free solid $R$-module. 

Now suppose that $P$ and $P'$ are solid $R$-modules such that $P \oplus P'$ is a free solid $R$-module. Then 
\[ (P \oplus P') \sotimes_E M = (P \sotimes_E M) \oplus (P' \sotimes_E M) \]
is projective in the category of solid $R$-modules, as we have just seen, so the direct summand $P \sotimes_E M$ is projective as well. 
\end{proof}

\begin{para}\label{constructing-hopf}{\it Constructing Hopf algebras.}
Here is an observation that we will use to produce solid Hopf algebras. Let $\cC$ be a category with all finite products. We write $*$ for a terminal object. A group object $G$ in $\cC$ is then naturally a Hopf algebra in the symmetric monoidal category $(\cC, \times)$, where:
\begin{enumerate}[(a)]
\item The multiplication map $G \times G \to G$ is the multiplication map of the group object structure. 
\item The unit map $* \to G$ is the unit map of the group object structure. 
\item The antipode map $G \to G$ is the inversion map of the group object structure.  
\item The comultiplication map $G \to G \times G$ is the diagonal map. 
\item The counit map $G \to *$ is the terminal map. 
\end{enumerate}
If $G$ is a group object in $\cC$ and $F : (\cC, \times) \to (\cA, \otimes)$ is a {\it lax} symmetric monoidal functor, then the multiplication and unit maps above induce the structure of an algebra in $(\cA, \otimes)$ on $F(G)$.\footnote{For the definition of a lax monoidal functor, see, e.g., \cite[definition 3.1]{Aguiar-Mahajan}.}
If $F$ is {\it strong} symmetric monoidal, then all of the maps above induce the structure of a Hopf algebra in $(\cA, \otimes)$ on $F(G)$.\footnote{For the definition of a strong monoidal functor, see, e.g., \cite[definition 2.4.1]{EtingofTensorCats}.}
\end{para}

\begin{example}
For an elementary example, we can take $\cC$ to be sets and $(\cA, \otimes)$ to be the category of vector spaces. The free vector space functor is strong symmetric monoidal, from which it follows formally that the group algebra is a (classical) Hopf algebra. 
\end{example}

\begin{example}
More interestingly, this generality is applied in \cite[appendix, pages 311--313]{ST05}, where $\cC$ is taken to be the category of paracompact locally analytic manifolds and $(\cA, \otimes)$ the symmetric monoidal category of complete locally convex vector spaces with the completed injective tensor product $- \hat{\otimes}_\iota -$. The fact that the locally analytic distributions functor $M \rightsquigarrow D(M)$ is strong symmetric monoidal is \cite[proposition A.3]{ST05}, from which it follows formally that the locally analytic distribution algebra of a locally analytic group is naturally a Hopf algebra with respect to the completed injective tensor product. 
\end{example}

\subsection{Analytic distribution algebras}

We now use \cref{constructing-hopf} to construct three types of solid Hopf algebras of analytic distributions, corresponding to the three types of analytic groups (affinoid, strictly ind-affinoid, strictly pro-affinoid). To this end, let $\cC$ be the category of affinoid spaces, $\cC^\circ$ the category of strictly pro-affinoid spaces, and $\cC^\dagger$ the category of strictly pro-affinoid spaces. For $\ast \in \{\emptyset,\o,\dagger\}$ and an object $\bbX^*$ in $\cC^*$ we set 
$$D(\bbX^*) = \sO(\bbX^*)^\vee \;.$$

\begin{proposition} \label{affinoid-monoidal}
For each $\ast \in \{\emptyset, \o, \dagger\}$ the functor on $\cC^*$ with values in the category of solid $E$-vector spaces
\[\bbX \rightsquigarrow D(\bbX) \]
is strong symmetric monoidal. 
\end{proposition}

\begin{proof}
We only give the proof for $\ast \in \{\o, \dagger\}$, as the case for affinoid spaces is essentially the same (and even easier) than that of strictly ind-affinoid spaces. 

(1) {\it  Strictly ind-affinoid spaces.} Let $\bbX^\o$ have a presentation $\bbX^{(0)} \Subset \bbX^{(1)} \Subset \cdots$, and let $\bbY^\o$ have a presentation $\bbY^{(0)} \Subset \bbY^{(1)} \Subset \cdots$. Then $\bbX^{(0)} \x \bbY^{(0)} \Subset \bbX^{(1)} \x \bbY^{(1)} \Subset \cdots$ is a presentation of $\bbX^\o \x \bbY^\o$, and 

\[\begin{array}{rcl}\sO(\bbX^\o \x \bbY^\o) &=& \varprojlim_n \sO(\bbX^{(n)} \times \bbY^{(n)}) = \varprojlim_n \sO(\bbX^{(n)}) \sotimes_E \sO(\bbY^{(n)}) \\
&&\\
& =&  \varprojlim_{n,m} \sO(\bbX^{(n)}) \sotimes_E \sO(\bbY^{(m)}) = \sO(\bbX^\o) \sotimes_E \sO(\bbY^\o)\;,
\end{array}\]
where we have used \cite[3.13]{RR} for the second equality and \cite[3.28 (2)]{RR} for the fourth equality. Dualizing we get  
\[ D(\bbX^\o \times \bbY^\o) = (\sO(\bbX^\o) \sotimes_E \sO(\bbY^\o))^\vee = \sO(\bbX^\o)^\vee \sotimes_E \sO(\bbY^\o)^\vee  = D(\bbX^\o) \sotimes_E D(\bbY^\o) \;, \]
where we have used \cite[3.40 (2)]{RR} for the second equality.
Clearly $D(*) = E$, so we are done. 

(2) {\it Strictly pro-affinoid spaces.} Given two strictly pro-affinoid spaces $\bbX^\dagger$ and $\bbY^\dagger$ the map
\[ \sO(\bbX^\dagger) \sotimes_E \sO(\bbY^\dagger) \lra \sO(\bbX^\dagger \times_E \bbY^\dagger) \]
is seen to be an isomorphism by reducing to the affinoid case using the fact that solid tensor products commute with filtered colimits. Using that isomorphism we have
\[ D(\bbX^\dagger \x \bbY^\dagger) = \Big(\sO(\bbX^\dagger) \sotimes_E \sO(\bbY^\dagger)\Big)^\vee = \sO(\bbX^\dagger)^\vee \sotimes_E \sO(\bbY^\dagger)^\vee = D(\bbX^\dagger) \sotimes_E D(\bbY^\dagger) \;,\]
where the second equality holds by \cite[theorem 3.40 (2)]{RR}, since $\sO(\bbX^\dagger)$ and $\sO(\bbY^\dagger)$ are Fr\'echet spaces.
\end{proof}

\begin{corollary}\label{affinoid-hopf}
For $\bbG^*$ an affinoid (resp. strictly ind-affinoid, resp. strictly pro-affnoid) group, the Smith space (resp. $LS$ space of compact type, resp. Fr\'echet space of compact type) $D(\bbG)$ is naturally a solid Hopf algebra. 
\end{corollary}

\begin{proof} This follows from the principle explained in \cref{constructing-hopf}.
\end{proof}

\subsection{Chevalley-Eilenberg resolutions}

\begin{para}
Recall that, for a finite dimensional Lie algebra $\frg$, the {\it Chevalley-Eilenberg complex} $\CE(\frg)$ is a complex
\[ \begin{tikzcd} \cdots \ar{r} & 0 \ar{r} & U(\frg) \otimes_E \bigwedge^{\dim(\frg)} \frg \ar{r} & \cdots \ar{r} & U(\frg) \otimes_E \frg \ar{r} & U(\frg) \end{tikzcd} \]
concentrated in cohomological degrees $[-\dim(\frg), 0]$ \cite[section 7.7]{WeibelH}. It is a complex of finite free $U(\frg)$-modules, and the augmentation $\CE(\frg) \to E$ is a quasi-isomorphism. 
\end{para}

\subsubsection{Strictly ind-affinoid case}

\begin{proposition}[{\cite[proposition 5.12]{RR}}] \label{chevalley-eilenberg-ind-affinoid}
Let $\bbG^\circ$ be a strictly ind-affinoid group with a presentation $\bbG^{(\bullet)}$ where the underlying analytic space of each $\bbG^{(n)}$ is a polydisk. If $\frg$ is its Lie algebra, then $D(\bbG^\circ) \sotimes_{U(\frg)} \CE(\frg)$ is a projective resolution of $E$ as a trivial $D(\bbG^\circ)$-module.\footnotemark{} \qed
\end{proposition}

\footnotetext{Abstractly, we can regard $U(\frg)$ as a discrete analytic ring and define $D(\bbG^\circ) \sotimes_U(\frg) -$ as the solidification of the base change functor for analytic rings \cite[proposition 7.7]{ScholzeCondensed}. However, since $\CE_j(\frg) = U(\frg) \otimes_E \bigwedge^j \frg$, we see that $D(\bbG^\circ) \sotimes_{U(\frg)} \CE_j(\frg) = D(\bbG^\circ) \otimes_E \bigwedge^j \frg$ is just a finite direct sum of copies of $D(\bbG^\circ)$.}

\begin{corollary} \label{projdim-ind-affinoid}
Fix notation as in \cref{chevalley-eilenberg-ind-affinoid}. The projective dimension of $E$ over $D(\bbG^\circ)$ is at most $\dim(\bbG^\circ)$. \qed
\end{corollary}

\subsubsection{Strictly pro-affinoid case}

\begin{proposition} \label{chevalley-eilenberg-pro-affinoid}
Let $\bbG^\dagger$ be a strictly pro-affinoid group with a presentation $\bbG_{(\bullet)}$ where the underlying analytic space of each $\bbG_{(n)}$ is a polydisk. If $\frg$ is its Lie algebra, then $D(\bbG^\dagger) \sotimes_{U(\frg)} \CE(\frg)$ is a projective resolution of $E$ as a trivial $D(\bbG^\dagger)$-module.
\end{proposition}

\begin{proof}
We use the same proof as \cite[proposition 5.12]{RR}, but now the relevant Poincar\'e lemma is the `overconvergent' one \cite[theorem 5.4]{MonskyWashI}.
\end{proof}

\begin{corollary} \label{projdim-pro-affinoid}
With $\bbG^\dagger$ as in \cref{chevalley-eilenberg-pro-affinoid}. The projective dimension of $E$ over $D(\bbG^\dagger)$ is at most $\dim(\bbG^\dagger)$. \qed
\end{corollary}

\subsection{Mixed distribution algebras}\label{mixed-dist-algs}

\begin{para}\label{analytic-pairs}{\it Analytic group pairs.} Later on we will often consider a locally $F$-analytic group\footnote{Following convention \ref{locally-F-analytic} we usually drop the "$F$" and just speak of locally analytic groups.} $G$ and an open $F$-uniform normal subgroup $H \sub G$. Using the constructions in \cref{F-uniform} we can then attach to $H$ analytic groups $\bbH[\rho]$, $\bbH[\rho]^\circ$, $\bbH[\rho]^\dagger$ for $\rho \in (0,1] \cap p^\bbQ$.\footnote{These groups even exist for $\rho \in (0,\rho^*) \cap p^\bbQ$, but in order to make sure that the set of $F$-valued points of any of these rigid analytic groups is contained in $G$, we only consider them here for $\rho \le 1$.} Denote by $\bbH^*$ any of these analytic groups and set $H^* = \bbH^*(F)$. We call $(\bbH^*,G)$ an {\it analytic group pair}. We will now construct an analytic group $G.\bbH^*$ if $G/H$ is finite.
\end{para}

\begin{proposition}\label{analytic-groups-pair-group} Let $(\bbH^*,G)$ be an analytic group pair associated to $H \sub G$ as in \ref{analytic-pairs}. Assume that $G/H$ is finite. Then there is a analytic group $G.\bbH^*$ and an exact sequence of analytic groups 
$$1 \lra \bbH^* \lra G.\bbH^* \lra G/H^* \lra 1 \;.$$
The set of $F$-valued points of $G.\bbH^*$ is $G$. $G.\bbH^*$ is affinoid (resp strictly ind-affinoid, resp. strictly pro-affinoid), if $\bbH^*$ is affinoid (resp. strictly ind-affinoid, resp. strictly pro-affinoid). 
\end{proposition}

\begin{proof} It follows from functoriality of the construction of the group $\bbH^*$ in \cref{F-uniform} that $H^*$ is also normal in $G$. Moreover, for $g \in G$ the conjugation action of $g$ on $H^*$ extends to a rigid analytic automorphism $\alpha(g)$ of $\bbH^*$, and $G$ acts by rigid analytic automorphisms on $\bbH^*$.  Let $\cR$ be a set of coset representatives of $H^*$ in $G$, and for each $g \in \cR$ set $g.\bbH^* = \{g\} \x \bbH^*$, and $G.\bbH^* = \coprod_{gH \in \cR} g.\bbH^*$. Given $g_1, g_2 \in \cR$ let $r(g_1,g_2) \in \cR$ and $c(g_1,g_2) \in H^*$ be the unique elements such that $g_1g_2 = r(g_1,g_2)c(g_1,g_2)$. Then we define `multiplication maps'
$$g_1.\bbH^* \x g_2.\bbH^* \ra r(g_1,g_2).\bbH^*\,, \;\;((g_1,h_1), (g_2,h_2)) \mapsto (r(g_1,g_2),  c(g_1,g_2)\alpha(g_2^{-1})(h_1)h_2) \;.$$    
These maps induce a multiplication $G.\bbH^* \x G.\bbH^* \ra G.\bbH^*$ which turns $G.\bbH^*$ into a analytic group, because $G/H^*$ is finite too (as $H^*$ is of finite index in $H$).   

$G.\bbH^*$ is affinoid (resp. strictly ind-affinoid, resp. strictly pro-affinoid) if $\bbH^*$ is affinoid (resp. strictly ind-affinoid, resp. strictly pro-affinoid), since a finite coproduct of affinoid (resp. strictly ind-affinoid, resp. strictly pro-affinoid) spaces is again affinoid (resp. strictly ind-affinoid, resp. strictly pro-affinoid). 
\end{proof}

\begin{para}\label{defn-mixed-dist-algs}{\it Distribution algebras for analytic group pairs.} Given an analytic group pair $(\bbH^*,G)$ associated to $H \sub G$ with $G/H$ finite we set 
$$D(\bbH^*,G) := D(G.\bbH^*) \;,$$   
While these algebras are just distribution algebras of affinoid, or strictly ind-affinoid, or strictly pro-affinoid groups, we refer to them as 'mixed distribution algebras', as they are built out of the distribution algebra of the analytic group $\bbH^*$ and the locally analytic group $G$. They have been previously studied by Emerton \cite[sec. 5.3, p. 108]{EmertonA}, and we have adopted the notation here.
\end{para}

\begin{para}
The crossed product of an algebra with a group was first introduced in \cite{PassmanAlgebraicCrossed}, and then independently generalized to the crossed product of an algebra with a Hopf algebra in \cite{BlattnerCohenMontgomery} and \cite{DoiTakeuchi}. This generalization can be categorified to construct the crossed product of an algebra object with a Hopf algebra object in any symmetric monoidal category, and we use this construction in the symmetric monoidal category of solid vector spaces. 
\end{para}

\begin{proposition} \label{mixed-crossed}
Let $(\bbH^*, G)$ be an analytic group pair associated to $H \sub G$ as in \ref{analytic-pairs}. Let $H \sub \wH \trianglelefteq G$ be another open normal subgroup (not necessarily $F$-uniform). Assume that $G/H$ is finite, and let $Q = G/\wH$ be the quotient group. Then 
\[ D(\bbH^*, G) = D(\bbH^*,\wH) \cross_\zeta E[Q], \]
where $\zeta \in Z^2(E[Q], D(\bbH^*,\wH))$ is a normalized Hopf 2-cocycle.\footnote{For $\cH$ an Hopf algebra and $A$ an algebra in a symmetric monoidal category $(\cC, \otimes)$, a {\it normalized Hopf 2-cocycle} is a pair $\zeta = (a, t)$ where $a : \cH \otimes A \to A$ is a weak action \cite[definition 1.1]{BlattnerCohenMontgomery}, $t : \cH \otimes \cH \to A$ is normalized \cite[defintion 4.3]{BlattnerCohenMontgomery}, and the pair $(a, t)$ satisfy the cocycle and twisted module conditions \cite[lemma 4.5]{BlattnerCohenMontgomery}.}
\end{proposition}

\begin{proof}
Regard the finite quotient group $Q$ as a locally analytic group. The trivial subgroup $1 \subseteq Q$ is a normal open subgroup, and there is a natural morphism of analytic pairs $(\bbH^*, G) \to (1, Q)$ which induces a morphism of rigid analytic groups $\pi: G.\bbH^* \ra Q$, and gives thus a  morphism 
\[\pi_* : D(\bbH^*, G) \to D(1, Q) = E[Q] \]
of solid Hopf algebras. For any $g \in G$, the delta distribution $\delta_g$ maps to the element $g\wH \in E[Q]$. The desired result will follow from \cite[theorem 4.14]{BlattnerCohenMontgomery} once we show that $\pi_*$ is split as a coalgebra map and that its left Hopf kernel\footnotemark{} (cf. \cite[defintion 4.12]{BlattnerCohenMontgomery}) is $D(\bbH^*,\wH)$. 

\footnotetext{The {\it left Hopf kernel}  of $\pi_* : D(\bbH^*, G) \to E[Q]$ is the equalizer of the two maps $(1 \otimes \pi_*) \circ \Delta_* : D(\bbH^*, G) \to D(\bbH^*, G) \sotimes_E D(\bbH^*, G) \to D(\bbH^*, G) \sotimes_E E[Q]$ and $1 \otimes \eta : D(\bbH^*, G) = D(\bbH^*, G) \sotimes_E E \to D(\bbH^*, G) \sotimes_E E[Q]$, where $\Delta$ is the comultiplication of $D(\bbH^*, G)$ and $\eta$ is the unit of $E[Q]$.}

First, choose a set of coset representatives $\cS \subseteq G$ for $\wH$ in $G$ containing 1 and consider the map $\gamma : E[Q] \to D(\bbH^*, G)$ given by $g\wH \mapsto \delta_g$ for all $g \in \cS$. It is straightforward to verify that $\gamma$ is a coalgebra map that splits $\pi_*$ and satisfies $\gamma(1) = 1$. Next, note that the left Hopf kernel of $\pi_*$ is the equalizer of $(1 \otimes \pi_*) \circ \Delta_*$ and $1 \otimes \eta_*$, where $\Delta_*$ is the comultiplication on $D(\bbH^*, G)$ and $\eta_*$ is the unit map of $E[Q]$.
\[ \begin{tikzcd} 
& D(\bbH^*, G) \sotimes_E D(\bbH^*, G) \ar{dr}{1 \otimes \pi_*} \\
D(\bbH^*, G) \ar{ur}{\Delta_*} \ar[equals]{dr} & & D(\bbH^*, G) \sotimes_E E[Q] \\ 
& D(\bbH^*, G) \sotimes_E E \ar{ur}[swap]{1 \otimes \eta_*} 
\end{tikzcd} \]
To show that this equalizer is $D(\bbH^*,\wH)$, we first observe that the following is an equalizer diagram:
\[ \begin{tikzcd} 
& & G.\bbH^* \times G.\bbH^* \ar{dr}{1 \times \pi} \\
\wH.\bbH^* \ar{r} & G.\bbH^* \ar{ur}{\Delta} \ar[equals]{dr} & & G.\bbH^* \times Q \\ 
& & G.\bbH^* \times \{1\} \ar{ur}[swap]{1 \times \eta} 
\end{tikzcd} \]
We want to show that this equalizer is preserved under $D(-)$. By the duality result of \cite[lemma 3.40 (1)]{RR}, it is sufficient to show that the dual diagram is a coequalizer diagram:
\[ \begin{tikzcd} 
& & \displaystyle \sO(G.\bbH^*) \sotimes_E \displaystyle \sO(G.\bbH^*) \ar{dl}[swap]{\Delta^*} \\
\sO(\wH.\bbH^*) & \displaystyle \sO(G.\bbH^*) \ar{l} & & \displaystyle \sO(G.\bbH^*) \sotimes_E  \displaystyle E[Q]^\vee \ar{ul}[swap]{1 \otimes \pi^*} \ar{dl}{1 \otimes \eta^*}  \\ 
& & \displaystyle \sO(G.\bbH^*) \sotimes_E E \ar[equals]{ul}
\end{tikzcd} \]
Note that $\sO(G.\bbH^*) = \bigoplus_{q \in Q} \sO(q\wH.\bbH^*)$. Furthermore, for $q \in Q$ let $\lambda_q \in E[Q]^\vee$ be the basis element with $\lambda_q(q') = \delta_{q,q'}$. Then an element of the space on the far right can be written as $\sum_{(q,q') \in Q^2} f_{q,q'} \ot \lambda_{q'}$ with $f_{q,q'} \in \sO(q\wH.\bbH^*)$. The top map $\Delta^* \circ (1 \otimes \pi^*)$ sends this to $\sum_{q \in Q}  f_{q,q}$, while the bottom map $1 \otimes \eta^*$ sends this to $\sum_{q \in Q} f_{q,1}$. This shows that the projection map $\sO(G.\bbH^*) = \bigoplus_Q \sO(q\wH.\bbH^*) \to \sO(\wH.\bbH^*)$ onto the 1-component is in fact the coequalizer. 
\end{proof}

\begin{corollary} \label{mixed-crossed-corollary}
Let the notation and assumptions be as in \cref{mixed-crossed}. The solid algebra $D(\bbH^*, G)$ is strongly $Q$-graded. It is projective as a left (resp. right) module over $D(\bbH,\wH)$, and 
\[ E[Q] = D(\bbH, G) \sotimes_{D(\bbH,\wH)} E \;. \]
\end{corollary}

\begin{proof}
This follows from \cref{mixed-crossed} and categorifications of well-known results. For example, the strong $Q$-grading and projectivity are discussed in \cite[section 4]{BensonGoodearl}, and the tensor product statement follows from the discussion of \cite[example 1.2]{Passman}.
\end{proof}

\subsection{Resolutions for mixed distribution algebras}

\subsubsection{Infinite 2-periodic resolution for \texorpdfstring{$E$}{E} over \texorpdfstring{$E[Q]$}{E[Q]}}

\begin{para} \label{projdim-group-algebra}
Let $Q$ be a finite group. Then 
\[ e = \frac{1}{\# Q} \sum_{q \in Q} q \]
is an idempotent in the group algebra $E[Q]$, $(1-e)E[Q]$ is the augmentation ideal, and $H^s(Q, M) = \Ext^s_{E[Q]}(E, M) = 0$ for all $E[Q]$-modules $M$ and all $s > 0$ \cite[proposition 6.1.10 and its proof]{WeibelH}. In other words, $E$ is projective over $E[Q]$ \cite[pd lemma 4.1.6]{WeibelH}. 
\end{para}

\begin{para} \label{infinite-periodic}
We have the idempotent decomposition $E[Q] = eE[Q] \oplus (1-e)E[Q]$, which yields an infinite 2-periodic resolution of $E$ by finite free $E[Q]$-modules given by
$$ \begin{tikzcd} \cdots \ar{r} & E[Q] \ar{r}{1-e} & E[Q] \ar{r}{e} & E[Q] \ar{r}{1-e} & E[Q]. \end{tikzcd} $$
\end{para}

\subsubsection{The case of strictly ind-analytic groups}

\begin{para}\label{ind-affinoid-mixed}
Suppose $(\bbHo, G)$ is an analytic group pair associated to the open normal $F$-uniform subgroup $H \sub G$ of finite index, and $\bbHo$ is equal to $\bbH[\rho]^\o$ for some $\rho \in (0,1] \cap p^\Q$. Then $\bbH^\circ$ has a presentation $\bbH^{(\bullet)}$ where the underlying analytic space of each $\bbH^{(n)}$ is a polydisk. Set $H^\o = \bbHo(F)$ and let $Q = G/H^\o$ be the finite quotient group.
\end{para}

\begin{para} \label{wall-complex}
For all non-negative integers $i$ and $j$, let  
$$ X_{i,j} := D(\bbH^\circ, G) \sotimes_{U(\frg)} \CE_j(\frg) = D(\bbH^\circ, G) \otimes_E \bigwedge^j \frg $$
where $\CE_\bullet(\frg)$ is the Chevalley-Eilenberg complex \ref{chevalley-eilenberg-ind-affinoid}. Using \cref{mixed-crossed} and \cref{mixed-crossed-corollary}, we see that each $X_{i,\bullet}$, with the maps induced by the Chevalley-Eilenberg complex, is a free resolution $X_{i,\bullet} \to E[Q]$ over $D(\bbH^\circ, G)$. See \cref{wall-complex-figure}. 
\begin{figure}
\[ \begin{tikzcd}
& \vdots \ar{d} & \vdots \ar{d} & \vdots \ar{d} \\
& 0 \ar{d} & 0 \ar{d} & 0 \ar{d} \\
& D(\bbH^\circ, G) \otimes \bigwedge^d \frg \ar{d} & D(\bbH^\circ, G) \otimes \bigwedge^d \frg \ar{d} & D(\bbH^\circ, G) \otimes \bigwedge^d \frg \ar{d} \\
& \vdots \ar{d} & \vdots \ar{d} & \vdots \ar{d} \\
& D(\bbH^\circ, G) \otimes \frg \ar{d} & D(\bbH^\circ, G) \otimes \frg \ar{d} & D(\bbH^\circ, G) \otimes  \frg \ar{d} \\
& D(\bbH^\circ, G) \ar[lightgray]{d} & D(\bbH^\circ, G) \ar[lightgray]{d} & D(\bbH^\circ, G) \ar[lightgray]{d} \\
\color{lightgray} \cdots \ar[lightgray]{r} & \color{lightgray} E[Q] \ar[lightgray]{r}{e} & \color{lightgray} E[Q] \ar[lightgray]{r}{1-e} & \color{lightgray} E[Q]
\end{tikzcd} \]
\caption{The gray part of the diagram above is the infinite 2-period resolution of $E$ over $E[Q]$ (cf. \cref{infinite-periodic}) and the augmentation maps from $X_{i,\bullet}$ into its $i$th term. The black part defines the data of the Wall complex.} \label{wall-complex-figure}
\end{figure}

Thus the data of the $(X_{i,j})_{i,j}$ defines a Wall complex whose total complex $T_\bullet$ is a resolution of $E$ by finite free $D(\bbH^\circ, G)$-modules \cite[section V.3.1]{LazardGroupesAnalytiques}, \cite[theorem 6.1]{Kohlhaase_Cohomology}. For all non-negative integers $n$, we have 
\[ T_n = \bigoplus_{k = 0}^{\min\{n, d\}} \left( D(\bbH^\circ, G) \sotimes_{U(\frg)} \CE_k(\frg) \right) = \bigoplus_{k = 0}^{\min\{n, d\}} \left( D(\bbH^\circ, G) \otimes_E \bigwedge^k \frg \right). \]
\end{para}

\begin{proposition}\label{mixed-ind-affinoid-ext}
For any left $D(\bbH^\circ, G)$-module $M$ and any integer $n$, we have 
\[ \Ext^n_{D(\bbH^\circ, G)}(E, M) = \Ext^n_{D(\bbH^\circ)}(E, M)^Q \,. \]
In particular, the projective dimension of $E$ over $D(\bbH^\circ, G)$ is at most $\dim(G)$. 
\end{proposition}

\begin{proof}
Since $D(\bbH^\circ, G) = D(\bbH^\circ) \cross E[Q]$ and we can regard $M$ as a bimodule via the augmentation, there is a spectral sequence 
\[ E_2^{r,s} = H^s(Q, \Ext^r_{D(\bbH^\circ)}(E, M)) \Longrightarrow \Ext^{r+s}_{D(\bbH^\circ, G)}(E, M) \,. \]
by \cite[corollary 3.2.3]{GuccioneHochschild}.\footnote{Let $R$ be an augmented $E$-algebra and $M$ a left $R$-module. Then
\[ \Hoch^n(R, M) = \Ext^n_R(E, M) \]
where, on the left-hand side, we regard $M$ as an $R$-bimodule by letting $R$ act on the right through the augmentation $R \to E$. See {\cite[proposition X.3.4 and the subsequent note]{MaclaneH}}.}
As we noted in \cref{projdim-group-algebra}, we have $E_2^{r,s} = 0$ for $s > 0$. In other words, the spectral sequence collapses on the second page and we have 
\[ \Ext^n_{D(\bbH^\circ)}(E, M)^Q = \Ext^n_{D(\bbH^\circ, G)}(E, M) \]
for all $n$. The final assertion now follows from \cref{projdim-ind-affinoid}. 
\end{proof}

\begin{corollary}\label{mixed-ind-affinoid-ext-cor}
Let $d = \dim(G)$. The canonical truncation $\tau_{\leq d}T_\bullet$ of $T_\bullet$ provides a bounded projective resolution $\tau_{\leq d} T_\bullet \to E$ of $E$ over $D(\bbH^\circ, G)$. 
\end{corollary}

\begin{proof}
This follows from \cref{wall-complex}, \cref{mixed-ind-affinoid-ext}, \cite[pd lemma 4.1.6]{WeibelH}, and \cite[tag 0118]{stacks-project}. 
\end{proof}

\begin{para}\label{ind-affinoid-mixed-generalized} {\it Generalization to some analytic group pairs $(\bbHo,G)$ with $G/H$ countable.} In \cref{mixing} we consider resolutions for representations of non-compact groups which may have a non-compact center. In this setting we will encounter groups $G$ (which are denoted $\Psd$ there) which are compact modulo the center, and we consider an open normal $F$-uniform subgroup $H \sub G$. Then $G/H$ is no longer finite in general, but it is still countable in the context of \cref{mixing}, which is what we will assume here too. 

\vskip8pt

{\it The mixed distribution algebra.} We then consider an analytic group pair $(\bbHo,G)$ for $\bbHo$ a group of the form $\bbH[\rho]^\o$. We can also construct the rigid analytic group $G.\bbHo = \coprod_{gH \in G/H^\o} g.\bbHo$ as in \cref{analytic-groups-pair-group}, where $H^\o = \bbHo(F)$. If $G/H^\o$ is not finite, then $G.\bbHo$ is no longer a strictly ind-affinoid group as defined in \cref{dfn-strictly-ind-affinoid-group}, but it is a group object in the category $\cC^+$ of rigid analytic spaces which have an admissible covering by countably many affinoid subdomains (indeed, a countable disjoint union of strictly ind-affinoid spaces belongs to $\cC^+$). This category has all finite products and the terminal object $\{\ast\}$. Let $\bbX$ be a space in $\cC^+$, and let $(\bbX^{(n)})_n$ be an admissible covering of $\bbX$ by countably many affinoid subdomains. Then $\sO(\bbX) = \varprojlim_n \sO(\bbX^{(n)})$ is a Fr\'echet space, and the proof of \cref{affinoid-monoidal} works also in this context, which shows that the functor $\bbX \rightsquigarrow D(\bbX) = \sO(\bbX)^\vee$ is strong symmetric monoidal as well. It follows that $D(\bbHo,G) := D(G.\bbHo)$ is again a solid Hopf algebra whose underlying space is an $LS$ space.  

\vskip8pt

{\it Resolutions for $E_\triv$ as $D(\bbHo,G)$-module.} Now suppose that $C \sub G$ is a finitely generated free abelian subgroup of the center of $G$ such that $G/C$ is compact and $C \cap H = \{1\}$ (this assumption will be satisfied in the context of \cref{mixing} with $G$ replaced by $\Psd$). Set $\wH = CH^\o$, which is isomorphic to $C \x H^\o$. Then $\sO(\wH.\bbHo) = \prod_{c \in C} \sO(c.\bbHo)$ and 
\[D(\bbHo,\wH) = D(\wH.\bbHo) = \bigoplus_{c \in C} \delta_c D(\bbHo) = D(\bbHo)[C] = D(\bbHo) \otimes_E E[C]\] 
is the group algebra of $C$ over $D(\bbHo)$. Denote by $E_\triv$ the trivial $E[C]$-module where each $c \in C$ acts by multiplication by 1. Let $(c_1, \ldots, c_s)$ be a basis of $C$, so that $E[C] \cong E[T_1, T_1^{-1}, \ldots, T_s, T_s^{-1}]$ is the ring of Laurent polynomials in $T_1, \ldots, T_s$, where the indeterminate $T_i$ corresponds to the generator $c_i$. For the case of one variable we have the standard free resolution of $E_\triv$ as $E[T_i,T_i^{-1}]$-module 
\[0 \lra E[T_i,T_i^{-1}](T_i - 1) \lra E[T_i,T_i^{-1}] \lra E_\triv \lra 0 \;.\]
Using \cite[ch. IX, 2.7]{Cartan-Eilenberg} we obtain a resolution of length $s$ of $E_\triv$ as $E[C]$-module by finitely generated free $E[C]$-modules. Using \cite[ch. IX, 2.7]{Cartan-Eilenberg} again for the so-obtained resolution of $E_\triv$ as $E[C]$-module and the Chevalley-Eilenberg resolution of $E_\triv$ as $D(\bbHo)$-module, cf. \cref{chevalley-eilenberg-ind-affinoid}, we obtain a resolution of length $s+d$, $d= \dim_F(G)$, of $E_\triv$ as $D(\bbHo)[C]$-module by finitely generated free $D(\bbHo)[C]$-modules. Since $Q := G/\wH$ is finite, the statement (and its proof) of \cref{mixed-crossed} also applies to the distribution algebra $D(\bbHo,G)$ and its subalgebra $D(\bbHo,\wH)$ which gives that $D(\bbHo,G) = D(\bbH^\o,\wH) \# E[Q]$, and we have $D(\bbHo,G) \sotimes_{D(\bbHo,\wH)} E = E[Q]$. The discussion of \cref{wall-complex} applies to $D(\bbHo,G)$ and its subalgebra $D(\bbHo,\wH)$, and we obtain a resolution 
\begin{numequation}\label{EC-res-for-mixed-non-compact}
\CE_\bullet(\bbHo,G) \lra E_\triv \lra 0
\end{numequation}
of $E_\triv$ as $D(\bbHo,G)$-module by finitely generated free $D(\bbHo,G)$-modules, and those modules and their differentials can be made quite explicit. Again, the analogues of \cref{mixed-ind-affinoid-ext} and \cref{mixed-ind-affinoid-ext-cor} apply, and the truncated complex $\tau_{\le s+d} \CE_\bullet(\bbHo,G)$ provides a resolution of length $\le s+d$ of $E_\triv$ as $D(\bbHo,G)$-module by finitely generated projective $D(\bbHo,G)$-modules. 
\end{para}

\subsubsection{The case of strictly pro-affinoid groups}

\begin{para}\label{mixed-resolution-strictly-ind-affinoid-setup}
Suppose $(\bbHd, G)$ is an analytic group pair associated to the open normal $F$-uniform subgroup $H \sub G$ of finite index, and $\bbHd$ is equal to $\bbH[\rho]^\dagger$ for some $\rho \in (0,1] \cap p^\Q$. Then $\bbH^\dagger$ has a presentation $\bbH_{(\bullet)}$ where the underlying analytic space of each $\bbH_{(n)}$ is a polydisk. Set $H^\dagger = \bbHd(F)$ and let $Q = G/H^\dagger$ be the finite quotient group.
\end{para}

\begin{para}
For all non-negative integers $i$ and $j$, let  
$$ X_{i,j} := D(\bbH^\dagger, G) \sotimes_{U(\frg)} \CE_j(\frg) = D(\bbH^\dagger, G) \otimes_E \bigwedge^j \frg $$
where $\CE_\bullet(\frg)$ is the Chevalley-Eilenberg complex \ref{chevalley-eilenberg-pro-affinoid}. Using \cref{mixed-crossed} and \cref{mixed-crossed-corollary}, we see that each $X_{i,\bullet}$, with the maps induced by the Chevalley-Eilenberg complex, is a free resolution $X_{i,\bullet} \to E[Q]$ over $D(\bbH^\dagger, G)$. Thus the data of the $(X_{i,j})_{i,j}$ defines a Wall complex whose total complex $T_\bullet$ is a resolution of $E$ by finite free $D(\bbH^\dagger, G)$-modules \cite[section V.3.1]{LazardGroupesAnalytiques} \cite[theorem 6.1]{Kohlhaase_Cohomology}. For all non-negative integers $n$, we have 
\[ T_n = \bigoplus_{k = 0}^{\min\{n, d\}} \left( D(\bbH^\dagger, G) \sotimes_{U(\frg)} \CE_k(\frg) \right) = \bigoplus_{k = 0}^{\min\{n, d\}} \left( D(\bbH^\dagger, G) \otimes_E \bigwedge^k \frg \right). \]
The same arguments use in the proofs of \cref{mixed-ind-affinoid-ext} and \cref{mixed-ind-affinoid-ext-cor} give then:
\end{para}

\begin{proposition} \label{mixed-pro-affinoid-ext}
Let the notation be as in \cref{mixed-resolution-strictly-ind-affinoid-setup}. For any left $D(\bbH^\dagger, G)$-module $M$ and any integer $n$, we have 
\[ \Ext^n_{D(\bbH^\dagger, G)}(E, M) = \Ext^n_{D(\bbH^\dagger)}(E, M)^Q \,. \]
In particular, the projective dimension of $E$ over $D(\bbH^\dagger, G)$ is at most $\dim(G)$. \qed
\end{proposition}

\begin{corollary}\label{mixed-pro-affinoid-ext-cor}
Let $d = \dim(G)$. The canonical truncation $\tau_{\leq d}T_\bullet$ of $T_\bullet$ provides a bounded projective resolution $\tau_{\leq d} T_\bullet \to E$ of $E$ over $D(\bbH^\dagger, G)$. \qed
\end{corollary}

\begin{remark}\label{pro-affinoid-mixed-generalized} In analogy to what has been done in \cref{ind-affinoid-mixed-generalized} one can also construct distribution algebras of the type $D(\bbHd,G)$ in the case when $(\bbHd,G)$ is an analytic group pair with $\bbHd$ strictly pro-affinoid and $G/H$ countable. The results \ref{mixed-crossed} and \ref{mixed-crossed-corollary} hold for $D(\bbHd,G)$ and its subalgebra $D(\bbHd,\wH)$ if $G/\wH$ is finite. Furthermore, let us assume again that there is a finitely generated free abelian subgroup $C$ of the center of $G$ such that $H \cap C = \{1\}$ and $G/C$ is compact. We set $\wH = CH$. Then the trivial one-dimensional module $E$ over $D(\bbHd,\wH)$ has a resolution of length $s+d$, where $s = \rk(C)$ and $d = \dim_F(G)$, by finitely generated free $D(\bbHd,\wH)$-modules. Using the Wall complex, one then constructs a resolution of $E$ by finitely generated free $D(\bbHd,G)$-modules. The projective dimension of $E$ is at most $s+d$. As we will not be making use of this construction and the corresponding results in this paper, we will not go further into the details here.      
\end{remark}

\subsection{Banach distribution algebras \`a la Schneider-Teitelbaum}

\begin{para} \label{Qp-uniform}
Let $G$ be a $\Qp$-uniform locally analytic group (i.e., a locally $F$-analytic group $G$ such that $\Res^F_\Qp G$ is $\Qp$-uniform; e.g., $G$ could be an $F$-uniform locally $F$-analytic group). Let $e = e(F/\Qp)$ be the ramification index, $q = |k_F|$ the cardinality of the residue field, and let $r \in [1/p,1)$ denote a real number. 
\end{para}

\begin{para}\label{ST-theory-setup}
We equip $G$ always with the canonical $p$-valuation $\omega^{\can}$ coming from the lower $p$-series as in \cite[2.2.3]{OrlikStrauchIRR}. The norm $\|\cdot\|_r$ on $D(\Res^F_\Qp G)$ is defined as in \cite[page 160]{ST03} or \cite[2.2.6]{OrlikStrauchIRR} with $\omega = \omega^{\can}$.\footnote{Recall our convention that $D(G)$ refers to $D(G, E) = D(G) \otimes_F E$ and $D(\Res^F_\Qp G)$ refers to $D(\Res^F_\Qp, E) = D(\Res^F_\Qp G) \otimes_{\Qp} E$ (cf. \cref{tacit-base-extension}).}
If $(h_1, \ldots, h_d)$ is a system of topological generators as in \cite[page 716]{OrlikStrauchIRR} and $b_i = h_i-1$, then we have $\|b_i\|_r = r^\kappa$, with $\kappa$ as in \cref{kappa}.
\end{para}

\begin{para} \label{distribution-algebra-as-quotient}
Recall that the locally analytic distribution algebra $D(G)$ is a quotient of $D(\Res^F_\Qp G)$ \cite[step 2 in the proof of theorem 5.1]{ST03}. The quotient norm on $D(G)$ is denoted by $\|\cdot\|_\ovr$, and we let $D_r(G)$ be the completion of $D(G)$ with respect to the quotient norm. This ring is also a quotient of $D_r(\Res^F_\Qp G)$. Indeed, if $I_G = \ker(D(\Res^F_\Qp G) \to D(G))$, then 
\[ D_r(G) = D_r(\Res^F_\Qp G)/I_GD_r(\Res^F_\Qp G) \] 
by \cite[proof of proposition 3.7]{ST03}. $D_r(G)$ is noetherian for $r \in p^\bbQ \cap [1/p,1)$ \cite[theorem 4.5 and remark 4.6]{ST03}.
\end{para}

\begin{para} \label{the-set-sR}
In order to make use of certain technical results from \cite{OrlikStrauchIRR}, we will assume below that $r$ is an element of
\begin{numequation}
\sR = \{r \in (1/p,1) \cap p^\bbQ \;\midc\;  \exists m \ge 0 \mbox{ such that } p^{-\frac{1}{p-1} - \frac{1}{eq^m}} < r^{\kappa p^m} < p^{-\frac{1}{p-1}}\} \;.
\end{numequation}
It is easy to see that for every $r \in \sR$ there is a unique $m \in \bbZ_{\ge 0}$ such that 
\[ p^{-\frac{1}{p-1} - \frac{1}{eq^m}} < r^{\kappa p^m} < p^{-\frac{1}{p-1}}. \]
\end{para}

\begin{proposition}\label{Dr-functoriality}
Let $G$ be a $\Qp$-uniform locally analytic group, and let $H$ be a $\Qp$-uniform open subgroup of $G$ (cf. \cref{Qp-uniform}). Let $r \in [1/p, 1)$. 
\begin{enumerate}
\item The natural map of group algebras $E[H] \to E[G]$ extends to a morphism of Banach algebras $i^r_{H,G}: D_r(H) \ra D_r(G)$ which is norm-decreasing (i.e., $\|i_{H,G}(\delta)\|_{\ovr,G} \leq \|\delta\|_{\ovr,H}$ for all $\delta \in D_r(H)$, where $\|\cdot\|_{\ovr,G}$ and $\|\cdot\|_{\ovr,H}$ denote the norms on $D(G)$ and $D_r(H)$, respectively). 
\item If $r \in \sR$ (cf. \cref{the-set-sR}), the map $i^r_{H,G}$ is injective.
\end{enumerate}\end{proposition}

\begin{proof} 
For (1), set $G_0 = \Res^F_\Qp G$, $H_0 = \Res^F_\Qp H$, $d = \dim_F(G)$, and $n=[F:\Qp]$.  Let $(h_1, \ldots, h_{nd})$ and $(g_1, \ldots, g_{nd})$ be systems of topological generators of $H$ and $G$, respectively, as in \cite[page 716]{OrlikStrauchIRR}. Write $b_i = g_i-1$, and $h_j = g_1^{\nu_1} \cdot \ldots \cdot g_d^{\nu_d}$ with $\nu_1, \ldots, \nu_d \in \Zp$. Then 
\[ \begin{aligned} h_j-1 &= (1+b_1)^{\nu_1} \cdot \ldots \cdot (1+b_d)^{\nu_d} -1 \\
&= \sum_{\substack{j_i \ge 0, \\ (j_1, \ldots, j_d) \neq (0, \ldots, 0)}} \binom{\nu_1}{j_1} \cdots \binom{\nu_d}{j_d} b_1^{j_1} \cdot \ldots \cdot b_d^{j_d} \end{aligned} \]
where the series on the right converges in $D_r(G_0)$ and has $\|\cdot\|_{r,G_0}$-norm at most $r^\kappa$, as all binomial coefficients are in $\Zp$ and the `constant term' is zero. Hence, if 
\[ \delta = \sum_\alpha d_\alpha (h_1-1)^{\alpha_1}\cdot \ldots \cdot (h_d-1)^{\alpha_d} \]
is a convergent series in $D_r(H_0)$, we find that for each term $d_\alpha (h_1-1)^{\alpha_1}\cdot \ldots \cdot (h_d-1)^{\alpha_d}$ we have $\|d_\alpha (h_1-1)^{\alpha_1}\cdot \ldots \cdot (h_d-1)^{\alpha_d}\|_{r,G_0} \le |d_\alpha|r^{\kappa |\alpha|}$, and this series therefore converges in $D_r(G_0)$ and its $\|\cdot\|_{r,G_0}$-norm is bounded by $\|\delta\|_{r,H_0}$. Set $I_{r,G} = I_G D_r(G_0)$ and $I_{r,H} = I_H D_r(H_0)$, with $I_G$ and $I_H$ as in \ref{distribution-algebra-as-quotient}. Let $\delta \in D_r(H_0)$, and denote by $\overline{\delta}$ its class in $D_r(H)$. Then we have for the quotient norms $\|\cdot\|_{\ovr,G}$ and $\|\cdot\|_{\ovr,H}$ on $D_r(G)$ and $D_r(H)$, respectively:
\[ \begin{aligned} 
\|i^r_{H,G}(\overline{\delta})\|_{\ovr,G} 
&= \inf_{\mu \in I_{r,G}} \Big\{\|i^r_{H_0,G_0}(\delta) + \mu\|_{r,G_0}\Big\} \\
&\leq \inf_{\mu \in I_{r,H}} \Big\{\|i^r_{H_0,G_0}(\delta + \mu)\|_{r,G_0} \Big\} \\
&\leq \inf_{\mu \in I_{r,G}} \Big\{\|\delta + \mu\|_{r,H_0}\Big\} \\
&= \|\overline{\delta}\|_{\ovr,H}
\end{aligned} \]
For (2), we choose $\cO_F$-bases $(\frx_1, \ldots, \frx_d)$ and $(\fry_1, \ldots, \fry_d)$ of $\Lie_\Zp(G)$ and $\Lie_\Zp(H)$, respectively, such that $\fry_i = \vpi^{\nu_i}\frx_i$ for some $\nu_i \in \bbZ_{\ge 0}$. Set $g_j = \exp_G(\frx_j)$ and $b_j = g_j-1$, $j=1, \ldots,d$. (Caution: these elements are different from those above.) As we remarked in \cref{the-set-sR}, there is a unique $m \in \bbZ_{\ge 0}$ such that, when we set $s = r^{p^m}$, we have 
\[ p^{-\frac{1}{p-1} - \frac{1}{eq^m}} < s^\kappa < p^{-\frac{1}{p-1}}. \]
Note that $\frg = \Lie_\Zp(G) \otimes_{\Z_p} E$. Then $\frx_j = \log(1+b_j) = -\sum_{k>0} \frac{(-1)^k}{k}b_j^k$. A simple calculation shows that this series converges in $D_s(G)$, because $\|b_j\|_{\ovs,G} \le \|b_j\|_{s,G_0} = s^\kappa < p^{-\frac{1}{p-1}}$. Moreover $\|\frx_j\|_\ovs \le s^\kappa$, and $b_j = \sum_{k>0} \frac{1}{k!}\frx^k$ converges in $D_s(G)$ too. It follows that the universal enveloping algebra $U(\frg)$ is dense in $D_s(G)$. So the closure $U_s(\frg,G)$ of $U(\frg)$ in $D_s(G)$ is actually equal to $D_s(G)$. By \cite[1.4.2]{Kohlhaase_Distributions}, we therefore have  \[ D_s(G) = U_s(\frg,G) = \left\{ \sum_\alpha d_\alpha \frx^\alpha \;:\; |d_\alpha| \, \lim_{|\alpha| \ra \infty} \nu_{\ovs,G}(\frx^\alpha) = 0 \right\} \]
where $\frx^\alpha = \frx_1^{\alpha_1} \cdot \ldots \cdot \frx_d^{\alpha_d}$, and the norm $\nu_{\ovs,G}$ is equivalent to $\|\cdot\|_{\ovs,G}$. Because $\nu_\ovr(\sum_\alpha d_\alpha \frx^\alpha) = \max |d_\alpha|\nu_\ovr(\frx^\alpha)$, the coefficients $d_\alpha$ are uniquely determined by the sum $\sum_\alpha d_\alpha \frx^\alpha$. For exactly the same reasons we also have 
\[ D_s(H) = U_s(\frg,H) = \left\{\sum_\alpha c_\alpha \fry^\alpha \;:\; |c_\alpha| \, \lim_{|\alpha| \ra \infty} |c_\alpha|\,\nu_{\ovs,H}(\fry^\alpha) = 0\right\} \]
where the norm $\nu_{\ovs,H}$ is equivalent to $\|\cdot|_{\ovs,H}$ and the coefficients $c_\alpha$ are uniquely determined by $\sum_\alpha c_\alpha \fry^\alpha$. The map $i^s_{H,G}: D_s(H) \ra D_s(G)$ is then given by 
\[ i^s_{H,G}\left(\sum_\alpha c_\alpha \fry^\alpha\right) = \sum_\alpha c_\alpha \left(\prod_{j=1}^d \vpi^{\nu_j\alpha_j}\right) \frx^\alpha \]
which shows that this map is injective. We finish the proof with the following commutative diagram
\[\begin{tikzcd}
D_r(H) \ar{r}{i^r_{H,G}} \ar{d}[swap]{\can} & D_r(G) \ar{d}{\can} \\ 
D_s(H) \ar{r}[swap]{i^s_{H,G}} & D_s(G)
\end{tikzcd}\]
where the vertical maps are the canonical transition maps in the systems of Banach algebras. By our assumption on $r$ and \cite[5.2.1 (i)]{OrlikStrauchIRR} these canonical maps are injective. Since the bottom horizontal map is injective, as we have just seen, the map $i^r_{H,G}$ is injective too.
\end{proof}

\begin{remarks} 
(1) It is not hard to see from the proof that in general one has 
\[ \|i^r_{H,G}(\delta)\|_{\ovr,G} < \|\delta\|_{\ovr,H}. \]
For example, this is the case when $\delta = h_j -1$ and $h_j = g_1^{\nu_1} \cdot \ldots \cdot g_d^{\nu_d}$ where all $\nu_i$ are in $p\Zp$. This implies that, in general, $i^r_{H,G}(D_r(H))$ is not equal to the closure of $D(H)$ in $D_r(G)$. 

\vskip8pt

(2) In \cref{Dr-functoriality} it seems plausible to us that $D_r(H) \ra D_r(G)$ is injective for all $r \in (\frac{1}{p},1)$. Because of our use of the results in \cite[5.2.1]{OrlikStrauchIRR}, we can prove it here only for $r \in \sR$. 
\end{remarks}

\begin{para}\label{D(h)}{\it The distribution algebras $D_{(h)}(G)$.} 
When $F = \Qp$ and $G$ is a $\Qp$-uniform locally analytic group, another kind of distribution algebra has been considered by Rodrigues Jacinto and Rodr\'iguez Camargo in \cite[definition 4.12]{RR}. For $h \in \bbQ_{>0}$, they define
\[ D_{(h)}(G) = \left\{\sum_\alpha d_\alpha \bb^\alpha \;:\; d_\alpha \in E,\; \sup_\alpha \{|d_\alpha|p^{-\frac{|\alpha|}{(p-1)p^h}}\} < \infty\right\} \]
where $\bb_\alpha  = b_1^{\alpha_1} \cdot \ldots \cdot b_d^{\alpha_d}$. Then $D_{(h)}(G)$ contains the $\cO_E$-module 
\[ D_{(h)}(G,\cO) = \left\{\sum_\alpha d_\alpha \bb^\alpha \;:\; d_\alpha \in E,\; \sup_\alpha \{|d_\alpha|p^{-\frac{|\alpha|}{(p-1)p^h}}\} \leq 1 \right\} \]
When $h \in \bbZ_{>0}$, then $p^{-\frac{1}{(p-1)p^h}} \in p^\bbQ$ and this $\cO_E$-module is profinite (cf. \cite[remark 4.13]{RR}), and one can use it to define on $D_{(h)}(G)$ the structure of a Smith algebra (i.e., a solid algebra whose underlying solid vector space is Smith). 
\end{para}

\begin{para}\label{D<r}{\it The distribution algebras $D_{<r}(G)$.} On the other hand, for any $h \in \bbQ_{>0}$, if $r^\kappa = p^{-\frac{1}{(p-1)p^h}}$, we have $D_{(h)}(G) = D_{<r}(G)$, where 
\[ D_{<r}(G) = D_{<r}(G,E) = \left\{\sum_\alpha d_\alpha \bb^\alpha \;:\; d_\alpha \in E,\; \sup_\alpha \{|d_\alpha|r^{\kappa |\alpha|}\} < \infty\right\} \,. \]
These $E$-algebras have been introduced for any $r \in (\frac{1}{p},1)$ in \cite[p. 162]{ST03}, and they are known to be noetherian for $r \in (\frac{1}{p},1) \cap p^\Q$ by \cite[4.8]{ST03}. In the latter paper they are equipped with the norm 
\[\|\sum_\alpha d_\alpha \bb^\alpha \|_r = \sup_\alpha \{|d_\alpha|r^{\kappa |\alpha|}\} \,,\]
and they are considered as Banach algebras in \cite{ST03}. 
\end{para}
\begin{para}\label{D<r-as-Smith-algebra}{\it $D_{<r}(G)$ as a Smith algebra.} In this subsection, we assume $r \in (\frac{1}{p},1) \cap p^\Q$. In that case we can also equip the ring $D_{<r}(G)$ with the structure of a Smith algebra over $E$. To see this, let $\vpi_E$ be a uniformizer of $E$ and we write $r^\kappa = |\vpi_E|^\frac{a}{m}$ with coprime positive integers $a,m$. Denote by $D_{<r}(G)_{\le 1} = \{\delta \in D_{<r}(G) \mid \|\delta\|_r \le 1\}$ the unit ball in $D_{<r}(G)$, and for any non-negative integer $k$ set $N_k = \{\alpha \in \bbN^d \midc \lfloor \frac{a |\alpha|}{m}\rfloor = k\}$. The following map 
\[\iota: D_{<r}(G)_{\le 1} \lra \prod_{k=0}^\infty \prod_{\alpha \in N_k} \cO_E\,,\;\; 
\sum_{\alpha \in \bbN^d} d_\alpha \bb^\alpha \mapsto (d_\alpha \vpi^k)_{k \ge 0, \alpha \in N_k} \,,\]
is easily seen to be bijective. On $\cO_E$ we consider its natural topology and give $\prod_{k=0}^\infty \prod_{\alpha \in N_k} \cO_E$ the product topology, which we transfer via $\iota^{-1}$ to $D_{<r}(G)_{\le 1}$. We write $D^{\rm compact}_{<r}(G)_{\le 1}$ for this topological $\cO_E$-module. Then we give $D_{<r}(G)$ the finest locally convex topology for which the inclusion $D^{\rm compact}_{<r}(G)_{\le 1} \hra D_{<r}(G)$ is continuous. An $\cO_E$-lattice $L \sub D_{<r}(G)$ is open in this topology if and only if for any $c \in E^\x$ the intersection $D^{\rm compact}_{<r}(G)_{\le 1} \cap cL$ is open in $D^{\rm compact}_{<r}(G)_{\le 1}$. With this topology $D_{<r}(G)$ becomes a Smith space which we denote by $D^S_{<r}(G)$. It follows from \cite[4.1]{ST03} that the multiplication on $D^S_{<r}(G)$ is continuous. When we consider $D_{<r}(G)$ as a Banach algebra we write $D^B_{<r}(G)$ for it. It follows from the explicit description of these rings that there are canonical continuous maps of $E$-algebras $D(G) \ra D_r(G) \ra D^B_{<r}(G) \ra D^S_{<r}(G)$. 

\vskip8pt

When $G$ is an $F$-uniform locally $F$-analytic group we set \[D_{<r}(G) := D_{<r}(\Res^F_\Qp G)/I_G D_{<r}(\Res^F_\Qp G)\,,\] with $I_G$ as in \ref{distribution-algebra-as-quotient}. The ideal $I_G \sub D_{<r}(\Res^F_\Qp G)$ is finitely generated by \cite[5.1]{SchmidtAUS}. Therefore,  $I_G D^B_{<r}(\Res^F_\Qp G)$ (resp. $I_G D^S_{<r}(\Res^F_\Qp G)$) is closed in $D^B_{<r}(\Res^F_\Qp G)$ (resp. $D^S_{<r}(\Res^F_\Qp G)$). If we equip $D_{<r}(\Res^F_\Qp G)$ with its Banach space (resp. Smith space) topology, then we give $D_{<r}(G)$ the corresponding quotient topology and write $D_{<r}^B(G)$ (resp. $D_{<r}^S(G)$) for this topological $E$-algebra which is again a Banach space (resp. Smith space, by \cite[3.9]{RR}). If a statement only depends on the algebraic structure of this ring, we drop the superscript.
\end{para}

\begin{prop}\label{weak-FS-via-D<r} Let $G$ be an $F$-uniform locally $F$-analytic group. Let $r,s \in (\frac{1}{p},1) \cap p^\Q$. Then 
\begin{enumerate}
    \item The ring $D_{<r}(G)$ is (left- and right-) noetherian. 
    \item The map $D(G) \ra D^S_{<r}(G)$ is continuous and (left- and right-) flat.
    \item The map $D(G) \ra D^S_{<r}(G)$ has dense image. 
    \item For $s < r$ the maps $D^S_{<r}(G) \ra D_s(G) \ra D^S_{<s}(G)$ are continuous. 
    \item Let $(r_n)_{n \ge 1}$ be an increasing sequence in $(\frac{1}{p},1) \cap p^\Q$ converging to 1, then the canonical continuous maps in (2) give rise to an isomorphism of topological algebras $D(G) = \varprojlim_n D^S_{<r_n}(G)$, and the projective system $(D^S_{<r_n}(G))_n$ defines on $D(G)$ the structure of a weak Fr\'echet-Stein algebra in the sense of \cite[1.2.6]{EmertonA}.
\end{enumerate}
\end{prop}

\begin{proof} (1) This is an immediate consequence of $D_{<r}(\Res^F_\Qp G)$ being noetherian, cf. \cite[4.8]{ST03}.

\vskip8pt

(2) Since $D_r(\Res^F_\Qp G) \ra D_{<r}(\Res^F_\Qp G)$ if flat by \cite[4.8]{ST03}, the same is true for $D_r(G) \ra D_{<r}(G)$, since 
\[D_{<r}(G) = D_{<r}(\Res^F_\Qp G) \ot_{D_r(\Res^F_\Qp G)} D_r(G) \;.\]
Because $D_r(G)$ is part of a Fr\'echet-Stein algebra structure on $D(G)$, the map $D(G) \ra D_r(G)$ is flat, cf. \cite[3.2]{ST03}. As a composition of flat maps, $D(G) \ra D_{<r}(G)$ is flat. And $D(G) \ra D^S_{<r}(G)$ is continuous because it is a composition of continuous maps $D(G) \ra D_r(G) \ra D^B_{<r}(G) \ra D^S_{<r}(G)$. The last of these maps is continuous since the Banach space topology is finer than the Smith topology. 

\vskip8pt

(3) It suffices to show that $D(G) \ra D^S_{< r}(G)$ has dense image when $F=\Qp$. Let $\delta = \sum_{\alpha \in \bbN^d} d_\alpha \bb^\alpha$ be any element in $D^S_{< r}(G)$, and let $L$ be an open lattice in $D^S_{< r}(G)$. Choose $c \in E^\x$ such that $c\delta \in D^S_{<r}(G)_{\le 1}$. With the notation introduced in \cref{D<r-as-Smith-algebra} we have that $\iota(cL \cap D^S_{<r}(G)_{\le 1})$ is open in $\prod_{k=0}^\infty \prod_{\alpha \in N_k} \cO_E$. Hence there is $k_0 \ge 0$ such that $(cd_\alpha \vpi^k)_{k \ge k_0, \alpha \in N_k} \in \iota(cL \cap D^S_{<r}(G)_{\le 1})$. Therefore, if we set $\delta_0 = \sum_{\alpha \in \bbN^d, |\alpha < k_0} d_\alpha \bb^\alpha$, which lies in the image of $D(G) \ra D^S_{< r}(G)$, we find that $c\delta -c\delta_0 \in cL \cap D^S_{<r}(G)_{\le 1}$, and it follows that $\delta-\delta_0 \in L$.

\vskip8pt

(4) Since $D_s(G) \ra D^S_{<s}(G)$ is continuous, the assertion follows once we have seen that $D^S_{<r}(G) \ra D_s(G)$ is continuous. We show that the preimage of the unit ball in $D_s(G)$ is open in $D^S_{<r}(G)$. As these rings are equipped with the quotient topologies coming from the corresponding rings for $\Res^F_\Qp G$, we may assume $F=\Qp$ here. Denote by $\phi: D^S_{<r}(G) \ra D_s(G)$ the canonical map, fix $t \in \bbR_{>0}$, and let $D_s(G)_{\le t}$ be the open lattice of elements $\delta = \sum_\alpha d_\alpha \bb^\alpha \in D_s(G)$ such that $\|\delta\|_s := \sup_\alpha \{|d_\alpha|s^{\kappa |\alpha|}\} \le t$. Choose $k_0 \in \Z_{\ge 0}$ such that $\left(\frac{s}{r}\right)^{\kappa k_0} \le |\vpi| t$. Choose $m \in \bbZ_{\ge 0}$ be large enough so that $p^{-m} \le |\vpi| t$ and set $M := \prod_{|\alpha| < k_0} p^m\cO_E \x \prod_{|\alpha| \ge k_0} \cO_E$. Then, if $\lambda := \sum_\alpha c_\alpha \bb^\alpha \in D^S_{<r}(G)$ is such that $\iota(\lambda) \in M$, then
\[\|\phi(\lambda)\|_s = \sup_\alpha \left\{|c_\alpha|s^{\kappa |\alpha|}\right\} \le |\vpi|^{-1} \sup_\alpha \left\{|c_\alpha \vpi^{\lfloor \frac{a|\alpha|}{m}\rfloor}| \left(\frac{s}{r}\right)^{\kappa |\alpha|}\right\} \, \le t \,.\] 
It follows that $\phi^{-1}\Big(D_s(G)_{\le t}\Big)$ contains $M$, which is an open subgroup of $D^{\rm compact}_{<r}(G)_{\le 1}$, and $\phi$ is thus continuous.  

\vskip8pt

(5) We show that the conditions of \cite[1.2.6]{EmertonA} are satisfied for this projective system. Each algebra $D^S_{<r}(G)$ is a Smith space, hence complete, and since the class of Smith spaces is closed under passage to quotients by closed subspaces \cite[3.9]{RR}, Smith spaces are hereditarily complete, cf. \cite[1.1.39]{EmertonA} for the definition of {\it hereditarily complete}. By (4) the map $D^S_{<r}(G) \ra D^S_{<s}(G)$ is a BH-map. It also follows from (4) that the canonical continuous map $D(G) = \varprojlim_n D_{r_n}(G) \ra \varprojlim_n D^S_{<r_n}(G)$ is an isomorphism of topological algebras. Finally, by (3), the map $D(G) \ra D^S_{<r}(G)$ has dense image.
\end{proof}

\vskip8pt

The algebras $D_{<r}(G)$ can also be used to describe the distribution algebras of strictly ind-affinoid or strictly pro-affinoid groups as follows. 

\begin{proposition}\label{distalgs-of-wide-open-and-overconv-gps-and-Dr}
Let $G$ be an $F$-uniform locally $F$-analytic group. Suppose $\rho \in (0,1] \cap p^\bbQ$, and let $\bbG[\rho]^\circ$ and $\bbG[\rho]^\dagger$ be the associated strictly ind-affinoid and strictly pro-affinoid groups, respectively. Set $r(\rho) = p^{-\frac{\rho}{\kappa(p-1)}}$. Then one has isomorphisms of topological $E$-algebras
\begin{enumerate}
\item $\displaystyle D(\bbG[\rho]^\circ, G) = \varinjlim_{r > r(\rho)} D^S_{<r}(G) = \varinjlim_{r > r(\rho)} D_r(G)$.
\item $\displaystyle D(\bbG[\rho]^\dagger, G) = \varprojlim_{r < r(\rho)} D^S_{<r}(G) = \varprojlim_{r < r(\rho)} D_r(G)$.
\end{enumerate}
\end{proposition}

\begin{proof} (1) We assume first $F=\Qp$. Recall that the affinoid group denoted by $\bbG_h$ in \cite[definition 4.4]{RR} is equal to the group denoted here $\bbG[p^{-h}]$, cf. \cref{F-uniform-and-RR}. Furthermore, the groups denoted $\bbG_{h^+}$ and $\bbG^{(h^+)}$ in \cite[definition 4.4]{RR} are the groups $\bbG[p^{-h}]^\circ$ and $G.\bbG[p^{-h}]^\circ$ considered here. By \cite[corollary 4.18]{RR} and \cref{D<r} we have,
\[ D(\bbG[p^{-h}]^\circ,G) = D^{(h^+)}(G) = \varinjlim_{h' > h} D_{(h')}(G) = \varinjlim_{r > r(h)} D^S_{< r}(G) = \varinjlim_{r > r(h)} D_r(G), \]
where $r(h) = p^{-\frac{1}{\kappa(p-1)p^h}}$. The stated formula follows if we set $\rho = p^{-h}$. 

\vskip8pt

If $F$ is any finite extension of $\Qp$, let $H = \Res^F_\Qp G$ which is a uniform pro-$p$ group (it is the same as $G$, but considered as a locally $\Qp$-analytic group). Let $\bbH$ be the affinoid group over $\Qp$ associated to $H$. By what we have just shown, statement (1) is true for $H$, i.e., 
\[D(\bbH[p^{-h}]^\circ,H) = \varinjlim_{r > r(h)} D_r(H)\,. \]
Now, $D(\bbG[p^{-h}]^\circ,G) = D(\bbH[p^{-h}]^\circ,H)/I_G D(\bbH[p^{-h}]^\circ,H)$, where $I_G = \ker(D(H) \ra D(G))$ is the ideal of continuous linear forms on $C^\la(H)$ which vanish on $C^\la(G)$. Moreover, as recalled in \cref{distribution-algebra-as-quotient}, we also have $D_r(G) = D_r(H)/I_G D_r(H)$. Therfore, the case of a general finite extension $F/\Qp$ follows from the case $F = \Qp$.

\vskip8pt

(2) We use (1) to observe that 
\[ D(\bbG[p^{-h}]^\dagger,G) = \varprojlim_{h'<h} D(\bbG[p^{-h'}]^\circ,G) = \varprojlim_{r<r(h)} \varinjlim_{r' >r} D_{r'}(G) = \varprojlim_{r<r(h)} D_r(G). \qedhere \]
\end{proof}

%% file: 3-representations.tex
\section{Condensed and solid representations} \label{representations}

\subsection{Condensed representations}

\begin{para}\label{underline-functor}
In this section, $G$ denotes a condensed group. For example, if $\cG$ is a topological group, then the functor $\underline{\cG}: \{\mbox{profinite sets}\}^\op \ra \{\mbox{groups}\}$, $S \rightsquigarrow C(S,\cG) = \{f: S \ra \cG \midc f \mbox{ continuous}\}$ is a condensed group \cite[example 1.5]{ScholzeCondensed}. We call $\cG \rightsquigarrow \underline{\cG}$ the {\it underline functor}. 
\end{para}

\begin{definition}
A {\it condensed representation} of $G$ is a condensed vector space $V$ over $E$ equipped with a morphism $G \times V \to V$ of condensed sets such that $G(S) \times V(S) \to V(S)$ is an $E$-linear action of the group $G(S)$ on the vector space $V(S)$ for all profinite sets $S$. 
We write $\Rep(G)$ for the category of condensed representations of $G$ over $E$. We write $\Hom_G(V, W)$ and $\uHom_{E_\solid[G]}(V, W)$ for the external and internal homs, respectively (the former is a vector space, the latter is a condensed vector space). 
\end{definition}

In the definition above, it is sufficient to require that $G(S) \times V(S) \to V(S)$ be a linear action of $G(S)$ on $V(S)$ for all extremally disconnected $S$ \cite[proposition 2.7]{ScholzeCondensed}.

\begin{definition}
We denote by $E[G]$ the condensed algebra that is the sheafification of the presheaf $S \mapsto E[G(S)]$, i.e., the presheaf which assigns to any profinite set $S$ the group algebra $E[G(S)]$ on the group $G(S)$ of sections of $G$ over $S$. 
\end{definition}

\begin{lemma} \label{condensed-representation-lemma}
The following are in natural bijection for any condensed $E$-vector space $V$. 
\begin{enumerate}[(a)]
\item Morphisms $G \times V \to V$ of condensed sets which give $V$ the structure of a condensed representation of $G$ over $E$. 
\item Morphisms $G \to \uAut_E(V)$ of condensed groups. 
\item Morphisms $E[G] \otimes_E V \to V$ of condensed vector spaces which give $V$ the structure of a condensed $E[G]$-module. 
\end{enumerate}
\end{lemma}

\begin{proof}
Observe that
\[ \Hom(G \times V, V) = \Hom(G, \uHom(V, V)). \]
Suppose that on the left side of this bijection, we have a map $G \times V \to V$ as in (a). Since $G(S) \times V(S) \to V(S)$ is an $E$-linear action of $G(S)$ on $V(S)$ functorially in the profinite set $S$, for any continuous map $T \to S$ of profinite sets we obtain a natural action map $G(S) \times V(T) \to V(T)$. Thus the corresponding map $G \to \uHom(V, V)$ on the right side of the above bijection has the property that the map 
\[ G(S) \to \uHom(V, V)(S) = \Hom(V|_S, V|_S) = \left(\Hom(V(T), V(T)) \right)_{T \to S} \]
on sections over $S$ factors uniquely through a group homomorphism 
\[ G(S) \to \uAut_E(V)(S) = \Aut_E(V|_S) = \left(\Aut_E(V(T)) \right)_{T \to S}. \]
Varying $S$ gives the equivalence of (a) and (b). 

\vskip8pt

The domain of an $E[G]$-module structure map $E[G] \otimes_E V \to V$ is the sheafification of the presheaf tensor product, so such a structure map is equivalent to a collection of $E[G(S)]$-module structure maps maps 
\[ E[G(S)] \otimes_E V(S) \to V(S) \]
functorially in the profinite set $S$. Such a structure map is equivalent to an action map $G(S) \times V(S) \to V(S)$, giving the equivalence of (a) and (c). 
\end{proof}

\begin{corollary} \label{condensed-rep-modules}
The category of condensed representations of $G$ over $E$ is equivalent to the category of condensed $E[G]$-modules. \qed
\end{corollary}

\begin{theorem} \label{condensed-rep-category}
\begin{enumerate}[(a)]
\item The forgetful functor from $\Rep(G)$ into the category of condensed vector spaces creates limits and colimits. In other words, limits and colimits exist and can be computed in the category of condensed vector spaces.
\item $\Rep(G)$ is a Grothendieck abelian category. 
\item The condensed representations $E[G] \otimes_E E[T]$ for extremally disconnected sets $T$ form a generating family of compact projective objects for $\Rep(G)$.
\end{enumerate}
\end{theorem}

\begin{proof}
The first two statements follow from \cref{condensed-rep-modules} plus generalities about categories of modules for sheaves of algebras on a site \cite[theorem 18.1.6]{KashiwaraSchapiraCategories}. The third follows from the fact that the condensed vector spaces $E[T]$ for extremally disconnected $T$ form a generating family of compact projective objects for the category of condensed vector spaces \cite[proof of theorem 2.2]{ScholzeCondensed}, together with the observation that
\[ \Hom_G(E[G] \otimes_E E[T], V) = \Hom_E(E[T], V) \]
for any condensed representation $V$.
\end{proof}

\subsection{Solid representations}

We continue to assume that $G$ is a condensed group. 

\begin{definition}
A {\it solid representation} of $G$ over $E$ is a condensed representation of $G$ over $E$ whose underlying condensed vector space is solid. We write $\Rep^\solid(G)$ for the full subcategory of $\Rep(G)$ consisting of the solid representations. 
\end{definition}

\begin{lemma} \label{solid-rep-reflective}
Solidification $V \mapsto V^\solid$ defines a left adjoint to the forgetful functor $\Rep^\solid(G) \to \Rep(G)$. In other words, $\Rep^\solid(G)$ is a reflective subcategory of $\Rep(G)$.
\end{lemma}

\begin{proof}
This is formal in light of \cite[proposition 7.5(i)]{ScholzeCondensed} and \cite[proposition A.17(i)]{BoscoDrinfeldv2}. Since solidification is a functor on the category of condensed vector spaces, we know that there is a natural morphism $\uAut_E(V) \to \uAut_E(V^\solid)$ of condensed groups for any condensed vector space $V$. If $V$ is a condensed representation of $G$, composing with the morphism $G \to \uAut_E(V)$ from \cref{condensed-representation-lemma} gives $V^\solid$ a natural structure of a solid representation. That the resulting functor $\Rep(G) \to \Rep^\solid(G)$ is left adjoint to the forgetful functor follows immediately from the corresponding fact for solidification of condensed vector spaces.
\end{proof}

\begin{theorem} \label{solid-rep-category}
$\Rep^\solid(G)$ is a full abelian subcategory of $\Rep(G)$ stable under extensions, limits, and colimits. The solid representations $(E[G] \otimes_E E[T])^\solid$ for extremally disconnected $T$ form a generating family of compact projective objects for $\Rep^\solid(G)$. In particular, $\Rep^\solid(G)$ is a Grothendieck abelian category. 
\end{theorem}

\begin{proof}
An extension of solid vector spaces in the category of condensed vector spaces is again solid (cf. \cite[proposition 7.5(i)]{ScholzeCondensed} and \cite[proposition A.17(i)]{BoscoDrinfeldv2}), so the same is true of representations. Limits and colimits in $\Rep(G)$ can be computed in the category of condensed vector spaces (cf. \cref{condensed-rep-category}), and solid vector spaces are stable under all limits and colimits (cf. \cite[proposition 7.5(i)]{ScholzeCondensed} and \cite[proposition A.17(i)]{BoscoDrinfeldv2}), so $\Rep^\solid(G)$ is stable under all limits and colimits. The final statement follows from the fact that $E[G] \otimes_E E[T]$ form a generating family of compact projective objects for $\Rep(G)$ (cf. \cref{condensed-rep-category}), plus the adjunction isomorphism
\[ \Hom_G((E[G] \otimes_E E[T])^\solid, V) = \Hom_G(E[G] \otimes_E E[T], V) \]
for any solid representation $V$ (cf. \cref{solid-rep-reflective}). 
\end{proof}

We now show compatibility of our definitions with those of \cite{RR}. 

\begin{lemma} \label{solidification-iwasawa}
Suppose $G$ is a profinite group regarded as a condensed group via the underline functor of \cref{underline-functor}. The solidification $E[G]^\solid$ of the condensed algebra $E[G]$ is naturally isomorphic to the Iwasawa algebra $E_\solid[G]$ of \cite[definition 4.1]{RR}. 
\end{lemma}

\begin{proof}
Let $\Z[G]$ be the sheafification of the presheaf $S \mapsto \Z[G(S)]$ on the site of profinite sets, so that 
\[ E[G] = E \otimes \Z[G]. \]
Since $G$ is profinite, the solidification of $\Z[G]$ is 
\[ \Z[G]^\solid = \varprojlim_H \Z[G/H] \]
where $H$ varies over normal open subgroups of $G$ \cite[definition 5.1 and theorem 5.8]{ScholzeCondensed}. Thus, if we let $\sM_E$ be the functor of measures of the analytic ring associated to $E$, we have
\[ E[G]^\solid = (E \otimes \Z[G])^\solid = E \sotimes \Z[G]^\solid = \sM_E(G) = E_\solid[G] \]
where the second isomorphism is from \cite[proof of theorem 6.2]{ScholzeCondensed}, the third holds by definition of $\sM_E$, and the fourth is \cite[remark 4.3\footnotemark]{RR}.
\footnotetext{The group $G$ is assumed to be a {\it compact} $p$-adic Lie group in \cite{RR}, but this assumption is not used for \cite[definition 4.1 and remark 4.3]{RR}. }
\end{proof}

\begin{corollary}[cf. {\cite[lemma 4.19]{RR}}] \label{RR-compatibility}
Suppose $G$ is a profinite group regarded as a condensed group via the underline functor of \cref{underline-functor}. The category of solid representation of $G$ over $E$ is equivalent to the category of solid modules over the Iwasawa algebra $E_\solid[G]$ of \cite[definition 4.1]{RR}.
\end{corollary}

\begin{proof}
Recall that we write $- \otimes_E -$ for the tensor product of condensed vector spaces (ie, before solidification; cf. \cref{condensed}), even when both arguments are solid. If $V$ is a solid vector space, observe that 
\[ \begin{aligned}
\Hom_E(E_\solid[G] \otimes_E V, V) &= \Hom_E(E[G]^\solid \otimes_E V, V) \\
&= \Hom_E((E[G]^\solid \otimes_E V)^\solid, V) \\
&= \Hom_E((E[G] \otimes_E V)^\solid, V) \\
&= \Hom_E(E[G] \otimes_E V, V)
\end{aligned} \]
using \cref{solidification-iwasawa} for the first bijection, the fact that $V$ is solid for the second and fourth, and \cite[proof of theorem 6.2]{ScholzeCondensed} for the third. Under this natural bijection 
\[ \Hom_E(E_\solid[G] \otimes_E V, V) = \Hom_E(E[G] \otimes_E V, V), \]
$E_\solid[G]$-module structure maps on the left side correspond to the $E[G]$-module structure maps on the right side. 
\end{proof}

\subsection{Analytic vectors} \label{analytic-vectors}

\begin{para}
In this subsection we discuss a variety of functors on the category of solid representations. These functors are adaptations of the functor of invariant vectors in smooth representation theory, and are of several (overlapping) types: there are ``classical'' constructions of analytic vectors of representations of the condensed group associated to an analytic group, and then there are two adaptations of this construction.  
\end{para}

\subsubsection{Classical constructions of analytic vectors}\label{an-vectors-classical}

\begin{definition}[cf. {\cite[definition 3.3.1]{EmertonA}, \cite[definition 4.29]{RR}}]
Let $\bbG$ be an affinoid group and let $G = \bbG(F)$. For $V$ be a solid representation of $G$, the tensor product $\sO(\bbG) \sotimes_{E} V$ carries an action of $G \times G \times G$ where the first and second factors act on $\sO(\bbG)$ via the left and right regular actions, respectively, and the third factor acts on $V$. We define 
\[ V^{\bbG\anv} = \uHom_{E_\solid[G]}(E, \sO(\bbG) \sotimes_{E} V), \] 
using the $\star_{1,3}$ action of $G$ on $\sO(\bbG) \sotimes_E V$.\footnote{See \cite[definition 4.24]{RR} for notation.} Explicitly, if we think of an `element' in $\sO(\bbG) \sotimes_{E} V$ as a $V$-valued function on $\bbG$, then the $G$-action is given by $(g \star_{1,3} f)(x) = g \cdot_V f(g^{-1}x)$, where $\cdot_V$ indicates the action of $G$ on $V$. The action of $G$ on this space of invariants is then given by $\star_2$-action, i.e., $(g \ast_2 f)(x) = f(xg)$.
\end{definition}

\begin{definition}[cf. {\cite[definition 3.4.1]{EmertonA}, \cite[definition 4.29]{RR}}]
Suppose $\bbG^\circ$ is a strictly ind-affinoid group and let $\bbG^{(\bullet)}$ be a presentation. 
For $V$ a solid representation of the condensed group $G^\circ = \bbG^\circ(F)$, define \[ V^{\bbG^\circ\anv} = \varprojlim_n V^{\bbG^{(n)}\anv}, \] where $\bbG^{(\bullet)}$ is a presentation of $\bbG^\circ$. This is independent of choice of presentation. 
\end{definition}

\begin{definition}
Suppose $\bbG^\dagger$ is a strictly pro-affinoid group and let $\bbG_{(\bullet)}$ be a presentation. 
For $V$ be a solid representation of the condensed group $G^\dagger = \bbG^\dagger(F)$, define \[ V^{\bbG^\dagger\anv} = \varinjlim_n V^{\bbG_{(n)}\anv}, \] where $\bbG_{(\bullet)}$ is a presentation of $\bbG^\dagger$. This is independent of choice of presentation.
\end{definition}

\subsubsection{Emerton-style adaptations} \label{emerton-style-adaptations}

\begin{definition}
Let $G$ be a condensed group and let $D$ be a solid vector space equipped with commuting left and right actions of $G$. For $V$ a solid representation of $G$, define
\[ V^D = \uHom_{E_\solid[G]}(E, D^\vee \sotimes_E V) \]
using the $\star_{1,3}$ action of $G$ on $D^\vee \sotimes_E V$.\footnote{See \cite[definition 4.24]{RR} for notation.} The right action of $G$ on $D$ induces a left action of $G$ on $V^D$, so $V \rightsquigarrow V^D$ defines a functor from solid representations of $G$ to itself.
\end{definition}

\begin{example}
Suppose $\bbG$ is an affinoid group and $G = \bbG(F)$. Then $D(\bbG)^\vee = \sO(\bbG)$ \cite[lemma 3.10]{RR}, so $V^{D(\bbG)} = V^{\bbG\anv}$ for any solid representation $V$ of $G$. 
\end{example}

\begin{lemma} \label{frechet-ind-affinoid-analytic}
Suppose $\bbG^\circ$ is a strictly ind-affinoid group and that $V$ is a Fr\'echet representation of $G^\circ = \bbG^\circ(F)$. Then $V^{D(\bbG^\circ)} = V^{\bbG^\circ\anv}$.
\end{lemma}

\begin{proof}
Suppose $\bbG^{(\bullet)}$ is a presentation of $\bbG^\circ$. On the one hand, the definition of $V^{\bbG^\circ\anv}$ says that
\[ V^{\bbG^\circ\anv} = \varprojlim_n V^{\bbG^{(n)}\anv} = \varprojlim_n \uHom_{E_\solid[G^{(n)}]}(E, \sO(\bbG^{(n)}) \sotimes V) \]
On the other hand, we have
\[ V^{D(\bbG^\circ)} = \uHom_{G^\circ}(E, \sO(\bbG^\circ) \sotimes V) = \uHom_{E_\solid[G]}(E, (\varprojlim_n \sO(\bbG^{(n)}))\sotimes V), \]
using the fact that $D(\bbG^\circ)^\vee = \sO(\bbG^\circ)$ \cite[theorem 3.40]{RR}. It follows from \cite[proposition 1.1.32]{EmertonA} plus the fact that $\uHom$ commutes with limits in the second entry that $V^{\bbG^\circ\anv} = V^{D(\bbG^\circ)}$ when $V$ is Fr\'echet.
\end{proof}

\begin{lemma}
Suppose $\bbG^\dagger$ is a strictly pro-affinoid group and that $V$ is a solid representation of $G = \bbG^\dagger(F)$, then $V^{\bbG^\dagger\anv} = V^{D(\bbG^\dagger)}$.
\end{lemma}

\begin{proof}
Observe that $D(\bbG^\dagger)^\vee = \sO(\bbG^\dagger)$ \cite[theorem 3.40]{RR}. If $\bbG_{(\bullet)}$ is a presentation for $\bbG^\dagger$, then
\[ \begin{aligned} V^{D(\bbG^\dagger)} &= \uHom_{E_\solid[G]}(E, \sO(\bbG^\dagger) \sotimes_E V) \\
&= \uHom_{E_\solid[G]}(E, (\varinjlim \sO(\bbG_{(n)})) \sotimes_E V) \\
&= \uHom_{E_\solid[G]}(E, \varinjlim (\sO(\bbG_{(n)}) \sotimes_E V)) \\
&= \varinjlim_n \uHom_{E_\solid[G]}(E, \sO(\bbG_{(n)}) \sotimes_E V) \\
&= \varinjlim_n V^{\bbG_{(n)}\anv},
\end{aligned}\]
using the fact that solid tensor products commute with filtered colimits for the third isomorphism, and the fact that $E$ is a compact object in the category of solid representations of $G$ for the fourth. 
\end{proof}

\begin{proposition}\label{Dr-analytic-Smith} Suppose $V$ is an admissible locally analytic representation of $G$. Then, for all $r \in p^\bbQ \cap [1/p, 1)$, the underlying solid vector space of $V^{D_r(G)}$ is Smith.
\end{proposition}

\begin{proof} 
By definition, $V^{D_r(G)} = \uHom_{E_\solid[G]}(E,D_r(G)^\vee \sotimes_E V)$. Since $D_r(G)^\vee$ is a Smith space and $V$ is an $LS$ space, the tensor product $D_r(G)^\vee \sotimes_E V$ is an $LS$ space too. Applying \cite[4.30 (1) and (2)]{RR} we find that 
$$\uHom_{E_\solid[G]}(E,D_r(G)^\vee \sotimes_E V) = \uHom_{E_\solid[G]}(D_r(G) \sotimes_E V^\vee,E)\;.$$ 
Furthermore, $\uHom_{E_\solid[G]}(D_r(G) \sotimes_E V^\vee,E) = (D_r(G) \sotimes_{E_\solid[G]} V^\vee)^\vee$. By \cite[5.11]{RR} we have $D(G) \sotimes^\bbL_{E_\solid[G]} D(G) =  D(G)$ 
and thus $D(G) \sotimes_{E_\solid[G]} D(G) =  D(G)$ too (here we consider the non-derived solid tensor product). This shows that
$$\begin{array}{rcl}D_r(G) \sotimes_{E_\solid[G]} V^\vee & = & D_r(G) \sotimes_{D(G)} \left(D(G) \sotimes_{E_\solid[G]} D(G)\right) \sotimes_{D(G)} V^\vee \\
&&\\
& = & D_r(G) \sotimes_{D(G)} V^\vee
\end{array}$$
As $V^\vee$ is a coadmissible $D(G)$-module, $D_r(G) \sotimes_{D(G)} V^\vee$ is a finitely generated $D_r(G)$-module, and carries a canonical Banach space topology. The dual of this space is therefore a Smith space. 
\end{proof}

\subsubsection{Representable adaptations}

\begin{para} \label{representable-para}
Let $G$ be a condensed group and let $D$ be a solid vector space equipped with commuting left and right actions of $G$. For $V$ a solid representation of $G$, consider
\[ V \mapsto \uHom_{E_\solid[G]}(D, V), \]
using the left action of $G$ on $D$ to take the $G$-equivariant $\uHom$. The right action of $G$ on $D$ induces a left action of $G$ on $\uHom_{E_\solid[G]}(D, V)$, so $V \mapsto \uHom_{E_\solid[G]}(D, V)$ defines a functor from solid representations of $G$ to itself. This functor is representable (by $D$). One can consider this construction with $D$ being any of the distribution algebras discussed in \cref{emerton-style-adaptations} above. 
\end{para}

\begin{proposition}\label{uHomDrV-has-Dr-module-structure}
Suppose that $G$ is a condensed group, that $V$ is a solid representation of $G$, and that $D$ is a solid algebra equipped with a homomorphism $E[G] \to D$ of condensed algebras. Then $D$ is naturally equipped with commuting left and right actions of $G$, and $\uHom_{E_\solid[G]}(D, V)$ is naturally a solid left $D$-module via precomposition, ie, via the map\footnote{The condensed tensor product $D \otimes_E \uHom_{E_\solid[G]}(D, V)$ in this statement can be replaced with the solid tensor product $D \sotimes_E \uHom_{E_\solid[G]}(D, V)$.}
\[ \begin{aligned} 
D \otimes_E \uHom_{E_\solid[G]}(D, V) &\to \uHom_{E_\solid[G]}(D, V)  \\
(\mu, \phi) &\mapsto [x \mapsto \phi(x\mu)].
\end{aligned} \]
\end{proposition}

\begin{proof}
Observe that $\uHom_{E_\solid[G]}(D, V)$ is naturally a left $D$-module as described above since $D$ is an $(E[G], D)$-bimodule (cf. \cite[proposition 7.2.4]{AdkinsWeintraub}). To check that it is solid, observe first that $\uHom_E(D, V)$ is solid since both $D$ and $V$ are (cf. \cref{solid}). It follows that
\[ \uHom_E(E[S], \uHom_E(D, V)) = \uHom_E(E[S]^\solid, \uHom_E(D, V)) \]
for any profinite set $S$, and that this is also solid. Now $\uHom_{E_\solid[G]}(D, V)$ is the kernel of the map
\[ \uHom_E(D, V) \to \prod_{(S, g)} \uHom_E(E[S], \uHom_E(D, V)) \]
given by $\phi \mapsto (g\phi - \phi g)$, where the product is over all ($\kappa$-small) profinite sets $S$ and $g \in G(S)$. Products and kernels of solid vector spaces are solid (cf. \cref{solid}), so we conclude that $\uHom_{E_\solid[G]}(D, V)$ is in fact solid.  
\end{proof}

\begin{para}
In some cases, the functor of \cref{representable-para} coincides with functors discussed previously in \cref{an-vectors-classical,emerton-style-adaptations}. Here are some results along these lines. 
\end{para}

\begin{theorem}\label{Emerton-and-representable-style}
\begin{enumerate}
\item Suppose $\bbG$ is an affinoid group and $V$ is a solid representation of $G = \bbG(F)$. Then 
\[ V^{\bbG\anv} = \uHom_{E_\solid[G]}(D(\bbG), V). \]
\item Suppose $\bbG^\circ$ is a strictly ind-affinoid group and $V$ is a solid representation of $G^\circ = \bbG^\circ(F)$. Then 
\[ V^{\bbG^\circ\anv} = \uHom_{E_\solid[\Go]}(D(\bbG^\circ), V). \]
\end{enumerate}
\end{theorem}

\begin{proof}
Both statements follow from \cite[theorem 4.36 part (1)]{RR}.
\end{proof}

\begin{proposition}\label{Dr-and-HomDr-for-compact-ind-limits}
Suppose $G$ is a compact locally $F$-analytic group and $V$ is a solid representation of $G$ whose underlying solid vector space is LS \cite[definition 3.22(2)]{RR} (eg, if $V$ is an admissible locally analytic representation). Then, for any $r \in [1/p, 1)$, we have
\[ V^{D_r(G)} = \uHom_{E_\solid[G]}(D_r(G), V). \]
\end{proposition}

\begin{proof}
Observe that
\[ \uHom_E(D_r(G), V) = D_r(G)^\vee \sotimes_E V = \uHom_E(E, D_r(G)^\vee \sotimes_E V) \]
since $D_r(G)$ is Banach \cite[theorem 3.40(2)]{RR}, and taking $G$-linear maps on both sides gives
\[ \uHom_{E_\solid[G]}(D_r(G), V) = \uHom_{E_\solid[G]}(E, D_r(G)^\vee \sotimes_E V) = V^{D_r(G)}. \qedhere \]
\end{proof}

\begin{para}\label{derived-analytic-vectors}{\it Derived $F$-analytic vectors.} In \cref{mixing} we will also consider derived $F$-analytic vectors. These are only defined for solid $G$-representations which are equipped the structure of a $D(G)$-module, cf. \cite[3.1.4]{RRII}. Let $G$ be an $F$-uniform $F$-analytic groups and $\bbG$ its associated rigid analytic group over $F$. Given a solid $D(G)$-module $V$, and $\rho \in (0,1] \cap p^\Q$ one defines 
\[V^{R\bbG[\rho]\anv} = R\uHom_{D(G)}\Big(E,\Big(\cO(G.\bbG[\rho]) \otimes^\bbL_{E_\solid} V\Big)_{\star_{1,3}}\Big) \,,\]
and
\[V^{R\bbG[\rho]^\o\anv} = R\varprojlim_{\rho' < \rho} V^{R\bbG[\rho']\anv} = R\uHom_{D(G)}\Big(E,\cO(G.\bbG[\rho]^\o,V)_{\star_{1,3}}\Big) \,.\]
These are objects in the $\infty$-category $\Mod_{E_\solid}(D(G))$, and their cohomology in degree zero coincides with the analytic vectors we defined in \cref{an-vectors-classical}, cf. the discussion in \cite[3.1.6]{RRII}. By \cite[3.1.8]{RRII} we have 
\[V^{R\bbG[\rho]\anv} = R\uHom_{D(G)}\Big(D(\bbG[\rho],G),V) \,,\]
and
\[V^{R\bbG[\rho]^\o\anv} = R\uHom_{D(G)}\Big(D(\bbG[\rho]^\o,G),V\Big) \,.\]
\end{para}

\begin{proposition}\label{higher-derived-an-vectors} Suppose $G$ is an $F$-uniform locally $F$-analytic group, and let $W$ be an admissible locally $F$-analytic representation of $G$. Then the canonical map 
\[W^{R\bbG[\rho]^\o\anv} \ra W^{\bbG[\rho]^\o\anv}\] 
is a quasi-isomorphism. In other words, all higher derived $\bbG[\rho]^\o$-analytic vectors of $W$ vanish.
\end{proposition}

\begin{proof} This statement is analogous to the corresponding statement for admissible Banach space representations in \cite[4.48]{RR}. Set $M := W^\vee$. If $W$ were a Banach space then $M$ would be a Smith space, and one would have $W = R\uHom_{E_\solid}(M,E_\solid)$, by \cite[3.8, 3.10]{RR}. This identity is crucially used in the proof of \cite[4.48]{RR} for admissible Banach space representations. In our situation, however, $W$ is a space of compact type, in which case we do not know that $W \stackrel{?}{=} R\uHom_{E_\solid}(M,E_\solid)$.\footnote{This identity would follow from \cite[Conj. 3.41, 3.42]{RR}. Yet this conjecture seems to be intractable, being of similar type as questions about derived duals of solid Banach spaces.} Instead, we first show that the complex of derived locally analytic vectors $R\uHom_{E_\solid}(M,E_\solid)^{R\la}$ is canonically isomorphic to $W$.\footnote{We thank Juan Esteban Rodr\'iguez Camargo for suggesting this proof which is significantly simpler than our previous proof.} In the following $\rho'$ runs through a set of numbers in $(0,1) \cap p^\Q$ tending to zero. Then we have

\begin{numequation}\label{step1}
\begin{array}{rcl}
R\uHom_{E_\solid}(M,E_\solid)^{R\la} & = & \lim_{\rho' \ra 0} R\uHom_{D(G)}(D(\bbG[\rho']^\o,G),R\uHom_{E_\solid}(M,E_\solid)) \\
&&\\
& = & \lim_{\rho' \ra 0} R\uHom_{D(G)}\Big(D(\bbG[\rho']^\o,G)\otimes^\bbL_{E_\solid} M,E_\solid\Big)\\
&&\\
& = & \lim_{\rho' \ra 0} R\uHom_{E_\solid}\Big(D(\bbG[\rho']^\o,G)\sotimes^\bbL_{D(G)} M,E_\solid\Big) \\
&&\\
& = & \lim_{\rho' \ra 0} R\uHom_{E_\solid}\Big(D(\bbG[\rho']^\o,G)\sotimes_{D(G)} M,E_\solid\Big) \\ 
&&\\
& = & \lim_{\rho' \ra 0} \uHom_{E_\solid}\Big(D(\bbG[\rho']^\o,G)\sotimes_{D(G)} M,E_\solid\Big) \\
&&\\
& = & \lim_{\rho' \ra 0} W^{\bbG[\rho']^\o\anv} = W \,,
\end{array}    
\end{numequation}
where the first equality holds by \cite[3.10]{RRII}. The fourth equality is true because $D(\bbG[\rho']^\o,G)$ is a flat over $D(G)$, cf. \cite[4.4]{Lahiri}, \cite[A.11]{EmertonJII}. Again, by \cite[A.11]{EmertonJII} the space $M_{\rho'} := D(\bbG[\rho']^\o,G)\sotimes_{D(G)} M$ is of compact type, and is thus a countable inductive limit of Smith spaces with transition maps which are injective and of trace class, cf. \cite[Def. 3.34, Cor. 3.38]{RR}. Its dual $M^\vee_{\rho'}$ is thus a projective limit of Banach spaces with transition maps that have dense image. By the topological Mittag-Leffler theorem \cite[3.27]{RR} we find that $R\uHom_{E_\solid}(M_{\rho'},E) = R\uHom_{E_\solid}((M_{\rho'}^\vee)^\vee,E) = M_{\rho'}^\vee = W^{\bbG[\rho']\anv}$, where the last equality holds by \cite[5.2]{Lahiri}, \cite[A.1]{EmertonJII}. This justifies the fifth an sixth equality. Now we can use \ref{step1} to compute the derived $\bbG[\rho]^\o$-analytic vectors as follows:

\[\begin{array}{rcl}
W^{R\bbG[\rho]^\o\anv} & = & R\uHom_{D(G)}(D(\bbG[\rho]^\o,G),W) \\
&&\\
& \stackrel{\ref{step1}}{=} & R\uHom_{D(G)}\Big(D(\bbG[\rho]^\o,G),R\uHom_{E_\solid}(M,E_\solid)^{R\la}\Big)\\
&&\\
& = & R\uHom_{D(G)}\Big(D(\bbG[\rho]^\o,G),R\uHom_{E_\solid}(M,E_\solid)\Big)\\
&&\\
& = & R\uHom_{D(G)}\Big(D(\bbG[\rho]^\o,G)\sotimes^\bbL_E M,E_\solid\Big) \\ 
&&\\
& = & R\uHom_\Es\Big(D(\bbG[\rho]^\o,G)\sotimes^\bbL_{D(G)} M,E_\solid\Big) \\ 
&&\\
& = & R\uHom_\Es\Big(D(\bbG[\rho]^\o,G)\sotimes_{D(G)} M,E_\solid\Big) \\ 
&&\\
& = & \uHom_\Es\Big(D(\bbG[\rho]^\o,G)\sotimes_{D(G)} M,E_\solid\Big) = W^{\bbG[\rho]\anv}\,,\\ 
\end{array}\]  
where the first equality holds by \cite[3.1.8]{RRII}, the third holds by \cite[3.2.7]{RRII}, and the remaining equality signs are justified as above.
\end{proof}

%% file: 4-schneiderstuhler.tex
\section{Analytic Schneider-Stuhler complexes}

\begin{para} \label{notation-reductive}
Let $\bG$ be a connected reductive group over $F$ and set $G = \bG(F)$. Denote by $\BT = \BT(\bG/F)$ the {\it semisimple} Bruhat-Tits building of $\bG$ over $F$. Let $\ell$ be the semisimple split rank of $\bG$. For any integer $q \in [0, \ell]$, write $\BT_q$ for the set of $q$-dimensional facets, and $\BT_{(q)}$ for the oriented $q$-dimensional facets. See \cite[sections I.1, II.1]{ScSt97} for an overview. 
\end{para}

\subsection{The groups \texorpdfstring{$U^{(e)}_\sigma$}{U-sigma-(e)}} \label{Uesigma}

\begin{para}
To a facet $\sigma$ of $\BT$ and a non-negative integer $e$ called the {\it level}, Schneider and Stuhler define a compact open subgroup $U_\sigma^{(e)} \subseteq G$ \cite[p. 114]{ScSt97}.
\end{para}

\begin{para}
The goal of this section is the following result \cref{ss-groups-uniform}, which, in combination with the discussion of \cref{F-uniform}, implies that one can canonically associate a variety of analytic groups to $U_\sigma^{(e)}$ when the level $e$ is sufficiently large. This result generalizes the observation of \cite[remark 4.3.4]{PSS1}, which covers the case when $F = \Q_p$ and $G$ is split.
\end{para}

\begin{proposition} \label{ss-groups-uniform}
$U_\sigma^{(e)}$ is $F$-uniform for $e$ sufficiently large. 
\end{proposition}

\begin{proof}
We use a connection with work of Moy and Prasad \cite{MoyPrasad}, which introduces compact open subgroups $P_{x,r}$ for any point $x \in \BT$ and any real number $r \geq 0$ \cite[section 3.8]{MoyPrasad}.\footnote{We follow Vign\'eras in denoting these groups by $P_{x,r}$, instead of $\sP_{x,r}$ as in \cite{MoyPrasad} or $G_{x,r}$ as in \cite{AdlerKorman}.} Vign\'eras shows that 
\begin{equation}\label{SS=MP}
U_\sigma^{(e)} = P_{x,e+} := \bigcup_{r>e} P_{x,r}
\end{equation}
for every $x \in \sigma$ \cite[proposition 1.1]{VignerasSheaves}. Moy and Prasad further define $\cO_F$-Lie algebras $\frg_{x,r}$ and $\frg_{x,r+} := \bigcup_{s>r} \frg_{x,s}$ which are lattices in $\frg = \Lie(G)$ \cite[end of section 3.2]{MoyPrasad}. Moreover, $\exp : \frg_{x,r} \to P_{x,r}$ is bijective for all $r \gg 0$ \cite[end of 8.5]{AdlerKorman}, so
\[ U_\sigma^{(e)} = P_{x,e+} = \bigcup_{r > e} P_{x,r} = \exp \left( \bigcup_{r > e} \frg_{x,r} \right) = \exp(\frg_{x,e+}) \]
for $e \gg 0$. By \cite[section 9.4]{DDMS}, it is therefore sufficient to show that $\frg_{x,e+}$ is powerful for $e \gg 0$. Recall that \begin{equation}\label{commutator}
[\frg_{x,r+},\frg_{x,s+}] \subseteq \frg_{x,(r+s)+}
\end{equation}
by \cite[proposition 1.4.2]{Adler98}, and that \begin{equation}\label{vpi-shift}
\frg_{x,(r+c)+} = \varpi_F \frg_{x,r+},
\end{equation}
where $c = [L:F^\nr]$ is the degree of a certain finite extension $L$ of the maximal unramified extension $F^\nr$ of $F$ \cite[section 3.1 and line 14 of page 399]{MoyPrasad}.\footnote{The extension $L$ is the minimal Galois extension which splits $\bG^{\rm nr} = \bG \times_F F^\nr$, except if $\bG^\nr$ is a triality form of ${}^6D_4$, in which case $L$ is a degree 3 subextension of the Galois splitting field of $\bG^\nr$ of degree 6. The degree $[L:F^\nr]$ is denoted by $\ell$ in \cite{MoyPrasad}.} 
Suppose $e \geq \varepsilon \cdot c \cdot e(F/\Q_p)$, where $e(F/\Q_p)$ is the ramification index and $\kappa$ is as in \cref{kappa}. Then 
\[ [\frg_{x,e+}, \frg_{x,e+}] \subseteq \frg_{x,(2e)+} \subseteq \frg_{x,(e + \varepsilon \cdot c \cdot e(F/\Q_p))+} \subseteq \varpi_F^{\varepsilon \cdot e(F/\Q_p)} \frg_{x,e+} = p^{\varepsilon} \frg_{x,e+}, \]
so $\frg_{x,e+}$ is indeed powerful. 
\end{proof}

\begin{remark} 
The proof above shows that $\frg_{x,e+}$ is powerful for any 
\[ e \geq \varepsilon \cdot c \cdot e(F/\Qp). \]
We expect that any such $e$ should also be large enough to guarantee that $\exp : \frg_{x,e+} \to P_{x,e+}$ is bijective (and thus that $U_\sigma^{(e)}$ is $F$-uniform).
\end{remark}

\subsection{Analytic Schneider-Stuhler complexes} 

\begin{para}
The aim in this section is to construct functors from solid representations to augmented complexes, following the construction of \cite[section II.2]{ScSt97}\footnotemark{} but replacing the invariants that appear there with more general coefficient systems. The type of coefficient systems that we have in mind, and which will be investigated later on, are given by taking analytic vectors of various kinds. This extends the discussion of \cite{LahiriResolutions}.
\footnotetext{The construction of \cite[section II.2]{ScSt97} in the case of the trivial representation also appears in  \cite[proposition 46]{BernsteinLectures} and \cite{PrasadRaghunathan}.}
\end{para}

\begin{para} \label{functors}
Let $A = (A_\sigma)_{\sigma \in \BT}$ be a family of functors from the category of solid representations of $G$ to the category of solid vector spaces, equipped with natural morphisms $A_\tau(V) \to A_\sigma(V)$ whenever $\sigma$ is a face of $\tau$ in $\BT$, and with natural morphisms $A_\sigma(V) \to V$ whenever $\sigma$ is a vertex of $\BT$.
\end{para}

\begin{remark}
Stated differently, for any solid representation $V$ of $G$, the family $(A_\sigma(V))_{\sigma \in \BT}$ is a coefficient system of solid vector spaces on $\BT$ equipped with an augmentation into $V$. 
\end{remark}

\begin{examples}\label{examples-functors-A}
Fix a level $e \gg 0$ such that $U_\sigma^{(e)}$ is $F$-uniform for all facets $\sigma$ of $\BT$. We write simply $U_\sigma$ in place of $U_\sigma^{(e)}$. Fix $\rho \in (0, \rho^*) \cap p^\Q$ and $* \in \{\emptyset, \circ, \dagger\}$, so that we have the analytic groups $\bbU_\sigma[\rho]^*$ as in \cref{F-uniform}. Also, fix $r \in \sR$, where $\sR$ has been introduced in \ref{the-set-sR}. All of the following constructions, discussed above in \cref{analytic-vectors}, are examples of constructions that give the data of \cref{functors}.
\begin{enumerate}[(a)]
\item $A_\sigma(V) = V^{U_\sigma}$.
\item $A_\sigma(V) = V^{\bbU_\sigma[\rho]^*\anv}$. 
\item $A_\sigma(V) = V^{D(\bbU_\sigma[\rho]^*)}$.
\item $A_\sigma(V) = V^{D_r(U_\sigma)}$.

\item $A_\sigma(V) = \uHom_{U_\sigma}(D_\sigma, V)$, for $D_\sigma = D(\bbU_\sigma[\rho]^*)$ or $D_{r}(U_\sigma)$. 
\end{enumerate}
By various results in section \cref{analytic-vectors}, some of these constructions coincide for all solid representations of $G$, or only for certain of those representations.
\end{examples}

\begin{definition} \label{analytic-schneider-stuhler}
With $A$ as in \cref{functors}, we associate an augmented complex
\[ \cS^A_\bullet(V) \to V \]
to any solid representation $V$ of $G$ as follows. Let $\cS^A_\bullet(V)$ be the complex which in homological degree $q$ is given by
\[ \cS^A_q(V) = \left\{ (v_{(\sigma, c)})_{(\sigma, c) \in \BT_{(q)}} \in \bigoplus_{(\sigma, c) \in \BT_{(q)}} A_\sigma(V) : v_{(\sigma, -c)} = -v_{(\sigma, c)} \text{ if } q \geq 1 \right \}\footnote{This notation is abusive in that it abbreviates an equalizer of solid vector spaces.}  \]
for all $0 \leq q \leq \ell$ and $\cS^A_q(V) = 0$ otherwise. The differentials $d : \cS^A_{q+1}(V) \to \cS^A_q(V)$ are given by
\[ d(v) = \left( \sum_{\substack{(\sigma', c') \in \BT_{(q+1)} \\ \sigma' \; \subseteq \; \overline{\sigma} \\ \partial_{\sigma'}^\sigma(c) = c'}} v_{(\sigma', c')} \right)_{(\sigma, c) \in \BT_{(q)}}. \]
The natural maps $A_\sigma(V) \to V$ for $\sigma \in \BT_{(0)} = \BT_0$ induce the augmentation $\cS^A_0(V) \to V$ which is $G$-equivariant for the action of $G$ on $\cS^A_0(V)$ as described below. 
\end{definition}

\begin{para}{\it The group action on $\cS^A_q(V)$.} Note that for $g \in G$ and $v_{(\sigma,c)} \in A_\sigma(V)$ one has $g.v_{(\sigma,c)} \in A_{g.\sigma}$. Therefore, $G$ acts as follows on $\cS^A_q(V)$: for $v=(v_{(\sigma, c)})_{(\sigma, c) \in \BT_{(q)}} \in \cS^A_q(V)$ and $g \in G$ we set $g.v = (g.v_{(g^{-1}.\sigma,g^{-1}.c)})_{(\sigma,c)\in \BT_{(q)}}$. We want to describe this representation more concretely. To this end, when $q \ge 1$, choose for every $\sigma \in \BT_q$ an orientation $c_\sigma$, and denote the opposite orientation by $c'_\sigma$. Then the map
\begin{numequation}\label{c-ind-description-1}
\bigoplus_{\sigma \in \BT_q} A_\sigma(V) \lra \cS^A_q(V)
\end{numequation}
which sends $(v_\sigma)_{\sigma \in \BT_q}$ to $(v_{(\sigma,c)})_{(\sigma,c) \in \BT_{(q)}}$, where 
\[v_{(\sigma,c)} = \left\{\begin{array}{lcl} v_{\sigma,c_\sigma} & , & c = c_\sigma \\ -v_{\sigma,c_\sigma} & , & c = c'_\sigma \end{array} \right. \;,\]
is an isomorphism of solid $E$-vector spaces. Denote by $\Psd$ the stabilizer of the facet $\sigma$ (this group may not stabilize $\sigma$ pointwise). We define a character
\[\chi_\sigma: \Psd \lra \{1,-1\} \;, \;\;  h \mapsto \left\{\begin{array}{lcl} 1 & , & h.c_\sigma = c_\sigma \\ -1 & , & h.c_\sigma = c'_\sigma \end{array} \right. \;.\]
We can then define an action by $\Psd$ on $A_\sigma(V)$ by $h.v = \chi_\sigma(h) \cdot h(v)$, where $h(v)$ refers to the action of $h$ on $v$ induced by the action of $G$ on $V$. The map \ref{c-ind-description-1} becomes then an isomorphism of $G$-representations  
\begin{numequation}\label{c-ind-description-2}
\begin{tikzcd}
\bigoplus_{\sigma \in G\bksl \BT_q} \cind^G_\Psd A_\sigma(V) \ar[r,"\simeq"] & \cS^A_q(V) \;.    
\end{tikzcd}
\end{numequation}
Is it this description which we will often use in the following.  
\end{para}

\begin{question}
Fix a choice of $A = (A_\sigma)_{\sigma \in \BT}$ as in \cref{functors}. For which solid representations $V$ of $G$ is it the case that $\cS^A(V) \to V$ is a quasi-isomorphism? 
\end{question}

\begin{para}
We summarize what we know about the above question.
\begin{enumerate}[(a)]
\item If $A_\sigma(V) = V^{U_\sigma}$ is the $U_\sigma$-invariants of $V$, one of the main results of \cite{ScSt97} is precisely that $\cS^A(V) \to V$ is a quasi-isomorphism for smooth $V$. 
\item It is shown in \cite[theorem 1.1.2]{LahiriResolutions} that, if $V$ is a locally analytic principal series representation of $G = \GL_2(F)$, and $A_\sigma(V) = V^{\bbU_\sigma[1]\anv}$, then $\cS^A(V) \to V$ is a quasi-isomorphism for all sufficiently small $\rho$ depending on $V$. 
\item In \cref{universal} below, we show that $\cS^A(V) \to V$ is a quasi-isomorphism for $V = V^\univ$ a certain ``universal representation'' and $A_\sigma(V) = \uHom_{U_\sigma}(D_r(U_\sigma),V)$.
\end{enumerate}
\end{para}

\begin{remarks}\label{remarks-exactness-of-analytic-ScSt} (1) Denote by $\Sigma_0 = \bbP^{1,\rig}_F \setminus \bbP^1(F)$ the $p$-adic upper half plane over $F$, and let $\pr_m: \Sigma_m \ra \Sigma_0$ be the $m^{\rm th}$ Drinfeld covering. Denote by $\cO(\Sigma_m)^\chi$ the $\chi$-isotypic component for an irreducible representation $\chi$ of the finite covering group of $\Sigma_m/\Sigma_0$. Then $V = \left[(\pr_m)_* \cO(\Sigma_m)^\chi\right]'$ is a locally analytic representation of $G = \GL_2(F)$. We then expect that the analytic Schneider-Stuhler complex of $V$ for the coefficient system $A_\sigma(V) = V^{\bbU_\sigma[\rho]^\circ\anv}$ can be related to the \v{C}ech complex of the sheaf $(\pr_m)_* \cO(\Sigma_m)^\chi$ for the covering of $\Sigma_0$ given by the Bruhat-Tits tree. The radius $\rho$ of the groups $\bbU_\sigma[\rho]$ for which one obtains a resolution will depend on the representation $\chi$. 
\vskip8pt
(2) One may expect that certain features of the theory of the analytic Schneider-Stuhler complex, as defined above, resemble those that arise in the study of the Schneider-Stuhler complex for smooth representations on vector spaces over fields of characteristic $p$. In this setting it is known that the Schneider-Stuhler complex is not always a resolution \cite[remark 3.2, item (3)]{OllivierSchneiderGorenstein}. On the other hand, it is known that the Schneider-Stuhler complex for mod $p$ representations is a resolution for the universal representation, and for principal series representations of $\GL_n(F)$ \cite[Introduction]{OllivierResolutions}.
\end{remarks}

\begin{remark}
One can also form a ``derived variant'' of the analytic Schneider-Stuhler complex, roughly by replacing analytic vectors with derived analytic vectors. Exploration of this direction of variation is work in progress.  
\end{remark}

\input{4a-universal}

%% file: 4a-universal.tex
\subsection{Universal representation} \label{universal}

\begin{para}  \label{universal-convention}
Fix $r \in \sR$, where $\sR$ has been defined in \ref{the-set-sR}. In this section, we consider the family $A = (A_\sigma)_{\sigma \in \BT}$ of functors as in \cref{functors} where 
\[ A_\sigma(V) = \uHom_{U_\sigma}(D_r(U_\sigma), V). \]
\end{para}

\begin{para} \label{distributions-coset}
Let $U$ be an open compact $F$-uniform subgroup of $G$. For a (left) coset $s = gU \in G/U$, define
\[ D_r(s) = D_r(gU) := \delta_g D_r(U) \]
regarded as a subspace of $D_r(G)$.
\end{para}

\begin{para} \label{free-cocompletion-category}
Regard the set $G/U$ of cosets of $U$ in $G$ as a discrete category (i.e., a category whose objects are the elements of $G/U$ and with no non-identity morphisms), and let $\sC$ be the free cocompletion of this category. Explicitly, $\sC$ is a category whose objects are families $\bbT = (T_s)_{s \in G/U}$ of sets (i.e., where $T_s$ is a small set for any $s \in G/U$), and whose morphisms are families of set maps indexed by $s \in G/U$. The universal property of free cocompletion implies that the $s \mapsto D_r(s)$ construction of \cref{distributions-coset}, regarded as a functor on $G/U$, extends uniquely to $\sC$ cocontinuously. Explicitly, for $\bbT = (T_s)_{s \in G/U} \in \sC$, we have 
\[ D_r(\bbT) = \bigoplus_{s \in G/U} \bigoplus_{t \in T_s} D_r(s). \]
\end{para}

\begin{para}
Let $\bbS = (S_s)_{s \in G/U} \in \sC$ be the object where $S_s = *$ (i.e., $S_s$ is a singleton set) for all $s \in G/U$. The {\it universal representation} of interest in this section is
\[ V^\univ = D_r(\bbS) = \bigoplus_{s \in G/U} D_r(s). \]
This representation is universal in the sense that it (internally) represents the functor $W \mapsto \uHom_U(D_r(U), W)$ on the category of solid representations of $G$. 
\end{para}

\begin{para} \label{generated-by-analytic-vectors}
Increasing the level $e$ further if necessary, we can assume that there exists a vertex $\sigma \in \BT_0$ such that $U_\sigma \subseteq U$. This implies that the image of $A_\sigma(V^\univ)$ in $V^\univ$ generates $V^\univ$ as a solid $G$-representation. Indeed, observe that clearly $D_r(U)$ generates 
\[ V^\univ = \bigoplus_{gU \in G/U} \delta_g D_r(U), \]
so it is sufficient to show that the inclusion of $D_r(U)$ into $V^\univ$ factors through $A_\sigma(V^\univ) \to V^\univ$. By \cref{Dr-functoriality} there is a morphism of solid $E$-algebras $D_r(U_\sigma) \to D_r(U)$. This gives rise to the map
\[ D_r(U) \to \uHom_{U_\sigma}(D_r(U_\sigma), D_r(U)) \,,\; \mu \mapsto [\delta \mapsto \delta \mu] \,.\]
Composing this map with
\[\begin{tikzcd} \uHom_{U_\sigma}(D_r(U_\sigma), D_r(U)) \ar{r} & \uHom_{U_\sigma}(D_r(U_\sigma), V^\univ) = A_\sigma(V^\univ) \ar{r} & V^\univ, \end{tikzcd} \]
where the first map is functoriality of $\uHom_{U_\sigma}(D_r(U_\sigma), -)$ and second map is evaluation at 1, is the natural inclusion of $D_r(U)$ as a direct summand of $V^\univ$. Thus $D_r(U)$ factors through $A_\sigma(V^\univ) \to V^\univ$, as claimed. 
\end{para}

\begin{theorem} \label{universal-rep}
$\cS^A_\bullet(V^\univ) \to V^\univ$ is a quasi-isomorphism.
\end{theorem}

\begin{para}
The remainder of this subsection is a proof of \cref{universal-rep}. To begin, for $\sigma \in \BT$, let $\bbS_\sigma \subseteq \bbS$ be the subobject of cosets $gU \in G/U$ such that $U_\sigma g U \subseteq gU$ and let ${}_\sigma \bbS$ be its complement.\footnote{More precisely, $\bbS_\sigma = (S_{\sigma, s})_{s \in G/U}$ has $S_{\sigma,s} = *$ if and only if $s = gU$ where $U_\sigma gU \subseteq gU$, and $S_{\sigma,s} = \emptyset$ otherwise. For ${}_\sigma \bbS = ({}_\sigma S_s)_{s \in G/U}$, the conditions are reversed.}
\end{para}

\begin{proposition}\label{analytic-vectors-computation}
We have, for every facet $\sigma$ of $\BT$,
\[ \uHom_{U_\sigma}(D_r(U_\sigma), V^\univ) =  \uHom_{U_\sigma}(D_r(U_\sigma), D(\bbS)) = D_r(\bbS_\sigma)\,. \]
\end{proposition}

\begin{proof} Denote by $\sS \sub G$ a system of representatives for $\bbS = G/U$, so that $V^\univ = \bigoplus_{g \in \sS} \delta_g D_r(U)$. Given a double coset $\overline{h} = U_\sigma h U$ with $h \in G$, we set 
\[V^\univ_\ovh := \bigoplus_{g \in \overline{h}/U} \delta_g D_r(U) = \bigoplus_{g \in \overline{h} \cap \sS} \delta_g D_r(U)\,.\] 
This is a Banach subspace which is stable by the action of $U_\sigma$, and $V^\univ = \bigoplus_{\overline{h} \in U_\sigma \bksl G/U} V^\univ_\ovh$. \vskip8pt

By \cite[3.32]{RR} we have 
\[\uHom_{U_\sigma}(D_r(U_\sigma), V^\univ) =  \bigoplus_{\overline{h} \in U_\sigma \bksl G/U} \uHom_{U_\sigma}\Big(D_r(U_\sigma), V^\univ_\ovh\Big) \,.\]

As $D_r(U_\sigma)$ and $V^\univ_\ovh$ are both Banach spaces (considered here in the first place as solid vector spaces), and because the category of Fr\'echet spaces over $E$ embeds fully faithfully into $\Vec^\solid_E$ \cite[Rem. A.32]{BoscoDrinfeldv2}, we will from now on consider $D_r(U_\sigma)$ and $V^\univ_\ovh$ as ordinary Banach spaces, and compute $\Hom^{\rm cont}_{U_\sigma}\Big(D_r(U_\sigma), V^\univ_\ovh\Big)$. Because $E[U_\sigma]$ is dense in $D_r(U_\sigma)$, an element $\lambda \in \Hom^{\rm cont}_{U_\sigma}\Big(D_r(U_\sigma), V^\univ_\ovh\Big)$ is determined by $\lambda(1) \in V^\univ_\ovh$.

\vskip8pt

(i) After these preliminaries, we show first that $D_r(\bbS_\sigma) \subset \uHom_{U_\sigma}(D_r(U_\sigma),  V^\univ)$. If $g^{-1}U_\sigma g \sub U$, then the multiplication map $E[U_\sigma] \times \delta_g D_r(U) \ra \delta_g D_r(U)$, $(\delta_h,\delta_g \mu) \mapsto \delta_h \delta_g \mu  = \delta_g \delta_{g^{-1}hg} \mu$, is well-defined and by \cref{Dr-functoriality} extends continuously to $D_r(U_\sigma) \times \delta_g D_r(U) \ra \delta_g D_r(U)$ by the formula $(\delta,\delta_g \mu) \mapsto \delta_g {\rm Ad}(g^{-1})(\delta)\mu$. Any $\delta_g \mu \in \delta_g D_r(U)$ gives thus rise to a continuous $U_\sigma$-linear map $\lambda: D_r(U_\sigma) \ra \delta_g D_r(U)$  by $\lambda(\delta) = \delta_g {\rm Ad}(g^{-1})(\delta)\mu$. This proves $D_r(\bbS_\sigma) \subset \uHom_{U_\sigma}(D_r(U_\sigma),  V^\univ)$.

\vskip8pt

(ii) Now we show that $\uHom_{U_\sigma}(D_r(U_\sigma),  V^\univ) \subset D_r(\bbS_\sigma)$. We will show that 
\[\uHom_{U_\sigma}(D_r(U_\sigma),  V^\univ_\ovh) = 0\,,\] 
unless $\ovh = U_\sigma hU = hU$ in which case $\uHom_{U_\sigma}(D_r(U_\sigma),  V^\univ_\ovh) = V^\univ_\ovh = \delta_h D_r(U)$, as we have seen in (i).

\vskip8pt

Consider an element $\lambda \in \uHom_{U_\sigma}(D_r(U_\sigma),  V^\univ_\ovh)$ and set $\mu := \lambda(1) = \sum_{g_1 \in \overline{h} \cap \sS} \delta_{g_1} \mu_{g_1} \in V^\univ_\ovh$ with $\mu_{g_1} \in D_r(U)$ for all $g_1 \in ovh \cap \sS$.  Fix $g \in \overline{h} \cap \sS$. \vskip8pt

{\bf Claim.} If $\mu_g \neq 0$ then $g^{-1}U_\sigma g \subset U$. 
\vskip8pt

{\it Proof of the claim.} Composing $\lambda$ with the continuous projection ${\rm pr}_g: V_\ovh \ra \delta_g D_r(U)$, and then with the isometry $\iota_1: \delta_g D_r(U) \ra D_r(U), \lambda \mapsto \delta_{g^{-1}}\lambda$, gives a continuous map 
\[\phi = \iota_1 \circ {\rm pr}_g \circ \lambda: D_r(U_\sigma) \lra D_r(U) \,.\]
For $u \in U_\sigma \cap gUg^{-1}$ we then have $\lambda(u) = u\lambda(1) = u\mu = \sum_{g_1 \in \overline{h} \cap \sS} \delta_{g_1} {\rm Ad}(g_1^{-1})(u)\mu_{g_1}$. Because $g_1 = u_1g$ for $u_1 \in U_\sigma$ we have $g_1^{-1}ug_1 \in U$ since $u \in U_\sigma \cap gUg^{-1} = U_\sigma \cap g_1Ug_1^{-1}$. And thus ${\rm Ad}(g_1^{-1})(u)\mu_{g_1} \in D_r(U)$. If we then apply $\iota_1 \circ {\rm pr}_g$, we get 
$$\phi(u) = \iota_1\left({\rm pr}_g\left(\sum_{g_1 \in \overline{h} \cap \sS} \delta_{g_1} {\rm Ad}(g_1^{-1})(u)\mu_{g_1}\right)\right) = \iota_1(\delta_g {\rm Ad}(g^{-1})(u)\mu_g) = {\rm Ad}(g^{-1})(u)\mu_g \,.$$
If then $\frx \in \frg$ one has 
$$\phi(\frx) = \phi\left(\lim_{t\ra 0} \frac{1}{t}(\exp(t\frx)-1)\right) = \lim_{t\ra 0} \frac{1}{t}(\exp(t{\rm Ad}(g^{-1})(\frx)\mu_g-\mu_g) = {\rm Ad}(g^{-1})(\frx)\mu_g$$
It follows that for all $w \in U(\frg)_E$ one has similarly $\phi(w) = {\rm Ad}(g^{-1})(w)\mu_g$. As $\phi$ is a continuous map of Banach spaces, there is $C>0$ such that for all $\delta \in D_r(U_\sigma)$ one has $\|\phi(\delta)\|_{D_r(U)} \le C \cdot \|\delta\|_{D_r(U_\sigma)}$. In particular, if $(w_n)_n$ is a Cauchy sequence in $U(\frg)_E$ for the $\|\cdot\|_r$-norm on $D_r(U_\sigma)$, then $(\phi(w_n))_n$ is a Cauchy sequence in $D_r(U)$, and all of its terms lie in the principal left ideal $D_r(U)\mu_g$. As a finitely generated submodule of a coadmissible $D_r(U)$-module, namely $D_r(U)$ itself, $D_r(U)\mu_g$ is itself coadmissible \cite[3.4 (iv)]{ST03} and hence closed \cite[3.6]{ST03}. The sequence thus has a limit in $D_r(U)\mu_g$. Denote by $U_r(\frg,U_\sigma)_E$ the closure of $U(\frg)_E$ in $D_r(U_\sigma)$. It follows that $\phi(U_r(\frg,U_\sigma)_E) \subset D_r(U)\mu_g$. \vskip8pt

We now assume that $\mu_g$ is non-zero. By \cite[5.2.1 and its proof]{OrlikStrauchIRR}, $D_r(U)$ is an integral domain, and the map $D_r(U) \ra D_r(U)\mu_g$, $\delta \mapsto \delta\mu_g$, is thus injective. As a continuous bijective map of Banach spaces it is thus an isomorphism of topological vector spaces. Let $\iota_2: D_r(U)\mu_g \ra D_r(U)$ be the inverse map and set
\[\psi = \iota_2 \circ \phi|_{U_r(\frg,U_\sigma)_E}: U_r(\frg,U_\sigma)_E \lra D_r(U) \,.\]
By what we have observed earlier, we have $\psi(w) = {\rm Ad}(g^{-1})(w) \in U(\frg)_E$ for all $w \in U(\frg)_E$, and thus $\im(\psi) \subset U_r(\frg,U)_E$. Moreover, if $\frx \in \Lie_\Zp(U_\sigma)$ is such that $\exp(\frx) = \sum_{k \ge 0} \frac{1}{k!}\frx^k$ converges in $D_r(U_\sigma)$, hence in $U_r(\frg,U_\sigma)_E$, then $\psi(\exp(\frx))$ will be equal to $\sum_{
k \ge 0} \frac{1}{k!}{\rm Ad}(g^{-1})(\frx)^k$, and this series will converge in $U_r(\frg,U)_E$, and it is equal to ${\rm Ad}(g^{-1})(\exp(\frx))$, which is then an element of $g^{-1}U_\sigma g$. By the lemma below we have $U \cap U_r(\frg,U)_E = \exp(\varpi^\ell \Lie_\Zp(U))$, with $\ell$ as in \ref{closure-of-U}. The analogous statement we have for $U_\sigma$. Using that $\exp: \Lie_\Zp(U_\sigma) \ra U_\sigma$ is bijective (and similarly for $U$), we thus see that $\varpi^\ell{\rm Ad}(g^{-1})(\Lie_\Zp(U_\sigma)) \subset \varpi^\ell \Lie_\Zp(U)$. This in turn implies that ${\rm Ad}(g^{-1})(\Lie_\Zp(U_\sigma)) \subset \Lie_\Zp(U)$, which gives $g^{-1}U_\sigma g \subset U$. This completes the proof of the claim. \qed

\vskip8pt

We now finish the proof of the proposition. If $g^{-1}U_\sigma g$ is contained in $U$, then the inclusion of group algebras $E[g^{-1}U_\sigma g] \subset E[U]$ extends continuously to an injective homomorphism of algebras $D_r(g^{-1}U_\sigma g) \subset D_r(U)$, by \ref{Dr-functoriality}. Then we see that for $\delta \in D_r(U_\sigma)$ and any $\mu_g \in D_r(U)$ one has $\delta \delta_g \mu_g = \delta_g \cdot (\delta_{g^{-1}} \delta \delta_g) \mu_g \in \delta_g D_r(U)$. This proves that all of $\delta_g D_r(U)$ is contained in $\uHom_{U_\sigma}(D_r(U_\sigma),V^\univ)$, if $g^{-1}U_\sigma g \subset U$. 
\end{proof}

\begin{lemma}\label{closure-of-U} Let $U$ be an $F$-uniform group and fix $r \in \sR$, cf. \ref{the-set-sR}. Set $\Lambda = \Lie_\Zp(U)$ (which is an $\cO_F$-module) and $\frg = \Lambda \ot_{\cO_F} F$. Denote by $U_r(\frg,U)$ the closure of $U(\frg)_E$ in $D_r(U,E)$. Then $U \cap U_r(\frg,U) = \exp_U(\vpi^\ell\Lambda)$ where $\ell$ only depends on $r$. More precisely, if we set 
\begin{numequation}\label{definition-h}
h = \min\{k \in \Z_{\ge 0} \;:\; r^\kappa < p^{-\frac{1}{(p-1)p^k}}\}\,,    
\end{numequation}
then $\ell = \min\{m \ge 0 \; : \: |\vpi|^m p^h r^{\kappa p^h} < p^{-\frac{1}{p-1}}\}$. 
\end{lemma}

\begin{proof} We follow the notation of \cite{SchmidtAUS}\footnote{With some exceptions: the group $G$ in \cite{SchmidtAUS} is our group $U$ and the group $G_0$ in \cite{SchmidtAUS} is denoted $\Res^{F}_\Qp U$ here.}, as we will make use of results of this paper. We note right away that if $r \in \sR$, then $r^\kappa$ is non-critical (i.e., for any $k \ge 0$ the number $r^{\kappa p^k}$ is not equal to
$p^{-\frac{1}{p-1}}$), and the number $m$ in \ref{the-set-sR} is equal to the number $h$ in \ref{definition-h}. 

\vskip8pt

Let $d = \dim_F(U)$ and $n = [F:\Qp]$. Given $\frx \in \Lambda \setminus \vpi\Lambda$, we can extend it to a $\cO_F$-basis $(\frx_j)_{j=1}^d$ of $\Lambda$ with $\frx_1 = \frx$. Let $(v_i)_{i=1}^n$ be a $\Zp$-basis of $\cO_F$ with $v_1 = 1$, and set $\frx_{i,j} = v_i\frx_j$, $h_{i,j} = \exp_U(\frx_{i,j})$, and $b_{i,j} = h_{i,j}-1$. We equip the locally $F$-analytic distribution algebra $D(U,E)$ with the filtration defined by the quotient norm $\|\cdot\|_\ovr$ (using the notation of \cite{SchmidtAUS}). Sending the principal symbol $\sigma(b_{i,j})$ in the graded ring ${\rm gr}^{\sbt}_r D_r(U,E)$ to a variable $X_{i,j}$ gives an isomorphism
$${\rm gr}^{\sbt}_r D_r(U,E) \xrightarrow{\;\simeq\;} ({\rm gr}^{\sbt}_r F)[X_{1,1}, \ldots, X_{n,d}]/(X_{i,j}^{p^h} - \ovv_i X_{1,j}^{p^h} \;:\; 1 \le j \le d, 2 \le i \le n)\,,$$
where $h$ is defined by \ref{definition-h}, cf. \cite[5.6, proof of 5.3]{SchmidtAUS}. Here $\ovv_i$ is the class of $v_i$ in the residue field $k_F$ of $F$. It follows from 
$${\rm gr}^{\sbt}_r D_r(U,E) \simeq {\rm gr}^{\sbt}_r D_r(\Res^{F}_\Qp U,E)/{\rm gr}^{\sbt}_r I_r(\Res^{F}_\Qp U,E) \,,$$
cf. \cite[p. 47]{SchmidtAUS}, where $\Res^{F}_\Qp U$ is denoted by $U_0$, that elements in degree $r^\kappa$ in ${\rm gr}^{\sbt}_r D_r(U,E)$ are mapped to homogeneous linear polynomials, as follows from the definition of the norm $\|\cdot\|_r$ on $D_r(\Res^{F}_\Qp U,E)$, cf. \cite[p. 40]{SchmidtAUS}. 
\vskip8pt
Let $\ovw_i \in k_F$ be such that $\ovw_i^{p^h} = \ovv_i$. Then the map $X_{i,j} \mapsto \ovw_i X_{1,j}$ induces a morphism of algebras over ${\rm gr}^{\sbt}_r F$
$$({\rm gr}^{\sbt}_r F)[X_{1,1}, \ldots, X_{n,d}]/(X_{i,j}^{p^h} - \ovv_i X_{1,j}^{p^h} \;:\; 1 \le j \le d, 2 \le i \le n)  \lra ({\rm gr}^{\sbt}_r F)[X_{1,1}, X_{1,2},\ldots, X_{1,d}] \,.$$
This shows that $\sigma(b_{1,j})$ is not nilpotent in ${\rm gr}^{\sbt}_r D_r(U,E)$, which implies that $b_{1,j}$ is `norm power multiplicative', by which we mean that for all $k \ge 0$ one has $\|b_{1,j}^k\|_\ovr = \|b_{1,j}\|_\ovr^k$. Moreover, since $\|b_{1,j}\|_r = r^\kappa$ in $D_r(\Res^{F}_\Qp U,E)$, and by the remark above, we also have for all $k \ge 0$ that $\|b_{1,j}^k\|_\ovr = r^{\kappa k}$ for the quotient norm $\|\cdot\|_\ovr$ on $D_r(U,E)$. 
\vskip8pt
Set $g = h_{1,1} = \exp_U(\frx)$ and $b = b_{1,1} = g-1$. Then $\frx = \log(1+b) = -\sum_{k >0} \frac{(-1)^k}{k}b^k$. As we assume that $r^\kappa$ is non-critical, we find that  
$$\|\frx\|_r = \max_{k>0}\left\{\left\|\frac{1}{k}b^k\right\|_\ovr\right\} = \max_{k>0}\left\{\frac{1}{|k|_p}r^{\kappa k}\right\} = \max_{k \ge 0}\left\{\frac{1}{|p^k|_p}r^{\kappa p^k}\right\} \,,$$ 
(the last equality is well-known and easy to prove). By the definiton of $h$ and since $r^\kappa$ is non-critical we have 
$$p^{-\frac{1}{(p-1)p^{h-1}}} < r^\kappa < p^{-\frac{1}{(p-1)p^h}}$$
and thus 
$$p^{h-1}r^{\kappa p^{h-1}} < p^h r^{\kappa p^k} \; \mbox{ and } \;  p^h r^{\kappa p^h} > p^{h+1}r^{\kappa p^{h+1}} \,.$$
Hence $\|\frx\|_r = p^h r^{\kappa p^h}$. We also note that $\frx$ is norm power multiplicative, because for the principal symbol $\sigma(\frx)$ we have $\sigma(\frx) = \sigma(p^h)\sigma(b)$, and because this element is not nilpotent in ${\rm gr}^{\sbt}_r D_r(U,E)$, it follows that $\frx$ is norm power multiplicative, so that for all $k \ge 0$ one has $\|\frx^k\|_\ovr = p^{kh} r^{\kappa kh}$. Therefore, $\exp_U(\vpi^m \frx) = \sum_{k \ge 0} \frac{\vpi^{m k}}{k!} \frx^k$ converges in $D_r(U,E)$ if and only if  
$\|\vpi^m \frx\|_\ovr = |\vpi|^m p^h r^{\kappa p^h} < p^{-\frac{1}{p-1}}$, and the smallest $m \ge 0$ with this property is precisely $\ell$, as defined above. 
\vskip8pt
Finally, given any non-trivial $g \in U$ we write $g = \exp_U(\fry)$ for a unique non-zero $\fry \in \Lambda$. Write $\fry = \vpi^a \frx$ with $\frx \in \Lambda \setminus \vpi \Lambda$ and $a \ge 0$. Then $g = \exp_U(\fry)$ is in $U_r(\frg,U)$ if and only if the exponential series $\exp(\fry) = \sum_{k \ge 0} \frac{1}{k!}\fry^k$ converges in $D_r(U,E)$. We have $\|\fry^k\|_\ovr = |\vpi|^{ak} \|\frx^k\|_\ovr$, which shows that $\fry$ is norm power multiplicative too. Hence the exponential series converges if and only if $\|\fry\|_\ovr = |\vpi|^a \|\frx\|_\ovr = |\vpi|^a p^h r^{\kappa p^h} < p^{-\frac{1}{p-1}}$, and this is the case if and only if $a \ge \ell$, and thus $\fry \in \vpi^\ell \Lambda$. 
\end{proof}

\begin{para}
If $A$ is the coefficient system given by $A_\sigma(V) = \uHom_{U_\sigma}(D_r(U_\sigma),V)$, then 
\[ \cS_q^A(V^\univ) = \bigoplus_{\sigma \in \BT_q} \uHom_{U_\sigma}(D_r(U_\sigma),V^\univ) \;\stackrel{\ref{analytic-vectors-computation}}{=} \; \bigoplus_{\sigma \in \BT_q} D_r(\bbS_\sigma) \]
for $q = 0, 1, \dotsc, \ell$.
\end{para}

\begin{para}
Consider the $(k+1)$-fold pushout $\bbS^{\sqcup k}_{(\sigma)}$ of $\bbS$ over $\bbS_\sigma$, i.e., 
\[\bbS^{\sqcup k}_{(\sigma)} = \underbrace{\bbS \sqcup_{\bbS_\sigma} \ldots \sqcup_{\bbS_\sigma} \bbS}_{k+1 \text{ times}}. \]
This pushout exists since $\sC$ is cocomplete. Explicitly, it is the family $(T_s)_{s \in G/U}$ where $T_s = *$ for every $s = gU \in G/U$ such that $U_\sigma gU \subseteq gU$, and $T_s = \{0, 1, \dotsc, k\}$ for all other $s$. Observe that $\bbS^{\sqcup 0}_{(\sigma)} = \bbS$. 
\end{para}

\begin{para} \label{contractibility}
As $k$ varies, the sets $\bbS^{\sqcup k}_{(\sigma)}$ assemble into a cosimplicial set $\bbS^{\sqcup\bullet}_\sigma$ equipped with a map $\bbS_\sigma \to \bbS_{(\sigma)}^{\sqcup \bullet}$, as in \cite[tag \href{https://stacks.math.columbia.edu/tag/016N}{016N}]{stacks-project}.
\begin{equation} \label{cosimplicial} \begin{tikzcd} 
\bbS_\sigma \ar[hookrightarrow]{d} \\
\bbS \ar[equals]{r} & \bbS^{\sqcup 0}_{(\sigma)} \ar[shift left=0.2em]{r} \ar[shift right=0.2em]{r} & \bbS^{\sqcup 1}_{(\sigma)}  \ar{r} \ar[shift left=0.4em]{r} \ar[shift right=0.4em]{r} & \bbS^{\sqcup 2}_{(\sigma)} \ar[shift right=0.2em]{r} \ar[shift right=0.6em]{r} \ar[shift left=0.2em]{r} \ar[shift left=0.6em]{r} & \cdots 
\end{tikzcd} \end{equation}
Since ${}_\sigma \bbS$ splits the inclusion $\bbS_\sigma \to \bbS$, the map $\bbS_\sigma \to \bbS_{(\sigma)}^{\sqcup \bullet}$ is a homotopy equivalence of cosimplicial objects of $\sC$ \cite[tag \href{https://stacks.math.columbia.edu/tag/019Z}{019Z}]{stacks-project}.
\end{para}

\begin{para} \label{contractibility-2}
Applying the functor $D_r(-)$ of \cref{free-cocompletion-category} to this cosimplicial object of $\sC$, the map $D_r(\bbS_\sigma) \to D_r(\bbS_{(\sigma)}^{\sqcup \bullet})$ is a homotopy equivalence of cosimplicial solid vector spaces. Regarding $D_r(\bbS_{(\sigma)}^{\sqcup \bullet})$ as a complex via the Dold-Kan correspondence \cite[tag \href{https://stacks.math.columbia.edu/tag/019H}{019H}]{stacks-project}, we see that $D_r(\bbS_\sigma) \to D_r(\bbS_{(\sigma)}^{\sqcup \bullet})$ is a homotopy equivalence of complexes. In other words, the complex
\[ \begin{tikzcd}
D_r(\bbS) \ar{r} & D_r(\bbS_{(\sigma)}^{\sqcup 1}) \ar{r} & D_r(\bbS_{(\sigma)}^{\sqcup 2}) \ar{r} & \cdots 
\end{tikzcd} \]
in degrees $[0, \infty)$ has cohomology only in degree 0, where the cohomology is $D_r(\bbS_\sigma)$. 
\end{para}

\begin{para} \label{double-complex}
Varying $\sigma \in \BT$, we obtain a second-quadrant double complex as follows.
\[ \begin{tikzcd}
& & \vdots & & \vdots \\ 
\cdots \ar{r} & 0 \ar{r} & \displaystyle \bigoplus_{\sigma \in \BT_\ell} D_r(\bbS^{\sqcup 2}_{(\sigma)}) \ar{r} \ar{u} & \cdots \ar{r} & \displaystyle \bigoplus_{\sigma \in \BT_0} D_r(\bbS^{\sqcup 2}_{(\sigma)}) \ar{u} \\
\cdots \ar{r} & 0 \ar{r} & \displaystyle \bigoplus_{\sigma \in \BT_\ell} D_r(\bbS^{\sqcup 1}_{(\sigma)}) \ar{r} \ar{u} & \cdots \ar{r} & \displaystyle \bigoplus_{\sigma \in \BT_0} D_r(\bbS^{\sqcup 1}_{(\sigma)}) \ar{u} \\
\cdots \ar{r} & 0 \ar{r} & \displaystyle \bigoplus_{\sigma \in \BT_\ell} D_r(\bbS) \ar{r} \ar{u} & \cdots \ar{r} & \displaystyle \bigoplus_{\sigma \in \BT_0} D_r(\bbS) \ar{u}
\end{tikzcd} \]
\end{para}

\begin{para} \label{vertical-differentials-conclusion}
It follows from the contractibility of \cref{contractibility-2} that, if one considers the spectral sequence starting with vertical differentials on the $\mbox{}^v E_0$ page, then the $\mbox{}^v E_1$ page is concentrated in the row $j = 0$, where it is given by the analytic Schneider-Stuhler complex of $V^\univ = D_r(\bbS)$. 
\[ \mbox{}^v E_1^{i,j} = \begin{cases} \displaystyle \bigoplus_{\sigma \in \BT_{-i}} D_r(\bbS_{\sigma}) & \text{if } i \leq 0 \text{ and } j = 0 \\ 0 & \text{otherwise} \end{cases} \]
Thus this spectral sequences collapses on page 2 and shows that the total complex of the double complex of \cref{double-complex} is quasi-isomorphic to $\cS^A_\bullet(V^\univ)$. 
\end{para}

\begin{para}
We now compute the cohomology of the total complex using the spectral sequence $\mbox{}^h E_\bullet$ that starts with horizontal differentials. For a coset $s = gU \in G/U$, we consider the subcomplex
\[ \BT_s = \{\sigma \in \BT \midc U_\sigma g U \subseteq g U\}. \]
\end{para}

\begin{proposition} \label{subcomplex-contractible}
$\BT_s$ is contractible. 
\end{proposition}

\begin{proof}
It suffices to show that $\BT_s$ is geodesically convex \cite[p. 184]{RonanBuildings}. To see this, assume that $\sigma$ and $\tau$ are vertices belonging to $\BT_s$. Then for any vertex $\omega$ on the path between $\sigma$ and $\tau$ the group $U_\omega$ is contained in the subgroup generated by $U_\sigma$ and $U_\tau$ (cf. \cite[proposition I.3.1]{ScSt97} or \cite[1.28 Lemma]{VignerasSheaves}). Hence $\omega \in \BT_s$. 
\end{proof}

\begin{corollary} \label{pushout-contractible}
The $(j+1)$-fold pushout \[ \BT \sqcup_{\BT_s} \BT \sqcup_{\BT_s} \cdots \sqcup_{\BT_s} \BT \] is contractible. 
\end{corollary}

\begin{proof}
The case $j = 0$ is known. It suffices by induction to show that $X \sqcup_{\BT_s} \BT$ is contractible whenever $X$ is a contractible simplicial complex containing $\BT_s$ as a subcomplex. Then $\BT_s$ is a neighborhood deformation retract of both $X$ and $\BT$. In particular, the closed inclusions $\BT_s \to X$ and $\BT_s \to \BT$ are both cofibrations and $X \sqcup_{\BT_s} \BT$ is a homotopy pushout. 

Observe that $X \sqcup_{\BT_s} \BT$ is contractible if and only if $\pi_1(X \sqcup_{\BT_s} \BT) = 1$ and $H_i(X \sqcup_{\BT_s} \BT) = 0$ for all $i$ \cite[corollary 4.33]{Hatcher}. Since $\BT_s$ is contractible by \cref{subcomplex-contractible}, it is path-connected; as we noted above, it is also a neighborhood deformation retract of each of the contractible spaces $X$ and $\BT$, so the fact that $\pi_1(X \sqcup_{\BT_s} \BT) = 1$ follows from the Seifert-van Kampen theorem. It therefore suffices to show that $H_i(X \sqcup_{\BT_s} \BT) = 0$ for all $i$. Since $X \sqcup_{\BT_s} \BT$ is a homotopy pushout, we have a distinguished triangle
\[ \begin{tikzcd} C_\bullet(\BT_s) \ar{r} & C_\bullet(X) \oplus C_\bullet(\BT) \ar{r} & C_\bullet(X \sqcup_{\BT_s} \BT) \ar{r}{+} & \mbox{} \end{tikzcd} \]
so it suffices to show that $H_i(\BT_s) \to H_i(X) \oplus H_i(\BT)$ is an isomorphism for all $i$. This again follows from contractibility of $X, \BT$, and $\BT_s$, and so we are done. Note that we have used \cref{subcomplex-contractible} again here. 
\end{proof}

\begin{para} \label{decomp}
Fix $j \geq 0$. For $q = 0, 1, \dotsc, \ell$, observe that the coproduct 
\[ \bigsqcup_{\sigma \in \BT_q} \bbS^{\sqcup j}_{(\sigma)} \]
in $\sC$ is the family $(T_s)_{s \in G/U}$ where $T_s$ is the set of $q$-simplices in 
\[ \underbrace{\BT \sqcup_{\BT_s} \sqcup \ldots \sqcup_{\BT_s} \BT}_{(j+1)-\text{fold}}. \]
\end{para}

\begin{para}
Thus, we have 
\[ D_r \left(\bigsqcup_{\sigma \in \BT} \bbS^{\sqcup j}_{(\sigma)} \right) = \bigoplus_{s \in G/U} D_r(s) \sotimes_E C_\bullet \left( \underbrace{\BT \sqcup_{\BT_s} \sqcup \ldots \sqcup_{\BT_s} \BT}_{(j+1)-\text{fold}} \right) \]
as chain complexes, where the left-hand side is the chain complex occurring in row $j$ of the the double complex in \cref{double-complex}. 
\end{para}

\begin{para}
We know from \cref{pushout-contractible} the iterated pushout $\BT \sqcup_{\BT_s} \sqcup \ldots \sqcup_{\BT_s} \BT$ is contractible, so $C_\bullet(\BT \sqcup_{\BT_s} \sqcup \ldots \sqcup_{\BT_s} \BT)$ is quasi-isomorphic to $E$. The functor 
\[ D_r(s) \sotimes_E - \]
preserves quasi-isomorphisms \cite[propositions A.31 and A.28]{BoscoDrinfeldv2}, so 
\[ D_r \left(\bigsqcup_{\sigma \in \BT} \bbS^{\sqcup j}_{(\sigma)} \right) = \bigoplus_{s \in G/U} D_r(s) = D_r(\bbS). \]
\end{para}

\begin{para}
Returning to the spectral sequence $\mbox{}^h E_\bullet$ whose $\mbox{}^h E_0$ page is the horizontal differentials of the double complex of \cref{double-complex}, we see that the $\mbox{}^h E_1$ page is concentrated in the column $i = 0$. 
\[ \mbox{}^h E_1^{i,j} = \begin{cases} D_r(\bbS) & \text{if } i = 0 \text{ and } j \geq 0 \\ 0 & \text{ otherwise} \end{cases} \]
\end{para}

\begin{lemma}
For any $j \geq 0$, the vertical differential $\mbox{}^h E_1^{0,j} \to \mbox{}^h E_1^{0,j+1}$ is 0 if $j$ is even and the identity if $j$ is odd. 
\end{lemma}

\begin{proof}
The map $D_r(\bbS_{(\sigma)}^{\sqcup j}) = \mbox{}^h E_0^{0,j} \to \mbox{}^h E_1^{0,j} = D_r(\bbS)$ is induced by the natural map $\bbS_{(\sigma)}^{\sqcup j} \to \bbS$ which maps each copy of $\bbS$ in $\bbS_{(\sigma)}^{\sqcup j}$ onto itself via the identity. The sign alternation of the differentials in the Dold-Kan complex associated to a cosimplicial solid vector space (cf. \cref{contractibility}) induces a corresponding sign alternation in the vertical differentials of the $\mbox{}^h E_1$ page, and an even (resp. odd) alternating sum of identity map is the zero (resp. identity) map. 
\end{proof}

\begin{para}
It follows that the $\mbox{}^h E_2$ page has only one nonzero term: $\mbox{}^h E_2^{0,0} = D_r(\bbS)$. Combined with the observations of \cref{vertical-differentials-conclusion}, this concludes the proof of \cref{universal-rep}. \qed
\end{para}

%% file: 5-mixing.tex
\section{Mixing Schneider-Stuhler and Chevalley-Eilenberg complexes}\label{mixing}

We continue to use the notation introduced in \cref{notation-reductive}. 

\begin{para}{\it The representations in the analytic Schneider-Stuhler complex are in general not projective solid locally analytic representations.}
To motivate the constructions below we start with the following observation. Let $A$ be any of the coefficient systems in \cref{examples-functors-A}. Let $V$ be an irreducible smooth representation with a central character $\chi$.  Then the representations in the analytic Schneider-Stuhler complex $\cS^A(V)$  of \cref{analytic-schneider-stuhler} are projective in the category of smooth representations with central character $\chi$, by \cite[II.2.2]{ScSt97}. However, the representations $\cS^A(V)$ (if they are not zero, which we can assume by taking the level $e$ large enough) are in general no longer projective in the category of solid locally analytic representations with fixed central character, as we will explain in the following example.  
\end{para}

\begin{example}\label{counterexample-projective} 
Let $M(0)$ denote the Verma module with weight 0 in category $\cO$ for $\frg = \frs\frl_2$. Let $\alpha$ be the unique positive root for the Borel subalgebra of upper triangular matrices. Then we have an exact sequence
$$0 \lra M(-\alpha) \lra M(0) \lra L(0) \lra 0 \,,$$
where $M(-\alpha)$ is the Verma module with weight $-\alpha$ and $L(0)$ is the trivial one-dimensional representation. Taking the dual $M \rightsquigarrow M^\vee$ in the BGG category $\cO$ we obtain the exact sequence
$$0 \lra  L(0) \lra M(0)^\vee \lra M(-\alpha)  \lra 0 \,.$$
Let $G = \SL_2(F)$ and $B \sub G$ the subgroup of upper triangular matrices. Applying the functor $\cF(-) = \cF^G_B(-)$ from  \cite{OrlikStrauchJH} to this exact sequence  gives the exact sequence
\[\begin{tikzcd}
0 \ar{r} & \cF(M(-\alpha)) \ar{r}{\phi} & \cF(M(0)^\vee) \ar{r}{\psi} & \cF(L(0)) = \ind^G_B({\bf 1}) \ar{r} & 0 \,,    
\end{tikzcd}\]
where $V := \ind^G_B({\bf 1})$ denotes the smooth induction of the trivial character. For any of the coefficient systems $A$ in \cref{examples-functors-A} the representation $\cS^A_0(V)$ is smooth, and it surjects onto $V$ via the augmentation $\vep: \cS^A_0(V) \ra V$ if the level $e$ is large enough, what we will assume in the following. If $\cS^A_0(V)$ would be a projective object in the category of solid representations of $G$ (or in the category of solid locally analytic representations of $G$), then this would imply that $\vep$ lifts to a $G$-homomorphism $\tau: \cS^A_0(V) \ra \cF(M(0)^\vee)$. Then $\im(\tau)$ is a smooth subrepresentation of $\cF(M(0)^\vee)$ which surjects onto $V$ via $\psi$. We thus find that the space $\cF(M(0)^\vee)^\frg$ annihilated by $\frg$ is non-zero. Passing to dual spaces, we see that for the space of $\frg$-coinvariants we have 
\begin{numequation}\label{coinvariants}
D(G) \otimes_{D(\frg,B)} M(0)^\vee/\frg \Big(D(G) \otimes_{D(\frg,B)} M(0)^\vee\Big) \; \neq \; 0 \;.    
\end{numequation}
However, 
$$\frg \Big(D(G) \otimes_{D(\frg,B)} M(0)^\vee\Big)=  D(G) \otimes_{D(\frg,B)} \frg M(0)^\vee  = D(G) \otimes_{D(\frg,B)} M(0)^\vee$$
since $\frg M(0)^\vee = M(0)^\vee$, as one can easily show. This contradicts \ref{coinvariants}. Therefore, $\cS^A_0(V)$ is not a projective object in the category of solid (or solid locally analytic) representations of $G$. 
\end{example}

\subsection{Wall complexes with analytic vectors} 

\begin{para}\label{fixtures-mixing}
The discussion in \ref{counterexample-projective} prompts us to further resolve the representations $\cS_q^A(V)$ defined in \ref{analytic-schneider-stuhler}, which is what we will be doing in the following. For the remainder of \cref{mixing} we fix the following objects and will make the following assumptions:
\begin{enumerate}
\item $\rho \in (0,1] \cap p^\bbQ$ (cf. \ref{kappa}) and $r \in \left(\frac{1}{p},p^{-\frac{\rho}{\kappa(p-1)}}\right] \cap \sR$ (cf. \ref{the-set-sR}). For example, one may take $\rho=1$ and $r \in \Big(p^{-\frac{1}{\kappa(p-1)}-\frac{1}{\kappa e(F/\Qp)}},p^{-\frac{1}{\kappa (p-1)}}\Big)$. 
\item A Schneider-Stuhler level $e$ so that for all facets $\sigma$ of $\BT$ the group $\Us = U^{(e)}_\sigma$ is $F$-uniform (cf. \cref{ss-groups-uniform}). 
\item The functor $A$ which is given by $A_\sigma(V) = \uHom_\Us(\DrUs,V)$. This is a solid $D_r(\Us)$-module, cf. \cref{uHomDrV-has-Dr-module-structure}. 
\end{enumerate}
\end{para}

\begin{para} \label{H-sigma} {\it The groups $\Hs$ and $\bbHso$.}
Given a facet $\sigma$ of $\BT$ and a vertex $x \in \overline{\sigma}$ the group $U_x$ has the integral powerful\footnote{in the sense of \cite[sec. 9.4]{DDMS}} Lie algebra $\Lie_\Zp(U_x)$ which is an $\cO_F$-lattice in $\frg = \Lie(G)$. The intersection $\bigcap_{x \in \overline{\sigma} \cap \BT_0} \Lie_\Zp(U_x)$ is again a powerful $\Zp$-Lie algebra and an $\cO_F$-lattice in $\frg$. We then let $\Hs \sub G$ be the unique $F$-uniform compact open subgroup with that integral Lie algebra. It follows from \cite[I.2.11]{ScSt97} that $\Hs$ is contained in any of the groups $\Ut$ with $\tau$ a facet in $\overline{\sigma}$. To the fixed real number $\rho$ above we then consider the strictly ind-affinoid group
\[\bbHso := \bbHsro \;,\]
and we caution the reader that this group depends on $\rho$, but this is not visible in the notation (so as to lighten the notation somewhat). By \cref{distalgs-of-wide-open-and-overconv-gps-and-Dr} and \ref{Dr-functoriality} there are canonical morphisms of solid distribution algebras
\begin{numequation}\label{from-bbHo-to-Dr}
D(\bbHso) \lra D(\bbHso,H_\sigma) \lra D_r(H_\sigma) \lra D_r(\lra D_r(\Ut),
\end{numequation}
whenever $\tau$ is contained in $\ovsigma$. 
\end{para}

\begin{para}\label{the-group-C}{\it The group $C$ and the group $\wHso$.} We fix once and for all a finitely generated free abelian group $C \sub \bZ_\bG(F)$ of rank equal to the $F$-split rank of the center $\bZ_\bG$ of $\bG$ such that $\bZ_\bG(F)/C$ is compact. This group has trivial intersection with any compact subgroup of $G$. Moreover, as the center acts trivially on the semisimple Bruhat-Tits building $\BT$, the group $C$ is contained in any of the groups $\Psd$\footnote{$\Psd$ denotes the stabilizer of the facet $\sigma$ as in \cite[I.1.7, p. 105]{ScSt97}, which is not to be confused with the pointwise stabilizer of $\sigma$.} and $\Psd/C$ is compact. We also set $\Hso = \bbHso(F)$ and $\wHso = C\Hso \cong C \x \Hso$. Then $(\bbHso,\wHso)$ and $(\bbHso,\Psd)$ are both analytic group pairs, and they are of the type considered in \cref{ind-affinoid-mixed-generalized}. Moreover, using \ref{from-bbHo-to-Dr} and the fact that $\Psd$ normalizes $\Us$, we see that the $D_r(\Us)$-module $A_\sigma(V)$ is also a module over $\DHPsd$.
    
\end{para}

\begin{para}\label{using-CE-resolutions}{\it Using Chevalley-Eilenberg type resolutions.}
Let $V$ be a solid $G$-representation. 
Denote by 
$$(\CE_j(\bbHso,\Psd), d_{\sigma,j})_{j \ge 0}$$ 
the complex which gives a resolution of the trivial one-dimensional representation $E$ as a module over $\DHPsd$, as explained in \cref{ind-affinoid-mixed-generalized}. This is a homological complex of free finitely generated $\DHPsd$-modules in degrees $j \ge 0$ which resolves $E$. The differential $d_{\sigma,0}$ is the zero map. We take the solid tensor product of this complex with $A_\sigma(V)$ and equip this with the 'diagonal' module structure using that $\DHPsd$ is a solid Hopf algebra. The resulting complex is denoted by $(\CE_\bullet(\bbHso,\Psd) \sotimes_E A_\sigma(V), d_{\sigma,\bullet}$). We denote by $\vep_\sigma$ the augmentation map, so that we obtain a resolution 
\begin{numequation}\label{mixed-CE-resolution}
\begin{tikzcd} 
\CE_0(\bbHso,\Psd) \sotimes_E A_\sigma(V) = \DHPsd \sotimes_E A_\sigma(V) \ar{r}{\vep_\sigma} & A_\sigma(V) \ar{r} & 0 \,. \end{tikzcd} 
\end{numequation}
As we have fixed the coefficient system $A$ in \cref{fixtures-mixing}, we write $\cS_q(V)$ instead of $\cS_q^A(V)$ in the following. Let $\cR(\BT_q) \sub \BT_q$ be a complete system of representatives for the orbits of $G$ on $\BT_q$, of which there are only finitely many. Recall from \ref{c-ind-description-2} that
\[\cS_q(V) = \bigoplus_{\sigma \in \cR(\BT_q)}  \cind^G_\Psd \left(A_\sigma(V)\right) \,.\]
Passing from \ref{mixed-CE-resolution} to compact induction we obtain a complex
\[ \left(\cS^\CE_{\sigma,j}(V) := \cind^G_\Psd \Big(\CE_j(\bbHso,\Psd) \sotimes_E A_\sigma(V)\Big), \dG_{\sigma,j} = \cind^G_\Psd (d_{\sigma,j}) \right)_{j \ge 0} \] 
which has an augmentation
\[ \begin{tikzcd} 
\cind^G_\Psd \left(\CE_\bullet(\bbHso,\Psd) \sotimes_E A_\sigma(V)\right) \ar{r}{\vep^G_\sigma} & \cind^G_\Psd \left(A_\sigma(V)\right) \ar{r} & 0,
\end{tikzcd} \]
which give a resolution of $\cind^G_\Psd \left(A_\sigma(V)\right)$. Taking direct sums of these maps over $\sigma \in \cR(\BT_q)$ gives a complex
\begin{numequation}\label{res-of-Sq}
\left(\cS^\CE_{q,\bullet}(V) = \bigoplus_{\sigma \in \cR(\BT_q)}  \cS^\CE_{\sigma,\bullet}(V)\;, \;\; d_{q,j}^G = \bigoplus_{\sigma \in \cR(\BT_q)} \;\; d_{\sigma,j}^G\right)_{j \ge 0}
\end{numequation}
which has an augmentation 
\begin{numequation}\label{res-of-Sq-k}
\begin{tikzcd} \cS^\CE_{q,0}(V) \ar{r}{\vep^G_q} & \cS_q(V) \ar{r} & 0 \,, \end{tikzcd}
\end{numequation}
which gives a resolution of $\cS_q(V)$. The differentials in the complex $\cS_\bullet(V)$ are denoted by $d^\cS_q$, and we let $\vep^\cS: \cS_0(V) \ra V$ be the augmentation map to $V$.
\end{para}

\begin{theorem}\label{thm-mixed-res}  
Assume that for all facets $\sigma$ of $\BT$\footnote{Equivalently, for every representative for $G\bksl BT_q$ and every $q = 0, \ldots, \ell$.} the solid vector space $A_\sigma(V)$ is a Smith space.\footnote{For example, this is the case when $V$ is an admissible locally analytic representation, by \ref{Dr-and-HomDr-for-compact-ind-limits} and \ref{Dr-analytic-Smith}.}  Then there are morphisms of solid $G$-representations
$$d^{(k)}_{q,j}: \cS^\CE_{q,j}(V) \lra \cS^\CE_{q-k,j+k-1}(V) \,,\; 0 \le k \le q,$$
such that, if we set 
$$\cS^\CE_n(V) = \bigoplus_{\scalebox{.7}{$\begin{array}{c} 0 \le q \le \ell\,, \; 0 \le j \\
q+j = n \end{array}$}} \cS^\CE_{q,j}(V)$$ 
and $\Delta_n = \sum_{q+j=n} \sum_{k=0}^{q} d^{(k)}_{q,j}: \cS^\CE_n(V) \ra \cS^\CE_{n-1}(V)$, one has $\Delta_{n-1} \circ \Delta_n = 0$, and 
$$\mbox{for all } n \ge 0: \; h_n(\cS^\CE_\bullet(V),\Delta_\bullet) = h_n(\cS_\bullet(V),d^\cS) \;.$$ 
In particular, if $\cS_\bullet(V) \xrightarrow{\vep^\cS} V$ is a resolution of $V$, then so is $\cS^\CE_\bullet(V) \xrightarrow{\vep^\cS \, \circ \, \vep^G_0} V$.
\end{theorem}

\begin{proof} By the general formalism of the Wall complex, cf. \cite[ch. V, 3]{LazardGroupesAnalytiques}, we only need to show the existence of the $G$-homomorphisms $d^{(k)}_{q,j}$. It is helpful to sketch the arrangement of the representations $\cS^\CE_{q,j}$ as follows:
$$\begin{tikzcd}[sep=large]
&  \ar[d,"d^{(0)}_{q,2}"] &  \ar[d] & \ar[d,"d^{(0)}_{0,2}"] \\
 \ar[r] & \cS^\CE_{q,1} \arrow{r}{}[swap]{d^{(1)}_{q,1}} \ar[d,"d^{(0)}_{q,1}"] \ar[rru,"d^{(2)}_{q,1}" near start, crossing over] & \cS^\CE_{q-1,1} \ar[r,"d^{(1)}_{q-1,1}"] \ar[d] & \cS^\CE_{q-2,1} \ar[d,"d^{(0)}_{q-2,1}"] \ar[r,"d^{(1)}_{q-2,1}"] & \cdots  \ar[r,"d^{(1)}_{1,1}"] & \cS^\CE_{0,1} \ar[d,"d^{(0)}_{0,1}"]\\
 \ar[r] &  \cS^\CE_{q,0} \arrow{r}{}[swap]{d^{(1)}_{q,0}}  \ar[d,"\vep^G_{q}"] \ar[rru,"d^{(2)}_{q,0}" near start, crossing over] &  \cS^\CE_{q-1,0} \ar[r,"d^{(1)}_{q-1,0}"]  \ar[d,"\vep^G_q-1"] & \cS^\CE_{q-2,0} \ar[d,"\vep^G_{q-2}"]  \ar[r,"d^{(1)}_{q-2,0}"] & \cdots \ar[r,"d^{(1)}_{1,0}"] & \cS^\CE_{0,0} \ar[d,"\vep^G_0"] \\
  \ar[r] &  \cS_{q} \ar[r,"d^\cS_q"] &  \cS_{q-1} \ar[r,"d^\cS_{q-1}"] & \cS_{q-2} \ar[r,"d^\cS_{q-2}"] &  \cdots \ar[r,"d^\cS_1"] & \cS_0 \ar[r,"\vep^\cS"] & V  
\end{tikzcd}$$
In order not to overload the diagram with arrows, we did not depict any maps $d^{(k)}_{q,j}$ with $k \ge 3$. The idea of the proof is to construct these maps inductively so that the condition $\Delta_{n-1} \circ \Delta_n = 0$ is fulfilled. 
\vskip8pt

\noindent {\it Step 1. The vertical maps $d^{(0)}_{q,j}$ and the horizontal maps $d^{(1)}_{q,0}$.} We set for all $q,j \ge 0$ 
$$d^{(0)}_{q,j} = d^G_{q,j}: \cS^\CE_{q,j} \lra \cS^\CE_{q,j-1} \;, \, \mbox{ and in particular } \; d^{(0)}_{q,0} = 0 \;.$$ 
These are the vertical maps and the vertical complexes ({\it complexes fibres} in \cite{LazardGroupesAnalytiques}). Now we assume $q \ge 1$ and we show the existence of $d^{(1)}_{q,0}: \cS^\CE_{q,0} \ra \cS^\CE_{q-1,0}$ such that 
\begin{numequation}\label{zeroth-diagram}
\begin{tikzcd}[sep=large]
 \cS^\CE_{q,0} \ar[r,"d^{(1)}_{q,0}",dashrightarrow]  \ar[d,"\vep^G_q",twoheadrightarrow] & \cS^\CE_{q-1,0} \ar[d,"\vep^G_{q-1}",twoheadrightarrow] \\
 \cS_q \ar[r,"d^\cS_q"] & \cS_{q-1}
\end{tikzcd}    
\end{numequation}
commutes. By Frobenius reciprocity we only need to show that $\vep^G_{q-1} \circ d^{(1)}_{q,0} = d^\cS_q \circ \vep^G_q$ when these maps are restricted to any of the subspaces 
$$T_{\sigma,0,0}\footnote{The notation will be generalized below and will become more transparent then.} := \CE_0(\bbHso,\Psd) \sotimes_E A_\sigma(V) \sub \cS^\CE_{q,0}\;,$$ 
for $\sigma \in \cR(\BT_q)$. 
We have 
$$(d^\cS_q \circ \vep^G_q)\Big(T_{\sigma,0,0}\Big) \subset \bigoplus_{\scalebox{.7}{$\begin{array}{c}\tau \in \cR(\BT_{q-1}) \\ g \in G/\widetilde{P}_\tau \\
g.\tau \subset \ovsigma \end{array}$}} \delta_g \,.\, A_\tau(V) \;.$$
We set 
$$T_{\sigma,1,0} := \bigoplus_{\scalebox{.7}{$\begin{array}{c}\tau \in \cR(\BT_{q-1}) \\ g \in G/\widetilde{P}_\tau \\
g.\tau \subset \ovsigma \end{array}$}} \delta_g \,.\, \CE_0(\bbHto,\wPt) \sotimes_E A_\tau(V)\;.$$
If $g.\tau \subset \ovsigma$, then $g^{-1}\bbHso g = \bbH^\o_{g^{-1}.\sigma} \sub \bbHto$, and $\delta_g.\CE_0(\bbHto,\wPt) \sotimes_E A_\tau(V)$ is naturally a module over $D(\bbHso)$. Moreover, for $h \in \Psd$ and $g.\tau \subset \ovsigma$ we have $h.(g.\tau) \sub h.\ovsigma = \ovsigma$, and we see that the space $T_{\sigma,1,0}$
is stable under the action of $\Psd$. It is thus a solid module over $\DHPsd$. The existence of a commutative diagram \ref{zeroth-diagram} is therefore implied by the existence of a commutative diagram
\begin{numequation}\label{zeroth-diagram-T}
\begin{tikzcd}[sep=large]
T_{\sigma,0,0} \ar[r,"d^{(1)}_{q,0}",dashrightarrow]  \ar[d,"\vep^G_q",twoheadrightarrow] & T_{\sigma,1,0} \ar[d,"\vep^G_{q-1}",twoheadrightarrow] \\
A_\sigma(V) \ar[r,"d^\cS_q"] &  \bigoplus_{\scalebox{.7}{$\begin{array}{c}\tau \in \cR(\BT_{q-1}) \\ g \in G/\widetilde{P}_\tau \\
g.\tau \subset \ovsigma \end{array}$}} \delta_g \,.\, A_\tau(V)
\end{tikzcd}    
\end{numequation}
We note that all spaces in this diagram are $\DHPsd$-modules, and that the maps are $\DHPsd$-linear. Now we use that $T_{\sigma,0,0}$ is a projective solid $\DHPsd$-module. We can therefore complete the digram \ref{zeroth-diagram-T} by a map $d^{(1)}_{q,0}$ which is a $\DHPsd$-module homomorphism, and which gives rise to $d^{(1)}_{q,0}: \cS^\CE_{q,0} \ra \cS^\CE_{q-1,0}$ (by Frobenius reciprocity) which makes the diagram \ref{zeroth-diagram} commutative. We also set $d^{(1)}_{0,j} = 0$ for all $j \ge 0$. 
\vskip8pt
\noindent {\it Step 2. Induction hypotheses.} 
By induction we assume that all $d^{(k)}_{q,j}$ with $q+j<n$ and $0 \le k \le q$ have already been defined. We also assume that for any $\sigma \in \cR(\BT_q)$ we have
\begin{numequation}\label{Hyp-1}\tag{Hyp-1}
\begin{array}{c} \hskip-200pt d^{(k)}_{q,j}\Big(\CE_j(\bbHso,\Psd) \sotimes_E A_\sigma(V) \Big)  \\
\\
\hskip60pt \subset  \;\; T_{\sigma,k,j+k-1} := \bigoplus_{\scalebox{.7}{$\begin{array}{c}\tau \in \cR(\BT_{q-k}) \\ g \in G/\widetilde{P}_\tau \\
g.\tau \subset \ovsigma \end{array}$}} \delta_g \,.\, \CE_{j+k-1}(\bbHto,\wPt) \sotimes_E A_\tau(V) \;.
\end{array}
\end{numequation}
This is clearly the case for $k=0$, as then we actually have 
\begin{numequation}\label{Hyp-1-d0}
d^{(0)}_{q,j}\Big(\CE_j(\bbHso,\Psd) \sotimes_E A_\sigma(V) \Big) \subset \CE_{j-1}(\bbHso,\Psd) \sotimes_E A_\sigma(V) \;.
\end{numequation}
Condition \ref{Hyp-1} is also satisfied for the maps $d^{(1)}_{q,0}$ constructed in Step 1. We now strengthen condition \ref{Hyp-1} by requiring (in addition to \ref{Hyp-1}) that each map
\begin{numequation}\label{Hyp-2}\tag{Hyp-2}
d^{(k)}_{q,j}|_{T_{\sigma,0,j}}: \; T_{\sigma,0,j} \lra T_{\sigma,k,j+k-1} \;\; \mbox{ is $\DHPsd$-linear.}     
\end{numequation}
Again, this is the case for the maps $d^{(0)}_{q,j}$ as they are induced by the differentials in the Chevalley-Eilenberg complex for the resolution of the trivial module over $\DHPsd$. And it is true for the maps $d^{(1)}_{q,0}$ of Step 1, by construction. 
\vskip8pt
\noindent {\it Step 3. The existence of the maps $d^{(k)}_{q,j}$: generalities.}
Furthermore, $\Delta_{n-1} \circ \Delta_n = 0$ is equivalent to 
\begin{numequation}\label{induction-identity}
\sum_{h=0}^k d^{(k-h)}_{q-h,j+h-1} \circ d^{(h)}_{q,j} = 0 \;.    
\end{numequation} We thus fix $k \in \{1, \ldots, q\}$, and we can assume that $d^{(h)}_{q,j}$ has been defined for $0 \le h < k \le q$. If we set $\partial = \sum_{h=0}^{k-1} d^{(k-h)}_{q-h,j+h-1} \circ d^{(h)}_{q,j}$ we must have $d^{(0)}_{q-k,j+k-1} \circ d^{(k)}_{q,j} + \partial = 0$. Consider the following diagram
\begin{numequation}\label{first-diagram}
\begin{tikzcd}[sep=large]
 & \cS^\CE_{q-k,j+k-1} \ar[d,"d^{(0)}_{q-k,j+k-1}"] \\
 \cS^\CE_{q,j} \ar[r,"-\partial"] \ar[ru,"d^{(k)}_{q,j}",dashrightarrow] & \cS^\CE_{q-k,j+k-2} \ar[d,"d^{(0)}_{q-k,j+k-2}"] \\
 & \cS^\CE_{q-k,j+k-3}
\end{tikzcd}    
\end{numequation}
where the existence of the dashed arrow $d^{(k)}_{q,j}$ is to be established. We first note that 
$$\begin{array}{cl}  & d^{(0)}_{q-k,j+k-2} \circ \partial\\ 
&\\
= & \sum_{h=0}^{k-1} d^{(0)}_{q-k,j+k-2} \circ d^{(k-h)}_{q-h,j+h-1} \circ d^{(h)}_{q,j}  \\
&\\
= & -\sum_{h=0}^{k-1} \Big[\sum_{\ell =0}^{k-h-1} d^{(k-h-\ell)}_{q-h-\ell,j+h+\ell-2} \circ d^{(\ell)}_{q-h,j+h-1}\Big] \circ d^{(h)}_{q,j} \\
&\\
= & -\sum_{m=1}^{k-1} d^{(k-m)}_{q-m,j+m-2} \circ \Big[\sum_{h =0}^m d^{(m-h)}_{q-h,j+h-1} \circ d^{(h)}_{q,j}\Big] = 0\\
\end{array}$$
where the second equality comes from applying \ref{induction-identity} to the case when $n$ is replaced by $n-1$ and $k$ by $k-h$, and the third equality is obtained by setting $m = h+\ell$ and rearranging terms. The last equality sign comes again from the identity \ref{induction-identity} for $m$, which is smaller than $k$. This proves that $\im(\partial) \sub \ker(d^{(0)}_{q-k,j+k-2})$. As the `vertical' complex $(\cS^\CE_{q-k,\bullet}, d^{(0)}_{q-k,\bullet})$ is exact in positive degrees, this shows that $\im(\partial) \sub \im(d^{(0)}_{q-k,j+k-1})$ if $j+k-2 > 0$. 
\vskip8pt
\noindent If $j+k-2 = 0$, then we have $(j,k)=(1,1)$ or $(j,k)=(0,2)$, as we assume here $k>0$. When $(j,k)=(1,1)$ then $\partial = d^{(1)}_{q,0} \circ d^{(0)}_{q,1}$, and we find that 
$$\vep^G_{q-1} \circ \partial = \vep^G_{q-1} \circ d^{(1)}_{q,0} \circ d^{(0)}_{q,1} = d^\sC_q \circ \vep^G_q \circ d^{(0)}_{q,1} = 0 \;,$$
and thus $\im(\partial) \subset \im(d^{(0)}_{q-1,1})$. Similarly, if $(j,k) = (0,2)$ then $\partial = d^{(1)}_{q-1,0} \circ d^{(1)}_{q,0}$, and using again the properties of the maps $d^{(1)}_{q,0}$ we find that $\vep^G_{q-2} \circ \partial = 0$, so that $\im(\partial) \sub \im(d^{(0)}_{q-2,1})$.  
\vskip8pt
\noindent Now that we have established that we always have $\im(\partial) \sub \im(d^{(0)}_{q-k,j+k-1})$, we can replace the diagram \ref{first-diagram} by 
\begin{numequation}\label{second-diagram}
\begin{tikzcd}[sep=large]
 & \cS^\CE_{q-k,j+k-1} \ar[d,"d^{(0)}_{q-k,j+k-1}",twoheadrightarrow] \\
 \cS^\CE_{q,j} \ar[r,"-\partial"] \ar[ru,"d^{(k)}_{q,j}",dashrightarrow] & \im(d^{(0)}_{q-k,j+k-1})
\end{tikzcd}    
\end{numequation}
\vskip8pt
\noindent {\it Step 4. The existence of the maps $d^{(k)}_{q,j}$: Frobenius reciprocity.} When we apply the induction hypothesis \ref{Hyp-1} to the maps which make up $\partial$ we find that 
\begin{numequation}\label{induction-hyp-facets-partial}
\begin{array}{c} \hskip-200pt \partial\Big(\CE_j(\bbHso,\Psd) \sotimes_E A_\sigma(V) \Big)  \\
\\
\hskip70pt \subset  \; T_{\sigma,k,j+k-2} = \bigoplus_{\scalebox{.7}{$\begin{array}{c}\tau \in \cR(\BT_{q-k}) \\ g \in G/\widetilde{P}_\tau \\
g.\tau \subset \ovsigma \end{array}$}} \delta_g \;.\; \CE_{j+k-2}(\bbHto,\wPt) \sotimes_E A_\tau(V) \;.
\end{array}
\end{numequation}
By Frobenius reciprocity, we see that the existence of $d^{(k)}_{q,j}
$ follows from the identity 
$$d^{(0)}_{q-k,j+k-1} \circ d^{(k)}_{q,j}
= -\partial$$ 
when restricted to any of the generating subrepresentations $T_{\sigma,0,j} = \CE_j(\bbHso,\Psd) \sotimes_E A_\sigma(V)$, as $\sigma$ runs through $\cR(\BT_q)$. We fix $\sigma \in \cR(\BT_q)$ and consider the corresponding diagram of maps when restricted to $T_{\sigma,0,j}$, namely 
\begin{numequation}\label{second-diagram-T}
\begin{tikzcd}[sep=large]
 & T_{\sigma,k,j+k-1} \ar[d,"d^{(0)}_{q-k,j+k-1}",twoheadrightarrow] \\
 T_{\sigma,0,j} \ar[r,"-\partial"] \ar[ru,"d^{(k)}_{q,j}",dashrightarrow] & d^{(0)}_{q-k,j+k-1}\Big(T_{\sigma,k,j+k-1}\Big)
\end{tikzcd}    
\end{numequation}
By induction hypothesis \ref{Hyp-2} the maps $-\partial$ and $d^{(0)}_{q,j+k-1}$ are $\DHPsd$-linear. 
By \cref{hopf-projective}, $T_{\sigma,0,j}$ is a projective $\DHPsd$-module, and we can thus choose $d^{(k)}_{q,j}|_{T_{\sigma,0,j}}$ as a $\DHPsd$-linear map making the diagram \ref{second-diagram-T} commutative. Any such choice, made for every $\sigma \in \cR(\BT_q)$, gives rise to a $G$-linear map which maps $T_{\sigma,0,j}$ to $T_{\sigma,k,j+k-1}$, and whose restriction to $T_{\sigma,0,j}$ is a $\DHPsd$-linear map, by construction. Therefore, $d^{(k)}_{q,j}$ satisfies Hyp-1 and Hyp-2. This completes the induction step.
\end{proof}

\begin{para}\label{using-CE-resolution-truncated} {\it Using truncated Chevelley-Eilenberg type resolutions.} Set $\bd := \dim(\bG) + \rk_F(\bZ(\bG))$ where $\rk_F(\bZ(\bG))$ is the $F$-split rank of the center of $\bG$. Instead of working with the Chevalley-Eilenberg resolution $\CE_\bullet(\bbHso,\Psd)$ of $E$ as a $\DHPsd$-module, cf. \cref{ind-affinoid-mixed-generalized}, we can also use the canonical (or `good') truncation $\tau_{\le \bd}\CE_\bullet(\bbHso,\Psd)$ which is a complex of finitely generated projective $\DHPsd$-modules which resolves $E$ and is concentrated in degrees $0 \le j \le \bd$. In degree $j \le \bd$ it is equal to $\CE_\bullet(\bbHso,\Psd)$, and thus consists of free $\DHPsd$-modules, whereas the term in degree $\bd$ is projective and possibly not free.

Let $V$ be again a solid $G$-representation, and denote by $\cS_\bullet(V)$ the analytic Schneider-Stuhler complex associated to the coefficient system $A_\sigma(V) = \uHom_\Us(D_r(\Us),V)$ as before. We can then construct solid $G$-representations 
\[\tau\cS^\CE_{\sigma,j} := \cind^G_\Psd \Big(\tau_{\le \bd}\CE_j(\bbHso,\Psd) \sotimes_E A_\sigma(V)\Big)\]
as before, except that we take the truncated Chevalley-Eilenberg complex here. We then set 
\[\tau\cS^\CE_{q,j}(V) = \bigoplus_{\sigma \in \cR(\BT_q)} \tau\cS^\CE_{\sigma,j}(V)\]
and $\tau\cS^\CE_n(V) = \bigoplus_{q+j =n} \tau\cS^\CE_{q,j}(V)$, where summation is here only over those pairs $(q,j)$ for which $0 \le q \le \ell$ and $0 \le j \le \bd$. The existence of the differentials $d^{(k)}_{q,j}$ is shown exactly as in the proof of \cref{thm-mixed-res}. The only point where the argument differs slightly from the previous one is that $T_{\sigma,0,j} = \tau_{\le \bd}\CE_j(\bbHso,\Psd) \sotimes_E A_\sigma(V)$ is for $j = \bd$ not of the form $(\mbox{free } \DHPsd\mbox{-module}) \sotimes_E A_\sigma(V)$, but is rather of the form 
\[(\mbox{direct summand of f.g. free } \DHPsd\mbox{-module}) \sotimes_E A_\sigma(V) \,,\] 
but this is still a projective $\DHPsd$-module by \cref{hopf-projective}. The existence of the maps $d^{(k)}_{q,j}$ is then established exactly as in step 4 of the proof of \cref{thm-mixed-res}. We summarize these findings as
\end{para}

\begin{theorem}\label{thm-mixed-res-truncated}  
Assume that for all facets $\sigma$ of $\BT$ the solid vector space $A_\sigma(V)$ is a Smith space. Then there are differentials of solid $G$-representations $\tau\Delta_n: \tau\cS^\CE_n(V) \ra \tau\cS^\CE_{n-1}(V)$ such that $(\tau\cS^\CE_\bullet(V), \tau\Delta_\bullet)$ is a complex and 
$$\mbox{for all } n \ge 0: \; h^n(\tau\cS^\CE_\bullet(V),\tau\Delta_\bullet) = h^n(\cS_\bullet(V),d^\cS) \;.$$ 
In particular, if $\cS_\bullet(V) \xrightarrow{\vep^\cS} V$ is a resolution of $V$, then so is $\tau\cS^\CE_\bullet(V) \xrightarrow{\vep^\cS \, \circ \, \vep^G_0} V$.
\end{theorem}

\begin{remark}
Instead of working with mixed distribution algebras of the type $\DHPsd$, where $\bbHso$ is a strictly ind-affinoid group, one could also also work with strictly pro-affinoid groups $\bbHsd = \bbH_\sigma[\rho]^\dagger$ and their distribution algebras $D(\bbHsd,\Psd)$, provided $r< p^{-\frac{\rho}{\kappa (p-1)}}$. The latter condition is necessary for having a canonical homomorphisms $D(\bbHsd) \ra D_r(\Us)$, cf. \cref{distalgs-of-wide-open-and-overconv-gps-and-Dr}. In this setting we also have resolutions $\CE_j(\bbHsd,\Psd)$ of Chevalley-Eilenberg type of the trivial one-dimensional module $E$ as $D(\bbHsd,\Psd)$-module, cf. \cref{pro-affinoid-mixed-generalized}. These can then be used to construct representations of the form 
\[\cS^{\CE,\dagger}_{\sigma,j}(V) := \cind^G_\Psd\Big(\CE_j(\bbHsd,\Psd) \sotimes_E A_\sigma(V)\Big) \;.\]
For these one can prove the existence of a Wall complex as before, and one obtains a total complex $(\cS^{\CE,\dagger}_\bullet(V), \Delta^\dagger_\bullet)$ whose homology is equal to that of $(\cS_\bullet(V),d^\cS_\bullet)$. There is also a version $\tau\cS^{\CE,\dagger}_\bullet(V)$ of this complex when one starts with the truncated complex $\tau_{\le \bd}\CE_j(\bbHsd,\Psd)$.
\end{remark}

\subsection{Ext groups of admissible locally analytic representations}

We keep the setup of the previous section, cf. \cref{fixtures-mixing}. Recall that the $\DHPsd$-modules in the Chevalley-Eilenberg resolution of the trivial one-dimensional module are finitely generated free $\DHPsd$--modules of the form
$$\CE_j(\bbHso,\Psd) =  \DHPsd \sotimes_E M_{\sigma,j}$$
with an explicit finite-dimensional $E$-vector space $M_{\sigma,j}$, cf. paragraphs \ref{wall-complex} and \ref{ind-affinoid-mixed-generalized}. 

\begin{theorem}\label{Ext-for-admissible} Let $V,W$ be admissible locally $F$-analytic representations of $G = \bG(F)$. Assume conjecture \cite[3.41]{RR}. Furthermore, assume that the augmented complex $\cS_\bullet(V) = \cS^A_\bullet(V)$ is a resolution of $V$. 

\begin{enumerate}
\item $\uExt^n_G(V,W)$ is isomorphic to the $n^{\rm th}$ cohomology group of a complex $(\cE^\bullet(V,W), \partial^\bullet)$, where 
$$\cE^n(V,W) = \bigoplus_{\scalebox{.7}{$\begin{array}{c} 0 \le q \le \ell,\, 0 \le j \\
q+j = n \end{array}$}} \bigoplus_{\sigma \in \cR(\BT_q)} \uHom_E\Big(M_{\sigma,j} \sotimes_E V^{D_r(U_\sigma)},W^{\bbHso\anv}\Big) \;.$$
\item The cohomology of this complex vanishes in degrees $n > \ell + \bd = \dim(\bG) + \rk_F(\bG)$. 
\end{enumerate}\end{theorem}

\begin{proof} (1) Because we assume that $\cS_\bullet(V)$ is a resolution of $V$, the complex $\cS^\CE_\bullet(V)$ is a resolution of $V$ too, by \cref{thm-mixed-res}. It is here that we use that $V$ is admissible, so as to satisfy the assumption made in \cref{thm-mixed-res}. The Ext group $\uExt^n_G(V,W)$ is thus the $n^{\rm th}$ cohomology of the hypercomplex $R\uHom_G(\cS^\CE_\bullet(V),W)$. We will show that for each $n \ge 0$ the complex $R\uHom_G(\cS^\CE_n(V),W)$ is concentrated in degree zero. While showing this, we will derive a formula for this vector space and thereby prove the statement. 
\vskip8pt
\noindent In the following, the summation conditions on the pair $(q,j)$ are as in the statement of the theorem. We have 
$$\begin{array}{cl} & R\uHom_G(\cS^\CE_n(V),W) \\
&\\ 
= & \bigoplus_{q+j=n} R\uHom_G(\cS^\CE_{q,j}(V),W)\\
&\\ 
= & \bigoplus_{q+j=n} \bigoplus_{\sigma \in \cR(\BT_q)} R\uHom_G\Big(\cind^G_\Psd\Big(\CE_j(\bbHso,\Psd) \sotimes_W A_\sigma(V)\Big),W\Big)\\
&\\ 
= & \bigoplus_{q+j=n} \bigoplus_{\sigma \in \cR(\BT_q)} R\uHom_\Psd\Big(\CE_j(\bbHso,\Psd) \sotimes_E A_\sigma(V),W\Big)\\
&\\ 
= & \bigoplus_{q+j=n} \bigoplus_{\sigma \in \cR(\BT_q)} R\uHom_\Psd\Big(\Big(\DHPsd \sotimes_E M_{\sigma,j}\Big) \sotimes_E A_\sigma(V),W\Big)\\
\end{array}$$
where the third equality holds by Frobenius reciprocity. 
Now we use that the `diagonal' module structure on $\Big(\DHPsd \sotimes_E M_{\sigma,j}\Big) \sotimes_E A_\sigma(V)$ is the same as that on 
\[\Big(\DHPsd \sotimes_E M_{\sigma,j}\Big) \sotimes_E A_\sigma(V)_\triv = \DHPsd \sotimes_E \Big(M_{\sigma,j} \sotimes_E A_\sigma(V)_\triv\Big) \,,\] 
cf. \cref{hopf-tensor-trivial-free}, and we obtain that the previous solid vector space is 
\[\bigoplus_{q+j=n} \bigoplus_{\sigma \in \cR(\BT_q)} R\uHom_\Psd\Big(\DHPsd \sotimes_E \Big(M_{\sigma,j} \sotimes_E A_\sigma(V)_\triv\Big),W\Big) \;.\]
Furthermore, as module over $E_\solid[\Psd]$ we have $\DHPsd = \cind^\Psd_\Hso D(\bbHso)$. Applying Frobenius reciprocity again we find that the latter solid vector space is equal to
$$\bigoplus_{q+j=n} \bigoplus_{\sigma \in \cR(\BT_q)} R\uHom_\Hso \Big(D(\bbHso) \sotimes_E \Big(M_{\sigma,j} \sotimes_E A_\sigma(V)_\triv\Big),W\Big)$$
Using \cite[4.36]{RR} we obtain that the previous solid vector space is 
\[\bigoplus_{q+j=n} \bigoplus_{\sigma \in \cR(\BT_q)} R\uHom_E \Big(M_{\sigma,j} \sotimes_E A_\sigma(V),W^{R\bbHso\anv}\Big)\,,\]
where we have dropped the subscript `triv' here, as $A_\sigma(V)$ is here only considered as a solid vector space. Using that $W$ is admissible and our assumption that \cite[3.41]{RR} holds, we can apply \cref{higher-derived-an-vectors} and use that the higher analytic vectors of $W$ vanish, which finally shows that the latter solid vector space is
\[\begin{array}{cl}  & \bigoplus_{q+j=n} \bigoplus_{\sigma \in \cR(\BT_q)} R\uHom_E \Big(M_{\sigma,j} \sotimes_E A_\sigma(V),W^{\bbHso\anv}\Big)\\
&\\ 
= & \bigoplus_{q+j=n} \bigoplus_{\sigma \in \cR(\BT_q)} \uHom_E \Big(M_{\sigma,j} \sotimes_E A_\sigma(V), W^{\bbHso\anv}\Big) \,,\\
\end{array}\]
where the last equality holds because $M_{\sigma,j} \sotimes_E A_\sigma(V)$ is a Smith space, cf. \cref{Dr-analytic-Smith}. And since $V$ is admissible we find that $A_\sigma(V) = \uHom_\Us(D_r(\Us),V) = V^{D_r(\Us)}$, by 
\cref{Dr-and-HomDr-for-compact-ind-limits}. 

\vskip8pt

(2) Instead of working with the complex $\cS^\CE_\bullet(V)$ we now consider the complex $\tau\cS^\CE_\bullet(V)$, cf. \cref{using-CE-resolution-truncated}. Using \cref{thm-mixed-res-truncated}, we see that statement (2) follows, once we show that $\uExt^k_G(\tau\cS^\CE_n(V),W)$ vanishes for all $n \ge 0$ and all $k>0$. We have $\tau\cS^\CE_n(V) = \bigoplus_{q+j=n} \tau\cS^\CE_{q,j}(V)$, where $\tau\cS^\CE_{q,j}(V) =  \bigoplus_{\sigma \in \cR(\BT_q)} \tau\cS^\CE_{\sigma,j}$ with
\[\tau\cS^\CE_{\sigma,j}(V) = \cind^G_\Psd\Big(\tau_{\le \bd} \CE_j(\bbHso,\Psd) \sotimes_W A_\sigma(V)\Big) \,.\]
This representation vanishes for $j>\bd$ and is equal to $\cS^\CE_{\sigma,j}(V)$ for $j<\bd$. When $j = \bd$ the module $\tau_{\le \bd} \CE_\bd(\bbHso,\Psd)$ is a direct summand of $\CE_{\bd+1}(\bbHso,\Psd)$, and we can write
\[\CE_{\bd+1}(\bbHso,\Psd) = \tau_{\le \bd} \CE_\bd(\bbHso,\Psd) \oplus \CE'_\bd(\bbHso,\Psd) \,,\]
where $\CE'_\bd(\bbHso,\Psd)$ is another finitely generated $\DHPsd$-module. It follows that 
\[\cS^\CE_{\sigma,\bd+1}(V) = \tau\cS^\CE_{\sigma,\bd}(V) \oplus \cind^G_\Psd\Big(\CE'_\bd(\bbHso,\Psd) \sotimes_E A_\sigma(V)\Big) \,,\] 
and therefore $\uExt^k_G(\tau\cS^\CE_{\sigma,\bd}(V),W)$ is a direct summand of $\uExt^k_G(\tau\cS^\CE_{\sigma,\bd+1}(V),W)$ which vanishes for all $k>0$. 
\end{proof}

\begin{corollary}\label{cor-Ext-for-admissible} Let the assumptions be as in \ref{Ext-for-admissible}. 
\begin{enumerate}
\item For all $n > \dim(\bG) + \rk_F(\bG)$ one has $\uExt^n_G(V,W) = 0$.
\item If for all facets $\sigma$ of $\BT$\footnote{Equivalently, for every $q = 0, \ldots, \ell$ and every representative for $G\bksl BT_q$.} one has $W^{\bbHso\anv} = 0$, then $\uExt^n_G(V,W) = 0$ for all $n \ge 0$.
\item If for all facets $\sigma$ of $\BT$ one has $W^{\bbHso\anv} = 0$, then for all $n \ge 0$ the cohomology group $\underline{H}^n(G,W) = \uExt^n_G(E,W)$  vanishes.
\end{enumerate} 
\end{corollary}

\begin{proof} Statements (1) and (2) follow directly from \cref{Ext-for-admissible}. Statement (3) is true because the trivial 1-dimensional representation is smooth (and of course admissible), and the Schneider-Stuhler complex $\cS_\bullet(E)$ is a resolution of $E$.
\end{proof}